\documentclass[11pt]{article}
\usepackage{defs/packages}
\usepackage{defs/macros}
\usepackage{defs/diagrammatics}
\usepackage{defs/environments}
\usepackage{defs/custom-biblatex-style}

\usepackage[a4paper, total={6in, 9in}]{geometry}

\counterwithin{figure}{section}
\usepackage{perpage}
\MakePerPage{footnote}

\newcommand{\normafoam}[2]{#1}

\title{A local \texorpdfstring{$\mathfrak{gl}_{1|1}$}{gl11}-action on odd Khovanov homology}

\author{Mark Ebert \an Léo Schelstraete}
\date{\vspace{-5ex}}

\begin{document}

\maketitle

\begin{abstract}
  We show that odd Khovanov homology carries an action of the super Lie algebra $\mathfrak{gl}_{1|1}$, given extra choice of markings on the link.
  Moreover, we show that this action arises from an action on super $\mathfrak{gl}_{2}$-foams, in the extended-TQFT framework developed by the second author and Vaz; in particular, it extends to tangles.
  Finally, we relate the action to torsion $\mathbb{Z}/n\mathbb{Z}$ in pretzel links $P(n,n,-n)$. In particular, this shows that all torsion can appear in odd Khovanov homology.
\end{abstract}

\tableofcontents

\newpage

\section{Introduction}
\label{sec:introduction}

\subsection{Overview}

Quantum link homologies are homology theories for links in $S^3$ that arise as categorifications of polynomial link invariants associated with quantum groups.
Following Khovanov's construction of a categorification of the Jones polynomial \cite{Khovanov_CategorificationJonesPolynomial_2000}, several related homology theories have been developed.
These constructions are closely connected with higher representation theory and have led to many fruitful interactions between these fields.

\medbreak

Actions on these homologies, or on the categories underlying them, have been studied by various authors in different contexts and with different motivations.

Gorsky, Oblomkov, and Rasmussen \cite{GOR_StableKhovanovHomology_2013} conjectured that certain colored link homologies have graded dimensions given by the characters of representations of affine Lie algebras.
An $\slt$-action on triply-graded homology was constructed by Gorsky, Hogancamp and Mellit \cite{GHM_TautologicalClassesSymmetry_2024} used to show certain symmetries of triply-graded homology, giving a new proof of a conjecture of Dunfield, Gukov, and Rasmussen \cite{DGR_SuperpolynomialKnotHomologies_2006}; this action is further studied in \cite{CG_StructuresHOMFLYPTHomology_2024}.

Actions of Steenrod algebras have been constructed on even and odd Khovanov homology \cite{LS_SteenrodSquareKhovanov_2014,Schutz_TwoSecondSteenrod_2022}, induced from (or at least motivated by) the existence of (odd) Khovanov stable homotopy type \cite{LS_KhovanovStableHomotopy_2014,HKK_FieldTheoriesStable_2016,SSS_OddKhovanovHomotopy_2020}. See \cite{Rajapakse_SteenrodSquaresEven_2025} for recent developments.

For annular theories, Grigsby, Licata and Wehrli \cite{GLW_AnnularKhovanovHomology_2018} constructed an action the $\slt$ current algebra on annular Khovanov homology, while Grigsby and Wehrli constructed an action of $\gloo$ on odd annular Khovanov homology \cite{GW_Action$mathfrakgl1|1$Odd_2020}.

In \cite{KR_PositiveHalfWitt_2016}, Khovanov and Rozansky constructed an action of the positive half of the Witt algebra $\fW^+$ on triply-graded homology.
Inspired by this action, Qi, Robert, Wagner and Sussan \cite{QRS+_Symmetries$mathfrakgl_N$foams_2024a} constructed an action of $\fW^\infty_{-1}=\fW^+\cup\langle L_{-1}\rangle$ on equivariant $\fgl_N$-foams \cite{MSV_$mathfrakslN$linkHomology$N_2009,Khovanov_Sl3LinkHomology_2004,RW_ClosedFormulaEvaluation_2020}, where $L_{-1}$ is the degree $-2$ operator in the Witt algebra $\fW$; Guérin and Roz \cite{GR_ActionWittAlgebra_2025} later extended this action to equivariant Khovanov--Rozansky homology \cite{KR_MatrixFactorizationsLink_2008}, building on \cite{QRS+_SymmetriesEquivariantKhovanovRozansky_2023}.
Over a field of characteristic $p$, one can restrict to non-equivariant parameters, and the degree $2$ operator $L_1\in\fW^+$ recovers the $p$-DG structure used in
\cite{%
  % Khovanov_HopfologicalAlgebraCategorification_2016,
  QRS+_CategorificationColoredJones_2021,%
  QS_PdifferentialGradedLink_2022%
} to categorify the (resp.\ colored) Jones polynomial at root of unity.
On the other hand, the operator $L_{-1}$ recovers Wang's extension of Shumakovitch operation \cite{Wang_$mathfrakslN$LinkHomology_2024}.
These work have lead to certain topological applications and structural properties; see \cite{QRS+_SymmetriesEquivariantKhovanovRozansky_2023,QRS+_RemarksInfinitesimalSymmetries_2024,Roz_$mathfraksl_2$ActionLink_2023}.

In connection with some of the above work, Elias and Qi realised that various categories appearing in higher representation theory carried an $\slt$-action \cite{EQ_CategorifyingHeckeAlgebras_2020,EQ_Actions$mathfraksl_2$Algebras_2023}. In a related direction, Grlj and Lauda recently constructed an action of the positive Witt algebra on simply-laced categorified quantum groups \cite{GL_ActionWittAlgebra_2025}.
% \cite{BC_SteenrodStructuresCategorified_2018}

In this article, in analogy with this line of work, we describe a  $\gloo$-action on \emph{odd} Khovanov homology.

\medbreak

Odd Khovanov homology \cite{ORS_OddKhovanovHomology_2013} is a homological invariant of links. As (even) Khovanov homology, it categorifies the Jones polynomial. While the two theories are identical over $\bF_2$, they are distinct over $\bZ$, in the sense that one can find pair of knots distinguished by one theory but not the other \cite{Shumakovitch_PatternsOddKhovanov_2011}.
It was discovered in an attempt to lift to the integers the Ozsváth--Szabó spectral sequence from Khovanov homology to the Heegaard--Floer homology of the branched double cover \cite{OS_HeegaardFloerHomology_2005}.
While the existence of this spectral sequence remains conjectural, odd Khovanov homology is thought as more closely related to Heegaard--Floer theory than its even counterpart.
Various authors have explored odd Khovanov homology; see \cite{NP_OddKhovanovHomology_2020,Spyropoulos_HochschildCohomologyOdd_2025,Spyropoulos_JonesWenzlProjectorsOdd_2024,MW_FunctorialityOddGeneralized_2024} for recent structural results using this original construction.

Since its discovery, odd Khovanov homology has been expected to relate to various odd analogues in higher representation theory, and in particular to so-called ``supercategorification'' \cite{EL_OddCategorification$U_qmathfraksl_2$_2016}; see e.g.\ \cite{%
  EKL_OddNilHeckeAlgebra_2014,%
  BE_SuperKacMoody_2017,%
  % BE_MonoidalSupercategories_2017,
  BK_OddGrassmannianBimodules_2022,%
  % Ebert_NewPresentationOsp1|2polynomial_2022,
  % EQ_DifferentialGradedOdd_2016,
  EL_DGStructuresOdd_2020,%
  EL_OddCategorification$U_qmathfraksl_2$_2016,
  ELV_DerivedSuperequivalencesSpin_2022,
  ENW_RealSpringerFibers_2021,
  % KKO_SupercategorificationQuantumKacMoody_2013,
  % KKO_SupercategorificationQuantumKac_2014,
  LR_OddificationCohomologyType_2014,
  % NV_OddKhovanovsArc_2018,
}.
An explicit connection in that direction was given in \cite{SV_OddKhovanovHomology_2023}, where the second author and Vaz gave a foamy construction of odd Khovanov homology.
The main players are \emph{super $\glt$-foams}, gathering together as the super-2-category $\sfoam$; they lead to an invariant of tangles in the homotopy category of $\sfoam$.
A \emph{super-2-category} is a structure akin to a linear 2-category, but where 2-morphisms have parities and the interchange law only hold up to sign.
In the original construction of odd Khovanov homology, signs depend on whether a saddle is a split or a merge (a global data); in the foamy construction, parities only depend on wether a saddle is a zip or an unzip (a local data).
Despite these conceptual differences, the two constructions lead to the same invariant (when restricted to links).

Moreover, each construction comes in two flavours. The original construction is either ``type X'' or ``type Y'';
and the super-2-category $\sfoam$ admits a (essentially unique) variant, denoted $\sfoam'$ \cite{Schelstraete_RewritingModuloDiagrammatic_2025}.
Through the isomorphism between the two constructions, $\sfoam$ relates to type Y, while $\sfoam'$ relates to type X.
On the topological side, the existence of these variants comes from a sign choice ambiguity on so-called ``ladybug squares'', similar to the choice ambiguity appearing in Khovanov stable homotopy type \cite{LS_KhovanovStableHomotopy_2014}.
Despite this ambiguity, type X and type Y have been shown to be isomorphic \cite{Beier_IntegralLiftStarting_2012,Putyra_2categoryChronologicalCobordisms_2014}.

\medbreak

In work in progress, Migdail and Wehrli \cite{MW_ModuleStructureOdd_} (building on Migdail's PhD thesis \cite{Migdail_FunctorialityOddKhovanov_2025}) define an action of the first homology group of the branched double cover of the link, and study some of its topological consequences.
We learned about their work while preparing this manuscript; see \cref{rem:literature_comparison_Migdail_Wehrli} for details on how our work relates with theirs.

We now summarize our result:
\medbreak

\noindent \textsc{Extended abstract:}
\begin{enumerate}[(i)]
  \item There exists a $\gloo$-action on the super-2-category $\sfoam$ which gives rise to a $\gloo$-action on (the foamy construction of) odd Khovanov homology, well-defined for any tangle and choice of ``markings'' (see below).
  \item Markings behave differently in type X (i.e.\ $\sfoam'$) and type Y (i.e.\ $\sfoam$); while the type X and type Y odd Khovanov homologies are isomorphic, they are \emph{not} expected to be $\gloo$-equivariantly isomorphic.
  \item When restricting to links and comparing with the original construction of odd Khovanov homology, part of that action recovers Migdail and Wehrli's action of the first homology group of the branched double cover of the link.
  \item The pretzel link $P(n,n,-n)$ has torsion $\bZ/n\bZ$; this copy lies in the image of the $\gloo$-action. In particular, all torsions appear in odd Khovanov homology.
\end{enumerate}
Through (iii), item (ii) recovers a similar observation by Migdail and Wehrli in their work in progress \cite{MW_ModuleStructureOdd_}. Item (iv) in particular answers a question of Shumakovitch \cite{Shumakovitch_PatternsOddKhovanov_2011}, who showed that $P(n,n,-n)$ had torsion $\bZ/n\bZ$ for small $n$ and suggested this was a general pattern. Through (iii) again, this extends a remark of Migdail and Wehrli \cite{MW_ModuleStructureOdd_}, who showed that torsion in $P(3,3,-3)$ lies in the image of their action.

\newpage

\subsection{Results}

We now describe our results in more details.
Throughout we work over any ring in which $2$ is invertible; alternatively, one can ignore the condition that $2$ is invertible by restricting to $\sloo\subset\gloo$ (see \cref{rem:local_invariant_2_invertible}).
Recall that the super Lie algebra $\gloo$ has generators $\liee$, $\lief$, $\lieh_1$ and $\lieh_2$; see \cref{ex:defn_gloo}.

The $\gloo$-action depends on a choice of ``markings'' on the tangle.
Namely, a \emph{choice of markings} is a choice of diagram together with points on this diagram, each endowed with a triple of scalars $(\alpha,\beta_1,\beta_2)$ with $\alpha=\beta_1+\beta_2$. These scalars $\alpha$, $\beta_1$ and $\beta_2$ correspond to twists of the action of $\lief$, $\lieh_1$ and $\lieh_2$, respectively.

The super-2-category of super $\glt$-foams is reviewed in \cref{defn:foam_and_sfoam} and \cref{defn:review_foam}, which we write as $\sfoam$ (and $\sfoam'$ its variant) in this introduction for simplicity; see also \cref{defn:twisted_foams} and \cref{defn:colimit_of_foam} for the relevant versions with markings.

\begin{bigtheorem}[\cref{thm:topological_invariance} and \cref{lem:dot_slide_lemma}]
  \label{thm:main_local_action}
  There exists a $\gloo$-action on $\sfoam$, given on generators in \cref{tab:definition_action_foam_intro}, which extends to a $\gloo$-action on (the foamy construction of) Khovanov homology for any tangle and choice of markings. Moreover, the action is invariant under any move not involving the markings, under markings sliding along strands, and under markings sliding accross crossings as follows (here $\omega=(\alpha,\beta_1,\beta_2)$):
  \begin{gather*}
    \tikzpic{
      \webncr
      \node[green_mark] (B) at (.3,.15) {};
      \node[below={-1pt} of B] {\scriptsize $\omega$};
    }[scale=.7][(0,.5*.7)]
    \simeq^{\gloo}
    \tikzpic{
      \webncr
      \node[green_mark] (B) at (1-.3,1-.15) {};
      \node[above={-1pt} of B] {\scriptsize $\omega$};
    }[scale=.7][(0,.5*.7)]
    \quad\an\quad
    \tikzpic{
      \webpcr
      \node[green_mark] (B) at (.3,.15) {};
      \node[below={-1pt} of B] {\scriptsize $\omega$};
    }[scale=.7][(0,.5*.7)]
    \simeq^{\gloo}
    \tikzpic{
      \webpcr
      \node[green_mark] (B) at (1-.3,1-.15) {};
      \node[above={-1pt} of B] {\scriptsize $-\omega$};
      \node[green_mark] (C) at (1-.3,.15) {};
      \node[below={-1pt} of C] {\scriptsize $2\omega$};
    }[scale=.7][(0,.5*.7)]
    \;.
  \end{gather*}
  Here $\simeq^{\gloo}$ denotes isomorphism in the relative homotopy category (\cref{defn:formal_relative_homotopy_category}).
  Considering $\sfoam'$ instead exchanges the role of the overcross and the undercross in the above statement.
\end{bigtheorem}

\begin{table}
  \def\spc{2ex}
  \def\wdth{1.5cm}
  \centering
  \begin{tabular}{@{}l@{\hskip 5ex}*{4}{r@{\hskip 5ex}}r@{}}
    &
    $\vcenter{\hbox{\includegraphics[width=\wdth]{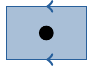}}}$
    &
    $\vcenter{\hbox{\includegraphics[width=\wdth]{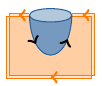}}}$
    &
    $\vcenter{\hbox{\includegraphics[width=\wdth]{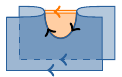}}}$
    &
    $\vcenter{\hbox{\includegraphics[width=\wdth]{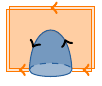}}}$
    &
    $\vcenter{\hbox{\includegraphics[width=\wdth]{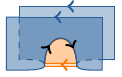}}}$
    \\*[2ex]
    \midrule
    %%%%%%%%%%%%%%%%%%%%%%%%%%%%%%
    $\lief$
    &
    $0$
    &
    $\vcenter{\hbox{\includegraphics[width=\wdth]{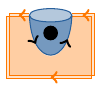}}}$
    &
    $\vcenter{\hbox{\includegraphics[width=\wdth]{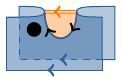}}}$
    &
    $-\vcenter{\hbox{\includegraphics[width=\wdth]{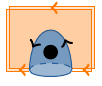}}}$
    &
    $\vcenter{\hbox{\includegraphics[width=\wdth]{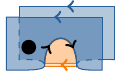}}}$
    \\*[\spc]
    %%%%%%%%%%%%%%%%%%%%%%%%%%%%%%
    $\liee$
    &
    $\vcenter{\hbox{\includegraphics[width=\wdth]{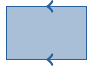}}}$
    &
    $0$
    &
    $0$
    &
    $0$
    &
    $0$
    \\*[\spc]
    %%%%%%%%%%%%%%%%%%%%%%%%%%%%%%
    $\lieh_1$
    &
    $-\vcenter{\hbox{\includegraphics[width=\wdth]{foam/dot.pdf}}}$
    &
    $\vcenter{\hbox{\includegraphics[width=\wdth]{foam/cup.pdf}}}$
    &
    $0$
    &
    $0$
    &
    $-\vcenter{\hbox{\includegraphics[width=\wdth]{foam/unzip.pdf}}}$
    \\*[\spc]
    %%%%%%%%%%%%%%%%%%%%%%%%%%%%%%
    $\lieh_2$
    &
    $\vcenter{\hbox{\includegraphics[width=\wdth]{foam/dot.pdf}}}$
    &
    $0$
    &
    $\vcenter{\hbox{\includegraphics[width=\wdth]{foam/zip.pdf}}}$
    &
    $-\vcenter{\hbox{\includegraphics[width=\wdth]{foam/cap.pdf}}}$
    &
    $0$
    \\*[\spc]
    % \cmidrule(r){2-6}
    %%%%%%%%%%%%%%%%%%%%%%%%%%%%%%
    % $\lieh\coloneqq\lieh_1+\lieh_2$
    % &
    % $0$
    % &
    % \includegraphics[width=\wdth]{foam/cup.pdf}
    % &
    % \includegraphics[width=\wdth]{foam/zip.pdf}
    % &
    % $-$\includegraphics[width=\wdth]{foam/cap.pdf}
    % &
    % $-$\includegraphics[width=\wdth]{foam/unzip.pdf}
  \end{tabular}
  \caption{Definition of the action of $\gloo$ by derivation on the generators of $\sfoam$.
  The source is given on the top row, and the target on the associated row; for instance, we have $\lief\cdot\vcenter{\hbox{\includegraphics[width=1cm]{foam/dot.pdf}}}=0$.}
  \label{tab:definition_action_foam_intro}
\end{table}

Note that while markings can ``freely'' overcross, the rule for undercrossing is more intricate; in fact, one can check that it cannot both freely overcross \emph{and} undercross (\cref{rem:marking_cannot_undercross}).
It follows that (see \cref{thm:main_global_action}):

\begin{corollary}
  The homology theories for marked tangles associated to $\sfoam$ and $\sfoam'$ are isomorphic, but in general 
   \emph{not} $\gloo$-equivariantly isomorphic.
\end{corollary}

As a natural odd analogue to the work of Elias and Qi \cite{EQ_Actions$mathfraksl_2$Algebras_2023}, and having in mind the work of Grlj and Lauda \cite{GL_ActionWittAlgebra_2025}, one wonders:

\begin{question}
  Are there other actions of super Lie algebras appearing in supercategorification, for instance on super Kac--Moody 2-categories \cite{BE_SuperKacMoody_2017}?
\end{question}

There is a unifying approach between even and odd Khovanov homology, replacing signs by scalars and super structures by graded structures.
In particular, there exists a \emph{graded-2-category $\gfoam$ of graded $\glt$-foams}, which specializes both to $\glt$-foams and super $\glt$-foams.
Working in this framework allows an explicit comparison of the two theories.

The action of $\liee$ does not work in the even setting, for a simple reason. For grading reason, it must act on dots as $\liee(\textdot)=\lambda\id$ for some scalar $\lambda$, but by the Leibniz rule, $\liee(\textdot^2)=\lambda\;\textdot+\lambda\;\textdot=2\lambda\;\textdot$, which contradicts $\textdot^2=0$ (at least if $2\lambda$ is invertible).
In the super context, the super Leibniz rule replaces ``$+$'' by ``$-$'', and hence there is no contradiction.
This is parallel to the fact that the $\slt$-action in (non-equivariant) $\fgl_p$-Khovanov--homology \cite{QRS+_SymmetriesEquivariantKhovanovRozansky_2023} is well-defined only over a field of characteristic $p$.

Nonetheless, one can define the action excluding $\liee$, and this can be unified at the level of graded $\glt$-foams.
For that purpose, we define $\grgl^\geq_2$ as a certain \emph{graded Lie algebra} (a structure that interpolates between Lie algebras and super Lie algebras; this is \emph{not} just a Lie algebra with a grading) interpolating between $\fgl_2$ and $\gloo$.
The homology interpolating even and odd Khovanov homology is known as \emph{covering} (or \emph{generalized}) \emph{Khovanov homology} \cite{Putyra_2categoryChronologicalCobordisms_2014}.

\begin{proposition}
  There exists a $\grgl^\geq_2$-action on $\gfoam$, which extends to a $\grgl^\geq_2$-action on covering $\glt$-Khovanov homology for any tangle and choice of markings.
  Moreover, the action is invariant under any move not involving the markings and under markings sliding along strands, away from crossings.
\end{proposition}

Note that in the graded case, markings do not seem to verify any particular crossing slide relation\footnote{Although see the relation in the proof of \cref{lem:csq_dot_slide_lemma}, which holds in general.}; the result that markings can slide over crossing is specific to the odd case.

\medbreak

Next, we compare with the original construction of odd Khovanov homology.
Our construction provides a certain ``$\gloo$-equivariant homotopy equivalence of complexes $\glOKh(T)$'' associated to a tangle with markings; to compare with the original construction, we need to apply a homology functor, given by the composition of the standard homology functor and a representable functor.
We denote $\slOKh^Y(L)$ the type Y (original construction of) odd Khovanov homology.

\begin{bigtheorem}[\cref{thm:equivalence_local_global}]
  \label{thm:main_global_action}
  Let $L$ be an oriented link and $D$ a diagram of $L$.
  There exists a $\gloo$-action on (the original construction of) odd Khovanov homology, for any oriented link and choice of markings.
  Moreover, there is a $\gloo$-equivariant isomorphism
  \begin{gather*}
    H_\bullet\Hom(\emptyset,\glOKh(D))\cong^{\gloo}\slOKh^Y(D).
  \end{gather*}
  Similarly, we have a $\gloo$-equivariant isomorphism considering $\sfoam'$ and type X instead.
\end{bigtheorem}

This allows us to relate with other constructions appearing in literature; see \cref{rem:literature_comparison_Shumakovitch} for comparison with Shumakovitch's operation $\nu$ \cite{Shumakovitch_TorsionKhovanovHomology_2014}, \cref{rem:literature_comparison_Migdail_Wehrli} for comparison with Manion's work \cite{Manion_SignAssignmentTotally_2014}\footnote{We thank Stephan Wehrli for pointing out that reference to us.} and Migdail and Wehrli's work in progress \cite{MW_ModuleStructureOdd_,Migdail_FunctorialityOddKhovanov_2025}, and \cref{rem:literature_comparison_Grigsby_Wehrli} for comparison with Grigsby and Wehrli's $\gloo$-action on odd annular Khovanov homology \cite{GW_Action$mathfrakgl1|1$Odd_2020}.

\medbreak

As noticed by Shumakovitch \cite{Shumakovitch_PatternsOddKhovanov_2011}, even and odd Khovanov homology typically have very different torsions.
As an example of that heuristics, Shumakovitch noticed that for certain pretzel links, reduced even and odd Khovanov homologies have the same torsion-free part, with only odd Khovanov homology having a non-trivial torsion part.
In particular, he computed that $P(n,n,-n)$ had $\bZ/n\bZ$ torsion in odd Khovanov homology for small $n\in\bN$, and asked whether this was a general pattern.

We verify this expectation, and relate it to our $\gloo$-action:

\begin{bigtheorem}
  \label{thm:main_pretzel}
  \def\webscl{.4}
  Let $n\in\bN$.
  The odd Khovanov homology of the pretzel link
  \begin{gather*}
    P(n,n,-n)\coloneqq
    \tikzpic{
      \webpcr\node at (2,.5) {$\ldots$};\webpcr[3][0]
      \webpcr[0][2]\node at (2,2.5) {$\ldots$};\webpcr[3][2]
      \webncr[0][4]\node at (2,4.5) {$\ldots$};\webncr[3][4]
      \draw[web1] (0,0) to[out=180,in=180] (0,5);
      \draw[web1] (0,1) to[out=180,in=180] (0,2);
      \draw[web1] (0,3) to[out=180,in=180] (0,4);
      \draw[web1] (4,0) to[out=0,in=0] (4,5);
      \draw[web1] (4,1) to[out=0,in=0] (4,2);
      \draw[web1] (4,3) to[out=0,in=0] (4,4);
      \node[green_mark] at (0,5) {};
      \node[above left=-1pt] at (0,5) {\scriptsize $(1,\frac{1}{2},\frac{1}{2})$};
      \node[green_mark] at (4,5) {};
      \node[above right=-1pt] at (4,5) {\scriptsize $(-1,-\frac{1}{2},-\frac{1}{2})$};
      \draw [decorate,decoration={brace,amplitude=5pt,mirror,raise=1ex}]
      (0,0) -- (4,0) node[midway,yshift=-1.5em]{$n$};
    }[scale=\webscl][(0,2*\webscl)]
  \end{gather*}
  has torsion $\bZ/n\bZ$. Moreover, this copy of $\bZ/n\bZ$ lies in the image of the action of $\lief\in\gloo$, for the given choice of markings (for both type X or type Y).
\end{bigtheorem}

In particular, the $\gloo$-action is non-trivial on $P(n,n,-n)$. As mentioned above, Migdail and Wehrli \cite{Migdail_FunctorialityOddKhovanov_2025} have shown an analoguous statement using their action (see \cref{rem:literature_comparison_Migdail_Wehrli}), for the pretzel knots $P(3,3,-3)$ and $P(3,4,-3)$.

It follows that:

\begin{corollary}
  All torsions appear in odd Khovanov homology.
\end{corollary}

To the authors' knowledge, this result has not appeared in the literature.
In contrast, and to the authors' knowledge again, it is not known whether all torsions appear in even Khovanov homology, in spite of active research on the question; see e.g.\ \cite{Shumakovitch_TorsionKhovanovHomology_2014,PS_TorsionKhovanovHomology_2014,MPS+_SearchTorsionKhovanov_2018,MS_ArbitrarilyLargeTorsion_2021}.

\begin{question}
  How much of the torsion in odd Khovanov homology can be explained by the $\gloo$-action?
\end{question}

\subsection{Organization}

\Cref{sec:action_on_foams} describes the action on $\sfoam$, \cref{sec:action_on_homology} describes the action on (the foamy definition of) odd Khovanov homology, \cref{sec:global_action_odd_khovanov} compares with the original construction when restricting to links, and \cref{sec:computation} does the torsion computation for pretzel links.

\subsection{Acknowledgments}

We thank Alexis Guérin, You Qi, Louis-Hadrien Robert, Felix Roz, Josh Sussan for discussing their previous work, and we thank Stephan Wehrli for discussing his work in progress with Jacob Migdail.
L.S.\ is supported by the Max Planck Institute for Mathematics (Bonn, Germany). M.E. is supported by the Simons Collaboration Award 994328 “New Structures in Low-Dimensional Topology.”

\section{Actions on super and graded \texorpdfstring{$\glt$}{gl2}-foams}
\label{sec:action_on_foams}

\subsection{Graded structures}

% \begin{verbatim}
%   NOTATIONS
%   - generic ring: \ring
%   - generic group: \grp
%   - elements of \grp: g, h
%   - generic bilinear map: \bil
%   - scalars: \lambda, (\mu)

%   - generic category: A
%   - objects in category: u, v, w
%   - morphisms in category: \alpha, \beta, \gamma
  
%   - generic 2-category: \cA
%   - objects in 2-categories: i, j, k
%   - 1-morphisms in 2-categories: u, v, w
%   - 2-morphisms in 2-categories: \alpha, \beta, \gamma

%   - vertical composition: \circ
%   - horizontal composition: \otimes
% \end{verbatim}

% \newpage

In this subsection, we describe the super, and more generally graded, analogue of various structures familiar in the commutative setting.
After defining graded associative algebras and graded Lie algebras, we review graded-2-categories from \cite{SV_OddKhovanovHomology_2023}, and define a $\fg$-2-category as a graded-2-category endowed with an action of a graded Lie algebra $\fg$; this specializes to the notion of $\slt$-categories from \cite{EQ_Actions$mathfraksl_2$Algebras_2023}.
Finally, we describe the graded analogue of twists \cite{KR_PositiveHalfWitt_2016,QRS+_Symmetries$mathfrakgl_N$foams_2024a}.

We fix throughout a commutative ring $\ring$, an abelian group $\grp$ and a pairing $\bil\colon\grp\times\grp\to\ring^\times$, that is, a bilinear map.
We further assume that $\bil$ is \emph{symmetric}, in the sense that
\[\bil(g,h)\bil(h,g)=1\quad\forall g,h\in\grp.\]
We write $\deg(v)$ the degree of an element $v$ in a $G$-graded object, although we often abuse notation and write $\mu(\deg v,\deg w)$ simply as $\mu(v,w)$.
Given two $G$-graded $\ring$\nbd-modules $M$ and $N$, we write $\Hom(M,N)$ the $\ring$\nbd-module of degree-preserving $\ring$-linear maps between $M$ and $N$, and $\uHom(M,N)$ the $\grp$-graded $\ring$-module of all $\ring$-linear maps, not necessarily degree-preserving.
We write $\End(M)\coloneqq\Hom(M,M)$ and $\uEnd(M)\coloneqq\uHom(M,M)$.

We denote $\Mod_{\grp,\bil}$ the closed symmetric monoidal category of $\grp$-graded $\ring$-modules and deg\-ree-preserving linear maps.
Its monoidal structure is the usual one on $G$-graded $\ring$-modules; note that it does not depend on $\bil$, and we write $\Mod_{\grp}=\Mod_{\grp,\bil}$ when considered only as a monoidal category.
The symmetric structure is given by $(x,y)\mapsto\bil(x,y)(y,x)$ and the inner Hom is given by $\uHom$.
% \[\underline{\Hom}_{\Mod_{\grp,\bil}}(M,N)=\uHom(M,N).\]

We denote $\uMod_{\grp,\bil}$ the symmetric monoidal category whose objects are $\grp$-graded $\ring$-modules (the same as $\Mod_{\grp,\bil}$) and with $\uHom(M,N)$ as $G$-graded homspace between $M$ and $N$.
In other words, the category $\uMod_{\grp,\bil}$ is the $\Mod_{\grp,\bil}$-enriched category determined by the closed monoidal structure of $\Mod_{\grp,\bil}$; as an $\Mod_{\grp,\bil}$-enriched category, its underlying category is $\Mod_{\grp,\bil}$ (see e.g.\ \cite[section~3.4]{Riehl_CategoricalHomotopyTheory_2014}).

We sometimes simplify notation and write $\Mod=\Mod_{\grp,\bil}$ and $\uMod=\uMod_{\grp,\bil}$.

\subsubsection{Graded associative algebras}
\label{subsubsec:graded_associative_algebras}

\begin{definition}
  A \emph{$\grp$-graded (associative) algebra} is a unital associative algebra object in the monoidal category $\Mod_{\grp}$.
\end{definition}

That is, a $\grp$-graded algebra is a unital and associative algebra $(A,\cdot_A,1_A)$, such that $A$ is $\grp$-graded as a $\ring$-module, the multiplication is degree-preserving and the unit has trivial degree.
Similarly, a \emph{morphism of $\grp$-graded algebras} is a morphism of unital associative algebra objects in the monoidal category $\Mod_{\grp}$; that is, a degree-preserving linear map preserving the unit and the product.

\medbreak

Let $M$ be a $\grp$-graded $\ring$-module. The algebra $\uEnd(M)$ of linear maps on $M$ has a canonical structure of $\grp$-graded algebra. If $A$ is a $\grp$-graded algebra, an \emph{action of $A$ on $M$} is morphism of $\grp$-graded algebra $A\to\uEnd(M)$. We say that $M$ is an \emph{$A$-module}, and a \emph{morphism of $A$-modules} is a degree-preserving linear map intertwining the actions.

\begin{definition}
  \label{defn:graded-commutative}
  A \emph{$(\grp,\bil)$-graded commutative algebra} is a commutative unital associative algebra object in the symmetric monoidal category $\Mod_{\grp,\bil}$.
\end{definition}

That is, a $(\grp,\bil)$-graded commutative algebra is a $\grp$-graded algebra where for every homogeneous $x$ and $y$, we have $xy=\bil(x,y)yx$.
Note that a $\grp$-graded algebra is always an algebra, while a $(\grp,\bil)$-graded commutative algebra needs not be a commutative algebra.

\subsubsection{Graded Lie algebras}
\label{subsubsec:graded_lie_algebras}

\begin{definition}
  A \emph{$(\grp,\bil)$-graded Lie algebra} is a Lie algebra object in the symmetric monoidal category $\Mod_{\grp,\bil}$.
\end{definition}

That is, a $(\grp,\bil)$-graded Lie algebra is a $\grp$-graded $\ring$-module $\fg$ equipped with a degree-preserving map $[-,-]\colon\fg\otimes\fg\to\fg$ such that
\begin{gather*}
  [x,y]+\bil(x,y)[y,x]=0\\
  % \bil(y,x)\bil(z,x)\bil(z,y)[x,[y,z]]+\bil(z,y)[y,[z,x]]+\bil(y,x)[z,[x,y]]=0\\
  % \Leftrightarrow
  [x,[y,z]]+\bil(x,y+z)[y,[z,x]]+\bil(x+y,z)[z,[x,y]]=0
\end{gather*}
Similarly, a \emph{morphism of $\grp$-graded Lie algebras} is a morphism of Lie algebra objects in the monoidal category $\Mod_{\grp}$; that is, a degree-preserving linear map preserving the bracket.

\medbreak

Let $A$ be a $\grp$-graded algebra.
We endow $A$ with the structure of a $(\grp,\bil)$-graded Lie algebra, stating that:
\begin{gather*}
  [f,g]\coloneqq f\circ g -\bil(f,g)\;g\circ f.
\end{gather*}
This applies in particular if $A=\uEnd(M)$ for some $\grp$-graded $\ring$-module $M$.
If $\fg$ is a $(\grp,\bil)$-graded Lie algebra, an \emph{action of $\fg$ on $M$} is a morphism of $(\grp,\bil)$-graded Lie algebras $\fg\to\uEnd(M)$. We say that $M$ is a \emph{$\fg$-module}, and a \emph{morphism of $\fg$-modules} is a degree-preserving linear map intertwining the $\fg$-action.
Given two $\fg$-modules $M$ and $N$, we write $\Hom^\fg(M,N)$ the $\ring$-module of morphisms of $\fg$-modules.
Abusing notation, we denote $\uHom(M,N)$ the $G$-graded $\ring$-module of all linear maps, now endowed with the following $\fg$-action:
\begin{equation}
  \label{eq:inner_hom_g_categories}
  g\cdot \alpha \coloneqq\tau_g^M\circ\alpha - \bil(g,\alpha)\;\alpha\circ\tau_g^N,
\end{equation}
for $g\in\fg$ and $\alpha\in\uHom(M,N)$, and where $\tau_g^M$ (resp.\ $\tau_g^N$) denotes the action of $g$ on $M$ (resp.\ $N$).

\begin{example}
  \label{ex:Lie_algebra_as_graded_algebra}
  If $(\grp,\bil)$ is trivial, a $(\grp,\bil)$-graded Lie algebra is a Lie algebra over $\ring$. If only $\bil$ is trivial, a $(\grp,\bil)$-graded Lie algebra is a Lie algebra over $\ring$ equipped with a $\grp$-grading.
\end{example}

\begin{example}[super Lie algebra]
  \label{ex:super_Lie_algebra_as_graded_algebra}
  If $\grp=\bZ/2\bZ=\{\ov{0},\ov{1}\}$ and $\bil(n,m) = (-1)^{nm}$, a $(\grp,\bil)$-graded Lie algebra is a super Lie algebra over $\ring$.
  In this setting, we often write $\abs{v}\coloneqq\deg v$.
  Explicitly, a \emph{Lie superalgebra} is a super vector space $\fg$ endowed with a bilinear degree-preserving map $[-,-]\colon\fg\otimes \fg\to\fg$, satisfying the following axioms:
  \begin{IEEEeqnarray*}{Cl}
    [v,w] = -(-1)^{\abs{v}\abs{w}}[w,v]&\text{graded symmetry}\\{}
    [u,[v,w]] + (-1)^{\abs{u}(\abs{v}+\abs{w})}
    [v,[w,u]] + (-1)^{\abs{w}(\abs{u}+\abs{v})}[w,[u,v]] 
    = 0
    \qquad&\text{graded Jacobi identity}
  \end{IEEEeqnarray*}
\end{example}

\begin{example}[$\gloo$]
  \label{ex:defn_gloo}
  The Lie superalgebra $\gloo$ is presented by generators $\{\lieh_1,\lieh_2,\liee,\lief\}$, where $\abs{\lieh_1}=\abs{\lieh_2}=\ov{0}$ and $\abs{\liee}=\abs{\lief}=\ov{1}$, and relations
  \begin{align*}
    [\liee,\lief] &= \lieh_1+\lieh_2 & [\liee,\liee]&=[\lief,\lief]=[\lieh_i,\lieh_j]=0\\
    [\lieh_1,\liee] &=\liee & [\lieh_1,\lief] &=-\lief\\
    [\lieh_2,\liee] &=-\liee & [\lieh_2,\lief] &=\lief
  \end{align*}
\end{example}

\begin{example}[$\sloo$]
  \label{ex:defn_sloo}
  Setting $\lieh\coloneqq\lieh_1+\lieh_2$ defined the Lie super algebra $\sloo$ as a sub-algebra $\sloo\subset\gloo$.
  In other words, the Lie superalgebra $\sloo$ presented by generators $\{\lieh,\liee,\lief\}$, where $\abs{\lieh}=\ov{0}$ and $\abs{\liee}=\abs{\lief}=\ov{1}$, and relations
  \begin{align*}
    [\liee,\lief] &= \lieh & [\liee,\liee]&=[\lief,\lief]=[\lieh,\lieh]=0\\
    [\lieh,\liee] &=0 & [\lieh,\lief] &=0
  \end{align*}
\end{example}

Anticipating, we give some specific data for $\ring$, $\grp$ and $\bil$ which will be used in the definition of graded $\glt$-foams, and define certain ``covering'' Lie algebras.

\begin{definition}
  \label{defn:graded_structure_foam}
  Let $\ringfoam$ be a commutative ring together with three invertible elements $X$, $Y$ and $Z\in{\ringfoam}^\times$ such that $X^2=Y^2=1$.
  Given this data, let $\bilfoam$ be the following bilinear form for the abelian group $G\coloneq \bZ^2$:
  \begin{align*}
    \bilfoam\colon \bZ^2\times\bZ^2 &\to {\ringfoam}^\times,\\
    ((a,b),(c,d)) &\mapsto X^{ac}Y^{bd} Z^{ad-bc}.
  \end{align*}
  We say ``restrict to the even case'' to mean choosing $X=Y=Z=1$, and ``restrict to the odd, or super, case'' to mean choosing $X=Z=1$ and $Y=-1$.
\end{definition}

\begin{example}[$\grgl_2$]
  \label{ex:defn_covering_gl2}
  Let $\ringfoam$ and $\bilfoam$ as in \cref{defn:graded_structure_foam}.
  Let $\grgl_2$, called \emph{covering $\fgl_2$}, be the $(\bZ^2,\bilfoam)$-graded Lie algebra defined as follows.
  As a $\ringfoam$-module, $\grgl_2$ is generated by the following homogeneous vectors:
  \begin{gather*}
    \deg(\lief)=(1,1),\;
    \deg(\liee)=(-1,-1),\;
    \deg(\lieh_1)=(0,0)
    \an\deg(\lieh_2)=(0,0).
  \end{gather*}
  The structure of graded Lie algebra is then given as follows:
  \begin{align*}
    [\liee,\lief] &= \lieh_1+XY\lieh_2 & [\liee,\liee]&=[\lief,\lief]=[\lieh_i,\lieh_j]=0\\
    [\lieh_1,\liee] &=\liee & [\lieh_1,\lief] &=-\lief\\
    [\lieh_2,\liee] &=-\liee & [\lieh_2,\lief] &=\lief. 
  \end{align*}
  % Leibniz: (automatically of if two of f,g,h are equal)
  % [e,[h_1,h_2]]: e-e=0
  % [f,[h_1,h_2]]: f-f=0
  % [e,[f,h_1]]: [e,f]+\bil(e,f)[f,e]=0
  We further denote $\grgl_2^-\coloneqq\langle\lief\rangle$ and $\grgl_2^{\leq}\coloneqq\langle\lief,\lieh_1,\lieh_2\rangle$, and $\grgl_2^+\coloneqq\langle\liee\rangle$ and $\grgl_2^{\geq}\coloneqq\langle\liee,\lieh_1,\lieh_2\rangle$.
  Restricting to even and odd, we have $\grgl_2^{\leq}\vert_{X=Y=Z=1}=\fgl_2^{\leq}$ and $\grgl_2^{\leq}\vert_{X=Z=1,Y=-1}=\gloo$, respectively.
\end{example}

\begin{example}[$\grsl_2$]
  \label{ex:defn_covering_sl2}
  Following \cref{ex:defn_covering_gl2}, set $\lieh\coloneqq\lieh_1-XY\lieh_2$. The $(\bZ^2,\bilfoam)$-graded Lie algebra $\grsl_2\subset\grgl_2$, called \emph{covering $\fsl_2$}, is defined as generated by $\lief$, $\liee$ and $\lieh$.
  In other words, it has the following defining relations:
  \begin{align*}
    [\liee,\lief] &= \lieh & [\liee,\liee]&=[\lief,\lief]=[\lieh,\lieh]=0\\
    [\lieh,\liee] &=(1+XY)\liee & [\lieh,\lief] &=-(1+XY)\lief
  \end{align*}
  Evaluating to even recovers $\slt\subset\glt$, while evaluating to odd recovers $\mathfrak{sl}_{1|1}\subset \mathfrak{gl}_{1|1}$.
  Note that when working over a field of characteristic two, $\fsl_2=\fsl_{1|1}$.
  Similarly to \cref{ex:defn_covering_gl2}, one can define $\grsl_2^-$, $\grsl_2^{\leq 0}$, $\grsl_2^+$ and $\grsl_2^{\geq 0}$.
\end{example}

\subsubsection{$\fg$-categories}
\label{subsubsec:g-categories}

We denote $\gMod$ the closed symmetric monoidal category of $\fg$-mo\-dules and morphisms of $\fg$\nbd-mo\-dules.
Its closed symmetric monoidal structure coincides with the closed symmetric monoidal structure of $\Mod_{\grp,\bil}$ via the forgetful functor; to complete the definition of the structure, it suffices to define the relevant $\fg$-actions.
For the monoidal structure, the $\fg$-action on the monoidal unit $\ring$ is trivial, and the $\fg$-action on the tensor product $M\otimes N$ is defined as
\[
g\cdot (m\otimes n)\coloneqq (g\cdot m)\otimes n + \bil(g,m) \,m\,(g\otimes n).
\]
One could view this symmetric monoidal structure as coming from some graded Hopf structure on the enveloping algebra of $\fg$; we omit this point of view.
The inner Hom is $\uHom$ with the structure of $\fg$-module given in \eqref{eq:inner_hom_g_categories}.

% Recall that given a monoidal category $\cV$, there is a notion of $\cV$-enriched category (see e.g.\ \cite[Chapter~6]{Borceux_HandbookCategoricalAlgebra_1994a}).

\begin{definition}
  A \emph{$\fg$-category} (resp.\ a \emph{$\fg$-functor}) is a $(\gMod)$-enriched category (resp.\ a $(\gMod)$-enriched functor).
\end{definition}

Note that this definition does not depend on the symmetric structure on $\gMod$.

We unpack the definition.
Given the forgetful functor $\gMod\to\Mod_{\grp}$, a $\fg$-category $A$ is in particular a $\grp$-graded category.
In addition, the $\fg$-category $A$ carries a family of linear maps
\begin{equation}
  \label{eq:g-cat-family-of-actions}
\fg\to \uEnd(\Hom_A(u,v))
\end{equation}
for each pair of objects $(u,v)$, that satisfies the \emph{$(\grp,\bil)$-graded Leibniz rule}:
\begin{equation}
  \label{eq:graded_Leibniz_rule}
  g\cdot (\alpha\circ\beta) = (g\cdot \alpha)\circ \beta + \bil(g,\alpha)\, f\circ(g\cdot \beta),
\end{equation}
where $\alpha$ and $\beta$ are suitably composable morphisms of $A$.
Whenever a $\grp$-graded category $A$ is equipped with a family of $\fg$-module morphisms as in \eqref{eq:g-cat-family-of-actions} satisfying the graded Leibniz rule \eqref{eq:graded_Leibniz_rule}, we say that \emph{$\fg$ acts by derivation on $A$}.

\begin{lemma}
  A $\fg$-category is the same as $\grp$-graded category equipped with an action of $\fg$ by derivation.\hfill\qed
\end{lemma}

\begin{remark}
  If $w$ is an object of $A$, it follows from the graded Leibniz rule that $g\cdot \id_w = g\cdot(\id_w\circ\id_w) = g\cdot \id_w + g\cdot \id_w$, so that $g\cdot \id_w=0$.
\end{remark}

\begin{example}
  \label{ex:ugMod}
  Let $\ugMod$ be the symmetric monoidal category whose objects are $\fg$-modules and with $\uHom(M,N)$ as the $\fg$-module homspace between $M$ and $N$.
  By definition, the category $\ugMod$ is a $\fg$-category.
  In fact, it is the $(\gMod)$-enriched category determined by the closed monoidal structure on $\gMod$, whose underlying category (as a $(\gMod)$-enriched category) is $\gMod$.
\end{example}

\begin{definition}
  \label{defn:g_equivariant}
  Let $A$ be a $\fg$-category. A morphism $\alpha$ is said to be \emph{$\fg$-equivariant} if $g\cdot\alpha=0$ for all $g\in\fg$.
\end{definition}

If $A=\ugMod$, then a morphism $\alpha$ is $\fg$-equivariant in the sense of \cref{defn:g_equivariant} if and only if it is $\fg$-equivariant in the usual sense, that is, if $\alpha$ intertwines the $\fg$-action on its source and target.

% Given a $\fg$-category $A$, one can consider the contravariant Yoneda embedding:
% \begin{align}
%   \label{eq:yoneda_embedding}
%   Y^{\mathrm{op}}\colon A^{\mathrm{op}}&\to \Fun(A,\gMod),\\
%   w&\mapsto\Hom_A(-,w),
% \end{align}

\subsubsection{$\fg$-2-categories}
\label{subsubsec:g-2-categories}

Recall that if $\cV$ is a symmetric monoidal category, then the category $\cV\md\cC at$ of $\cV$-enriched categories is itself symmetric monoidal, and one can enriched over $\cV\md\cC at$. A \emph{$\cV$-enriched 2-category} is a $(\cV\md\cC at)$-enriched category.

\begin{definition}[{\cite[Remark~2.7]{SV_OddKhovanovHomology_2023}}]
  A $(\grp,\bil)$-graded-2-category is a $(\Mod_{\grp,\bil})$-enriched 2-cate\-gory.
\end{definition}

Unpacking the definition, a $(\grp,\bil)$-graded-2-category is akin to a $\grp$-graded $\ring$-linear strict 2-category, except that the interchange law is replaced by the \emph{graded interchange law}:
\begin{equation*}
  \tikzpic{
    \draw (0,2) node[above] {\scriptsize $v'$} to
      node[fill=black,circle,inner sep=2pt,pos=.3] {}
      node[right,pos=.3] {\scriptsize $\beta$}
        (0,0) node[below] {\scriptsize $u'$};
    \draw (1,2) node[above] {\scriptsize $v$} to
      node[fill=black,circle,inner sep=2pt,pos=.7] {}
      node[right,pos=.7] {\scriptsize $\alpha$}
        (1,0) node[below] {\scriptsize $u$};
  }[scale=0.8]
  \;=\;\bil(\deg\alpha,\deg\beta)
  \tikzpic{
    \draw (0,2) node[above] {\scriptsize $v'$} to
      node[fill=black,circle,inner sep=2pt,pos=.7] {}
      node[right,pos=.7] {\scriptsize $\beta$}
        (0,0) node[below] {\scriptsize $u'$};
    \draw (1,2) node[above] {\scriptsize $v$} to
      node[fill=black,circle,inner sep=2pt,pos=.3] {}
      node[right,pos=.3] {\scriptsize $\alpha$}
      (1,0) node[below] {\scriptsize $u$};
  }[scale=0.8]
\end{equation*}

\begin{definition}
  A \emph{$\fg$-2-category} (resp.\ $\fg$-2-functor) is a $(\gMod)$-enriched 2-category (resp.\ $(\gMod)$-enriched 2-functor).
  A \emph{$\fg$-monoidal category} is a one-object $\fg$-2-category.
\end{definition}

We unpack the definition.
A $\fg$-2-category is in particular a $(\grp,\bil)$-graded-2-category, denoting its horizontal (resp.\ vertical) composition by $\otimes$ (resp.\ $\circ$). In addition, for each pair of objects $(x,y)$ the hom-category $\Hom(x,y)$ is a $\fg$-category.
Furthermore, the action of $\fg$ satisfy the $(\grp,\bil)$-graded Leibniz rule with respect to the horizontal composition;
equivalently, the action commutes with horizontal whiskering:
\begin{equation}
  \label{eq:axiomC_gcategories}
  \fg\cdot(\id_u\otimes\alpha\otimes\id_v) = \id_u\otimes(\fg\cdot\alpha)\otimes\id_v,
\end{equation}
where $u,v$ are 1-morphisms and $\alpha$ is a 2-morphism, suitably composable.

A $(\grp,\bil)$-graded-2-category $\cA$ equipped with a family of $\fg$-module morphisms
\[\fg\to\uEnd(\Hom_\cA(u,v))\]
indexed by pair of 1-morphisms $(u,v)$ with the same source and target, such that the action of $\fg$ defines an action by derivation on each Hom-category $\Hom_\cA(i,j)$ for pair of objects $(i,j)$, and furthermore verifies axiom \eqref{eq:axiomC_gcategories}, we say that \emph{$\fg$ acts by derivation on $\cA$}.

\begin{lemma}
  \label{lem:equivalent_defn_g2cat}
  A $\fg$-2-category is the same as a $(\grp,\bil)$-graded-2-category equipped an action of $\fg$ by derivation.\hfill\qed
\end{lemma}

\begin{example}
  \label{ex:sl2category_as_gcategory}
  Following up on \cref{ex:Lie_algebra_as_graded_algebra}, if $\ring=\bZ$, if $(\grp,\bil)=(\bZ,1)$ and if $\fg=\fsl_2$ equipped with the $\bZ$-grading $\abs{\lief}=2$, $\abs{\liee}=-2$ and $\abs{\lieh}=0$, then a $\fg$-monoidal category is an $\slt$-category in the sense of \cite{EQ_Actions$mathfraksl_2$Algebras_2023}.
\end{example}

\begin{example}
  \label{ex:dg_as_gcategory}
  Let $(\grp,\bil)=(\bZ,1)$ and $\ring$ a ring of characteristic $p$.
  If $\fg=\ring\partial$ is the one-dimensional abelian $(\grp,\bil)$-graded Lie algebra concentrated in degree $\abs{\partial}=2$, a $\fg$-monoidal category is a graded monoidal category equipped with an action by derivation $\partial$ of degree $2$.
  If this action is $p$-nilpotent, then this category is a $p$-DG-category in the sense of hopfological algebra \cite{Khovanov_HopfologicalAlgebraCategorification_2016,KQ_ApproachCategorificationSmall_2015,Qi_HopfologicalAlgebra_2014}.
\end{example}

\begin{example}
  \label{ex:super_dg_as_gcategory}
  Let $(\grp,\bil)$ as in \cref{ex:super_Lie_algebra_as_graded_algebra}.
  If $\fg=\ring\partial$ is the one-dimensional abelian super Lie algebra concentrated in degree $\abs{\partial}=\ov{1}$, then a $\fg$-2-category is a dg-2-super\-ca\-te\-gory in the sense of \cite{EL_DGStructuresOdd_2020}.
\end{example}

\begin{example}
  \label{ex:graded_commutative_dg_algebra}
  A $\fg$-2-category with one object and one morphism is a $(\grp,\bil)$-graded-com\-mu\-ta\-tive algebra equipped with an action of $\fg$ by derivation.
  In the setting of , then $(\grp,\bil)$-graded-commutativity recovers graded-commutativity in the usual sense, and if further the action of $\partial$ is nilpotent, we recover the notion of a graded-com\-mu\-ta\-tive DG-algebra.
\end{example}

\begin{remark}
  \label{rem:derivation_uniquely_defn_on_generators_of_A}
  If $\cA$ is a $(\grp,\bil)$-graded-2-category defined by generators and relations, an action by derivation is soly determined by the action on the generators. Conversely, to define an action by derivation, it suffices to define it on the generators and verify that it preserves the defining relations.
  The graded interchange law needs not be verified: it follows from the graded Leibniz rule that any action by derivation preserves the graded interchange law.
\end{remark}

\begin{remark}
  \label{rem:action_uniquely_defn_by_generators_of_g}
  If $\cA$ is a $(\grp,\bil)$-graded-2-category and $\fg$ is a $(\grp,\bil)$-graded Lie algebra defined by generators and relations, an action of $\fg$ on $\cA$ by derivation is solely determined by the action of the generators of $\fg$.
  Conversely, to define an action $\fg$ on $\cA$ by derivation, it suffices to define it on the generators of $\fg$, and verify that it satisfies the defining relations of $\fg$.
  % Indeed, the bracket of two elements that verify the graded Leibniz rule also verifies the graded Leibniz rule, the bracket of two elements that commute with the horizontal whiskering also commute with the horizontal whiskering.
  % \begin{align*}
  %   [f,g](xy)
  %   =&(f\circ g -\bil(f,g)g\circ f)(xy)
  %   =fg(xy)-\bil(f,g)gf(xy)
  %   \\
  %   =&fg(x)y+\bil(f,g(x))g(x)f(y)+\bil(g,x)f(x)g(y)+\bil(f,x)\bil(g,x)xfg(y)
  %   \\
  %   &{}-\bil(f,g)\big[gf(x)y+\bil(g,f(x))f(x)g(y)+\bil(f,x)g(x)f(y)+\bil(g,x)\bil(f,x)xgf(y)\big]
  %   \\
  %   =&[f,g](xy)+\bil(f+g,x)x[f,g](y)
  % \end{align*}
\end{remark}

\subsubsection{Twisting $\fg$-2-categories}
\label{subsubsec:twisting_g2categories}

Let $\cA$ be a $\fg$-2-category.
Consider a family of degree-preserving linear maps
\[\tau = (\tau_w\colon\fg\to\End_\cA(w))_w,\]
indexed by 1-morphisms $w$ of $\cA$.
We say that $\tau$ is \emph{flat} if for each $w$, we have
\begin{gather*}
  \tau_w([g,h])=g\cdot\tau_w(h)-\bil(g,h)\, h\cdot\tau_w(g),
\end{gather*}

\begin{definition}
  \label{defn:family_of_twists}
  Let $\cA$ be a $\fg$-2-category.
  A family $\tau$ as above is said to be a \emph{a family of twists} if it is flat, satisfies the Leibniz rule and has a graded-com\-mu\-ta\-tive image.
\end{definition}

Here ``satisfies the Leibniz rule'' means that $\tau_{u\otimes v}(g) = \tau_{u}(g)\otimes v + u\otimes\tau_v(g)$ and ``has a graded-com\-mu\-ta\-tive image'' means that the image of each $\tau_w$ is $(\grp,\bil)$-graded-com\-mu\-ta\-tive (see \cref{defn:graded-commutative}).

\begin{remark}
  \label{rem:twist_only_check_on_generators}
  A family of twists is determined by its value on generators of 1-morphisms. Moreover, flatness and graded-com\-mu\-ta\-tive image need only be checked on the generators.
  % \begin{align*}
  %   \tau_{u\otimes v}([g,h]) 
  %   &= \tau_{u}([g,h])\otimes v + u\otimes\tau_v([g,h])\\
  %   &=\big(g\cdot\tau_u(h)-\bil(g,h)\, h\cdot\tau_u(g)\big)\otimes v 
  %   + u\otimes\big(g\cdot\tau_v(h)-\bil(g,h)\, h\cdot\tau_v(g)\big)\\
  %   &= g\cdot\tau_{u\otimes v}(h)-\bil(g,h)\, h\cdot\tau_{u\otimes v}(g)
  % \end{align*}
\end{remark}

\begin{proposition}
  \label{prop:family_of_twists_gives_g2cat}
  Let $\cA$ be a $\fg$-2-category and $\tau$ a family as above.
  For each pair of 1-morphisms $(u,v)$ with the same source and target, define a degree-preserving linear map
  \[\fg\to\uEnd(\Hom_\cA(u,v)),\quad g\mapsto g\cdot_\tau (-)\]
  where for $\alpha\colon u\to v$  a 2-morphism in $\cA$:
  \begin{gather*}
    g\cdot_\tau \alpha \coloneqq \tau_v(g)\circ\alpha + g\cdot \alpha - \bil(g,\alpha)\,\alpha\circ\tau_u(g).
  \end{gather*}
  Let $\cA^\tau$ be the underlying $(\grp,\bil)$-graded-2-category of $\cA$ equipped with this family of maps. If $\tau$ is a family of twists, then $\cA^\tau$ is a $\fg$-2-category.
\end{proposition}

\begin{proof}
  We first check that the action is well-defined; that is, each map $\fg\to\uEnd_\ring(\Hom_\cA(u,v))$ is a $\fg$-morphism.
  For a 2-morphism $\alpha\colon u\to v$, we compute:
  \begin{IEEEeqnarray*}{rCl}
    g\cdot_\tau(h\cdot_\tau \alpha)
    &=&
    g\cdot_\tau(\tau_v(h)\circ \alpha + h\cdot \alpha - \bil(h,\alpha)\;\alpha\circ\tau_u(h))
    \\[1ex]
    &=&
    \tau_v(g)\;(\tau_v(h)\circ\alpha + h\cdot \alpha - \bil(h,\alpha)\;\alpha\circ\tau_u(h))
    \\
    &&{}+
    g\cdot(\tau_v(h)\circ\alpha + h\cdot \alpha - \bil(h,\alpha)\;\alpha\circ\tau_u(h))
    \\
    &&{}-\bil(g,h+\alpha)\;
    (\tau_v(h)\circ\alpha + h\cdot \alpha - \bil(h,\alpha)\;\alpha\circ\tau_u(h))\;
    \tau_u(g)
    \\[1ex]
    &=&
    \tau_v(g)\;(\underline{\tau_v(h)\circ\alpha}_1 + \underline{h\cdot \alpha}_2 - \underline{\bil(h,\alpha)\;\alpha\circ\tau_u(h)}_3)
    \\
    &&{}+
    \underline{(g\cdot\tau_v(h))\circ\alpha}_4 + \underline{\bil(g,h)\;\tau_v(h)\circ(g\cdot \alpha)}_2
    + \underline{g\cdot (h\cdot \alpha)}_5
    \\
    &&{}- \bil(h,\alpha) \big[\underline{(g\cdot \alpha)\circ\tau_u(h)}_6 + \underline{\bil(g,\alpha)\;\alpha\;g\cdot \tau_u(h)}_7\big]
    \\
    &&{}-\bil(g,h+\alpha)\;
    (\underline{\tau_v(h)\;\alpha}_3 + \underline{h\cdot \alpha}_6 - \underline{\bil(h,\alpha)\;\alpha\;\tau_u(h)}_8)\;
    \tau_u(g)
  \end{IEEEeqnarray*}
  Here we labelled each term with a number according to how they simplify in the computation below:
  \begin{IEEEeqnarray*}{rCl}
    g\cdot_\tau(h\cdot_\tau \alpha)-\bil(g,h)h\cdot_\tau(g\cdot_\tau \alpha)
    &=&
    \underline{\tau_v([g,h])\;\alpha}_4 + \underline{[g,h]\cdot \alpha}_5 - \underline{\bil(h+g,\alpha)\;\alpha\;\tau_u([g,h])}_7
    \\
    &=& [g,h]\cdot_\tau \alpha.
  \end{IEEEeqnarray*}
  Terms 4 and 7 simplify thanks to flatness, term 5 simplify as $\cdot$ is an action of $\fg$, and the remaining terms cancel, with terms 1 and 8 cancelling thanks to graded commutativity.

  Following \cref{lem:equivalent_defn_g2cat},
  it remains to check that the $\fg$-action verifies the Leibniz rule and commutes with horizontal whiskering.
  The former follows from graded Leibniz rule for $\cdot$, and the latter follows from the fact that $\tau$ .
\end{proof}

% \begin{remark}
%   It follows directly from the graded Leibniz rule that:
%   \begin{align*}
%     \partial^2(xy)&= \partial\big[\partial(x)y+\bil(\partial,x)x\partial(y)\big]\\
%     &=\partial^2(x)y+\big[\bil(\partial,\partial(x))+\bil(\partial,x)\big]\partial(x)\partial(y)+\bil(\partial,x)^2x\partial^2(y)\\
%     &=\partial^2(x)y+\big[\bil(\partial,\partial)+1\big]\bil(\partial,x)\partial(x)\partial(y)+\bil(\partial,x)^2x\partial^2(y)
%   \end{align*}
%   Assume $\partial^2(x)=\partial^2(y)=0$.
%   On the one hand, if $\bil(\partial,\partial)=-1$ (super case), then $\partial^2(xy)=0$.
%   On the other hand, if $\bil(\partial,\partial)=1$ and if the characteristic of $\ring$ is two ($2$-DG case), then $\partial^2(xy)=0$.
%   In particular, if we are in one of the above cases and $\partial^2$ is zero on a set of generators of $A$, then $(A,\partial)$ is a $(\grp,\bil)$-graded dg-algebra.
% \end{remark}

\subsection{Review of graded \texorpdfstring{$\glt$}{gl2}-foams}
\label{subsec:review_graded_foam}

In this subsection, we review the graded-2-category of $\glt$-foams $\pregfoam_d$ as introduced in \cite{SV_OddKhovanovHomology_2023}, and refer to \textit{op.\ cit.} for further details.

Fix a positive integer $d\in\bN$.
The objects of $\pregfoam_d$ are
\begin{gather*}
  \ob(\pregfoam_d)\coloneqq
  \bigsqcup_{k\in\bN}\{\lambda\in\{1,2\}^k\mid \lambda_1+\ldots+\lambda_k=d\}.
\end{gather*}
For each $\lambda\in\ob(\pregfoam_d)$ with $k$ coordinates, we label its coordinates with
\[l_\lambda\colon\{1,\ldots,k\}\to\{1,\ldots,d\},\]
setting $l_\lambda(i)=\sum_{j<i}\lambda_j+1.$
For instance, $l_{(1,1,2,1)}=(1,2,3,5)$.
In other words, the label $l_\lambda(i)$ is a sort of ``weighted coordinate'', where coordinate with value $2$ counts double.
Foreseeing the diagrammatics, we call this label the \emph{colour} of the coordinate.

The 1-morphisms of $\pregfoam_d$ are \emph{directed $\glt$-webs} (or simply \emph{webs}), such as:
\begin{center}
  \tikzpic{
    \webid[0][1.5][1]\webs[1][1.5]\webm[2][1.5]\webid[3][1.5][1]
    \webm\webid[1][0][1]\webid[2][0][1]\webs[3][0]
  }[xscale=.5,yscale=.4]
\end{center}
In general, a web is obtained from \emph{merge webs} ($M\coloneqq\tikzpic{\webm}[xscale=.5,yscale=.4]$) and \emph{split webs} ($S\coloneqq\tikzpic{\webs}[xscale=.5,yscale=.4]$), by adding single lines ($\tikzpic{\draw[web1] (0,0) to (1,0);}[xscale=.5,yscale=.4]$) and double lines ($\tikzpic{\draw[web2] (0,0) to (1,0);}[xscale=.5,yscale=.4]$) on top and on the bottom and then composing horizontally.
Note that we read webs from right to left.
Our webs are \emph{directed}, in the sense that when reading from right to left, the vertical cross section always has the same width (counting double the double lines); that is, the integer $d$ is fixed.
We sometimes emphasize that point by orienting our webs from right to left.
A web $W$ has an underlying unoriented flat tangle diagram, denoted $\undcomp(W)$, given by forgetting the double lines and the orientation.

We now turn to the 2-morphisms of $\pregfoam_d$.
For convenience, and in contrast to the introduction, we shall use the shading diagrammatics \cite{Schelstraete_OddKhovanovHomology_2024,Schelstraete_RewritingModuloDiagrammatic_2025} throughout the rest of the paper.
It is given by projecting $\glt$-foams onto the plane along the front-to-back direction, and recording only the seams and 2-facets:
\begin{IEEEeqnarray*}{CcC}
  \vcenter{\hbox{\includegraphics[width=2.5cm]{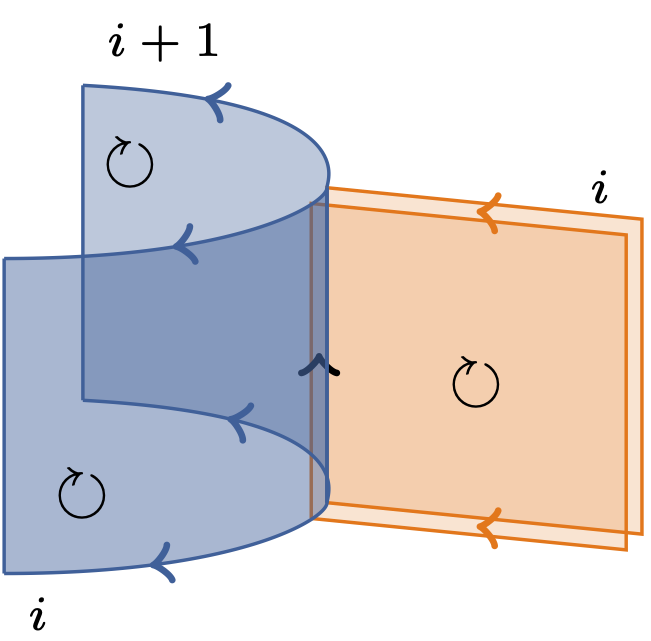}}}
  &
  \mspace{30mu}
  \leftrightarrow
  \mspace{30mu}
  &
  \tikzpic{\frst\node[left] at (0,1) {\footnotesize $i$};}[scale=1.5]
  \\*
  \text{\small $\glt$-foams}
  &&
  \text{\small shading diagrammatics}
\end{IEEEeqnarray*}
Recall the data $\ringfoam$, $\grpfoam$ and $\bilfoam$ from \cref{defn:graded_structure_foam}.

% \ringfoam for ring in foams
% \grpfoam for group in foams
% \bilfoam for bilinear map in foams

\begin{definition}
  \label{defn:review_foam}
  The $(\bZ^2,\bilfoam)$-graded-2-category $\pregfoam_d$ has its $(\bZ^2,\bilfoam)$-graded structure given as in \cref{defn:graded_structure_foam} and is presented with generators given in \cref{fig:string_generators_graded_foams} and relations given in \cref{fig:rel_diagfoam_graded}. 
\end{definition}

The \emph{quantum grading} is defined as $\qdeg(a,b)=a+b$ where $(a,b)$ is the $\grpfoam$-grading.
Although the quantum grading is defined from the $\grpfoam$-grading, we view it as a distinct grading. We denote $\gfoam_d$ the additive $q$-shifted closure of $\pregfoam$; that means we allow formal direct sums and shifts in the quantum grading on objects, and restrict to foams with quantum degree zero (see \cite[subsection~2.1]{SV_OddKhovanovHomology_2023} for details).
Compare to \cite{SV_OddKhovanovHomology_2023}, our notation is such that $\pregfoam_d=\gfoam_d^{\text{\cite{SV_OddKhovanovHomology_2023}}}$ 
and $\gfoam_d=((\underline{\gfoam_d})^\oplus_q)^{\text{\cite{SV_OddKhovanovHomology_2023}}}$.

Recall from \cref{defn:graded_structure_foam} what we mean by ``restrict to odd'' and ``restrict to even''.

\begin{definition}
  \label{defn:foam_and_sfoam}
  We denote $\prefoam_d=\pregfoam_d\vert_{X=Y=Z=1}$ the restriction of $\pregfoam_d$ to even and $\presfoam_d=\pregfoam_d\vert_{X=Z=1,Y=-1}$ the restriction of $\pregfoam_d$ to odd.
  We similarly define $\foam_d$ and $\sfoam_d$.
\end{definition}

This article mainly deals with three graded Lie algebras: the Lie algebra $\fg=\fgl_2^\leq$, the super Lie algebra $\fg=\gloo$, and the graded Lie algebra $\fg=\grgl_2^\leq$.
For each of these cases, we write $\fg\foam_d$ for $\foam_d$, for $\sfoam_d$ and for $\gfoam_d$, respectively.
We shall use similar notations throughout, depending on the choice of $\fg$.

\begin{remark}[monoidal 2-categorical structure]
  \label{rem:monoidal-2-categorical-structure}
  One could gather the graded-2-categories $\pregfoam_d$ together as a certain ``monoidal graded-2-category'', leveraging the canonical graded-2-functors
  \[
    \pregfoam_{d_1}\times\pregfoam_{d_2}\to\pregfoam_{d_1+d_2}
  \]
  given on the pair $(F_1,F_2)$ by putting $F_1$ in front of $F_2$; in shading diagrammatics, it amounts to shifting the labels of $F_2$ and superposing the diagrams. 
  While we avoid making this precise here, certain parts of our discussion implicitly use this extra monoidal structure. We refer to it as the \emph{front-back composition}, and denote it $\square$.
\end{remark}

\begin{remark}
  \label{rem:variant_graded_gl2_foams}
  There exists a variant of $\pregfoam_d$, denoted $\widetilde{\gfoam'_d}$ and with the same generators and relations, except for the following two relations:
  \begin{gather*}
    \tikzpic{
      \funzip[0][1]\fzip\node at (.8+.5,-.3+1) {\scriptsize $i$};
    }[scale=.7]
    =
    XYZ
    \mspace{10mu}
    \tikzpic{
      \fdot\node at (.2,-.2) {\scriptsize $i$};
    }[][(0,0)]
    + Z
    \mspace{10mu}
    \tikzpic{
      \fdot[0][0][2]\node at (.5,-.2) {\scriptsize $i+1$};
    }[][(0,0)]
    \quad\an\quad
    \tikzpic{
      \flst[0][3]\flst[0][2]\flst[0][1]\flst[0][0]
      \frst[.5][3][2]\frst[.5][2][2]\frst[.5][1][2]\frst[.5][0][2]
      \flst[1.5][3][2]\flst[1.5][2][2]\flst[1.5][1][2]\flst[1.5][0][2]
      \frst[2][3]\frst[2][2]\frst[2][1]\frst[2][0]
      \node[below=-2pt] at (0,0) {\scriptsize $i$};
      \node[below=-2pt] at (1.5,0) {\scriptsize $i+1$};
    }[scale=.35][(0,2*.35)]
    \;=\;XYZ^{-1}\;
    \tikzpic{
      \fzip[0][3][2]
      \funzip[0][0][2]
      \begin{scope}[scale=2]
        \fcup[-.25][1]
        \fcap[-.25][0]
      \end{scope}
      \node[below=-2pt] at (-.5,0) {\scriptsize $i$};
      \node[below=-2pt] at (1,0) {\scriptsize $i+1$};
    }[scale=.35][(0,2*.35)]
    \;.
  \end{gather*}
  It was shown in \cite{Schelstraete_RewritingModuloDiagrammatic_2025} that $\pregfoam_d$ and $\widetilde{\gfoam'_d}$ are to only two deformations of $\foam_d$, in a suitable sense.
  When comparing with the classical definition of odd Khovanov homology \cite{ORS_OddKhovanovHomology_2013}, working with $\pregfoam_d$ gives type Y odd Khovanov homology, while working with $\widetilde{\gfoam'_d}$ gives type X odd Khovanov homology.
  See also \cref{subsec:comparison_local_global}.
\end{remark}

\begin{figure}[p]
  \centering

  % SHADING DIAGRAMMATICS GENERATORS
  \begin{equation*}
    \allowdisplaybreaks
    \begin{array}{cc}
      \begin{tabular}{*{5}{c@{\hskip 7ex}}c}
        $\tikzpic{\fdot\node at (.2,-.2) {\scriptsize $i$};}$
        &
        $\tikzpic{\fcup\node at (1,.4) {\scriptsize $i$};}$
        &
        $\tikzpic{\fcap\node at (1,.4) {\scriptsize $i$};}$
        &
        $\tikzpic{\fzip\node at (1,.6) {\scriptsize $i$};}$
        &
        $\tikzpic{\funzip\node at (1,.6) {\scriptsize $i$};}$
        \\*[1ex]
        \small dot &\small cup & \small cap &  \small zip & \small unzip
        \\*[1ex]
        $(1,1)$ & $(0,-1)$ & $(-1,0)$ & $(1,0)$ & $(0,1)$
      \end{tabular}
      &
      \\[7ex]
      \begin{tabular}{*{4}{c@{\hskip 3ex}}c}
        %%%% downward crossing %%%%
        $\tikzpic{
          \fbcro
          \node[left=-3pt] at (0,0) {\scriptsize $i$};
          \node[right=-3pt] at (1,0){\scriptsize $j$};
          % \node at (1.4,0) {\scriptsize $\lambda$};
        }$
        &
        %%%% leftward crossing %%%%
        $\tikzpic{
          \flcro
          \node[left=-3pt] at (0,0) {\scriptsize $i$};
          \node[right=-3pt] at (1,0){\scriptsize $j$};
          % \node at (1.4,0) {\scriptsize $\lambda$};
        }$
        &
        %%%% upward crossing %%%%
        $\tikzpic{
          \ftcro
          \node[left=-3pt] at (0,0) {\scriptsize $i$};
          \node[right=-3pt] at (1,0){\scriptsize $j$};
          % \node at (1.4,0) {\scriptsize $\lambda$};
        }$
        &
        %%%% rightward crossing %%%%
        $\tikzpic{
          \frcro
          \node[left=-3pt] at (0,0) {\scriptsize $i$};
          \node[right=-3pt] at (1,0){\scriptsize $j$};
          % \node at (1.4,0) {\scriptsize $\lambda$};
        }$
        &
        if $\abs{i-j}>1$
        \\*[1ex]
        % \small \stackunder{\text{downward}}{\text{crossing}} & \small \stackunder{\text{rightward}}{\text{crossing}} & \small \stackunder{\text{upward}}{\text{crossing}} & \small \stackunder{\text{leftward}}{\text{crossing}}
        \small \text{downward crossing} & \small \text{leftward crossing} & \small \text{upward crossing} & \small \text{rightward crossing}
        \\*[1ex]
        $(0,0)$ & $(0,0)$ & $(0,0)$ & $(0,0)$
      \end{tabular}
    \end{array}
  \end{equation*}
  \vspace*{-.5cm}

  \caption{Generators in $\pregfoam_d$. Each generator has a grading in $\bZ\times\bZ$.}
  \label{fig:string_generators_graded_foams}

  % SHADING DIAGRAMMATICS RELATIONS
  \def\scl{.7}
  \newcommand{\ins}{0pt}% space between diagram and text
  \newcommand{\outs}{7pt}% space between two (set of) relation(s)
  \begin{gather*}
    %%%%% BRAID-LIKE MOVES %%%%%
    \tikzpic{
      \draw[foamdraw1]  +(0,-.75) node[below] {\textcolor{black}{\scriptsize $i$}}
        .. controls (0,-.375) and (1,-.375) .. (1,0)
        .. controls (1,.375) and (0, .375) .. (0,.75);
      \draw[foamdraw2]  +(1,-.75) node[below] {\textcolor{black}{\scriptsize $j$}}
        .. controls (1,-.375) and (0,-.375) .. (0,0)
        .. controls (0,.375) and (1, .375) .. (1,.75);
    }[scale=.8*\scl][(0,0)]
    \;=\;
    \tikzpic{
      \draw[foamdraw1] (0,-.75) node[below] {\textcolor{black}{\scriptsize $i$}} to (0,.75);
      \draw[foamdraw2] (1,-.75) node[below] {\textcolor{black}{\scriptsize $j$}} to (1,.75);
      % \node at (1.3,0) {\scriptsize $\lambda$};
    }[scale=.8*\scl][(0,0)]
    \mspace{80mu}
    \tikzpic{
      \draw[foamdraw1]  +(0,0)node[below] {\textcolor{black}{\scriptsize $i$}}
        .. controls (0,0.5) and (2, 1) ..  +(2,2);
      \draw[foamdraw3]  +(2,0)node[below] {\textcolor{black}{\scriptsize $k$}}
        .. controls (2,1) and (0, 1.5) ..  +(0,2);
      \draw[foamdraw2]  (1,0)node[below] {\textcolor{black}{\scriptsize $j$}}
        .. controls (1,0.5) and (0, 0.5) ..  (0,1)
        .. controls (0,1.5) and (1, 1.5) ..  (1,2);
      % \node at (2.1,1) {\scriptsize $\lambda$};
    }[scale=.7*\scl][(0,1*.7*\scl)]
    \;=\;
    \tikzpic{
      \draw[foamdraw1]  +(0,0)node[below] {\textcolor{black}{\scriptsize $i$}}
        .. controls (0,1) and (2, 1.5) ..  +(2,2);
      \draw[foamdraw3]  +(2,0)node[below] {\textcolor{black}{\scriptsize $k$}}
        .. controls (2,.5) and (0, 1) ..  +(0,2);
      \draw[foamdraw2]  (1,0)node[below]{\textcolor{black} {\scriptsize $j$}}
        .. controls (1,0.5) and (2, 0.5) ..  (2,1)
        .. controls (2,1.5) and (1, 1.5) ..  (1,2);
      % \node at (2.3,1) {\scriptsize $\lambda$};
    }[scale=.7*\scl][(0,1*.7*\scl)]
    \\[\ins]
    \text{\footnotesize braid-like relations}
    \\[\outs]
    %%%%% PITCHFORKS %%%%%
    \tikzpic{
      \draw[foamdraw1](-.5,.4) to (0,-.3);
      \draw[foamdraw2] (0.3,-0.3)
        to[out=90, in=0] (0,0.2)
        to[out = -180, in = 40] (-0.5,-0.3);
      \node at (-0.5,-.5) {\scriptsize $j$};
      \node at (0,-.5) {\scriptsize $i$};
      % \node at (0.4,0) {\scriptsize $\lambda$};
    }[scale=1.2*\scl][(0,0)]
    \;=\;
    \tikzpic{
      \draw[foamdraw1](.6,.4) to (.1,-.3);
      \draw[foamdraw2] (0.6,-0.3)
        to[out=140, in=0] (0.1,0.2)
        to[out = -180, in = 90] (-0.2,-0.3);
      \node at (-0.2,-.5) {\scriptsize $j$};
      \node at (0.1,-.5) {\scriptsize $i$};
      % \node at (0.7,0) {\scriptsize $\lambda$};
    }[scale=1.2*\scl][(0,0)]
    \mspace{80mu}
    \tikzpic{
      \draw[foamdraw1](-.5,-.3) to (0,.4);
      \draw[foamdraw2] (0.3,0.4)
        to[out=-90, in=0] (0,-0.1)
        to[out = 180, in = -40] (-0.5,0.4);
      \node at (-0.5,.6) {\scriptsize $j$};
      \node at (-0.05,.6) {\scriptsize $i$};
      % \node at (0.5,0.1) {\scriptsize $\lambda$};
    }[scale=1.2*\scl][(0,0)]
    \;=\;
    \tikzpic{
      \draw[foamdraw1](.6,-.3) to (.1,.4);
      \draw[foamdraw2] (0.6,0.4)
        to[out=-140, in=0] (0.1,-0.1)
        to[out = 180, in = -90] (-0.2,0.4);
      \node at (-0.25,.6)  {\scriptsize $j$};
      \node at (0.15,.6) {\scriptsize $i$};
      % \node at (0.7,0.1) {\scriptsize $\lambda$};
    }[scale=1.2*\scl][(0,0)]
    \\[\ins]
    \text{\footnotesize pitchfork relations}
    \\[\outs]
    %%%%% ZIGZAGS %%%%%
    \begingroup
      \tikzpic{
        \fcap[0][1]\flst[2][1]\flst\fzip[1][0]
        \node[left=-3pt] at (2,2) {\scriptsize $i$};
      }[scale=.5][(0,1*.5)]
      =
      \tikzpic{
        \flst[0][1]\flst\node[left=-3pt] at (0,2) {\scriptsize $i$};
      }[scale=.5][(0,1*.5)]
      \mspace{30mu}
      \tikzpic{
        \frst[0][1]\fcap[1][1]\fzip[0][0]\frst[2][0]
        \node[left=-3pt] at (0,2) {\scriptsize $i$};
      }[scale=.5][(0,1*.5)]
      = X
      \tikzpic{
        \frst[0][1]\frst\node[left=-3pt] at (0,2) {\scriptsize $i$};
      }[scale=.5][(0,1*.5)]
      \mspace{30mu}
      \tikzpic{
        \funzip[0][1]\frst[2][1]\frst\fcup[1][0]
        \node[left=-3pt] at (2,2) {\scriptsize $i$};
      }[scale=.5][(0,1*.5)]
      = Z^2
      \tikzpic{
        \frst[0][1]\frst\node[left=-3pt] at (0,2) {\scriptsize $i$};
      }[scale=.5][(0,1*.5)]
      \mspace{30mu}
      \tikzpic{
        \flst[0][1]\funzip[1][1]\fcup[0][0]\flst[2][0]
        \node[left=-3pt] at (0,2) {\scriptsize $i$};
      }[scale=.5][(0,1*.5)]
      =
      YZ^2\;
      \tikzpic{
        \flst[0][1]\flst\node[left=-3pt] at (0,2) {\scriptsize $i$};
      }[scale=.5][(0,1*.5)]
    \endgroup
    \\[\ins]
    \text{\footnotesize zigzag relations (or adjunction relations)}
    \\[\outs]
    %%%%% DOT-TYPE RELATIONS %%%%%
    \begin{IEEEeqnarraybox}{cCcCc}
      %%%%% DOT ANNIHILATION %%%%%
      \left(\tikzpic{
        \fdot\node at (.2,-.2) {\scriptsize $i$};
      }[][(0,0)]\right)^2
      =
      0
      &\mspace{50mu}&
      %%%%% DOT SWAP %%%%%
      \tikzpic{
        \fdot[-1][1]\node at (-.7,.8) {\scriptsize $i$};
        \frst[0][1]\frst\node[left=-3pt] at (0,2) {\scriptsize $i$};
      }[scale=.5][(0,1*.5)]
      =\;
      \tikzpic{
        \fdot[-1][1][2]\node at (-.7,.5) {\scriptsize $i+1$};
        \frst[0][1]\frst\node[left=-3pt] at (0,2) {\scriptsize $i$};
      }[scale=.5][(0,1*.5)]
      &\mspace{50mu}&
      %%%%% DOT SLIDE %%%%%
      \tikzpic{
        \fdot[-1][1][2]\node at (-.7,.8) {\scriptsize $j$};
        \frst[0][1]\frst\node[left=-3pt] at (0,2) {\scriptsize $i$};
      }[scale=.5][(0,1*.5)]
      \;=\;
      \tikzpic{
        \fdot[1][1][2]\node at (-.7+2,.8) {\scriptsize $j$};
        \frst[0][1]\frst\node[left=-3pt] at (0,2) {\scriptsize $i$};
      }[scale=.5][(0,1*.5)]
      \mspace{20mu}
      \text{if }j\neq i,i+1
      \\[\ins]
      \text{\footnotesize dot annihilation}
      &&
      \text{\footnotesize dot migration}
      &&
      \text{\footnotesize dot slide}
    \end{IEEEeqnarraybox}
    \\[\outs]
    %%%%% BUBBLES %%%%%
    \begin{IEEEeqnarraybox}{cCc}
      %%%%% COUNTER-CLOCKWISE BUBBLES %%%%%
      \tikzpic{
        \fdot[.5][1]\node at (.2+.5,-.2+1) {\scriptsize $i$};
        \fcap[0][1]\fcup\node at (.8+.5,-.3+1) {\scriptsize $i$};
      }[scale=.7]
      =
      1
      \mspace{40mu}
      \tikzpic{
        \fcap[0][1]\fcup\node at (.8+.5,-.3+1) {\scriptsize $i$};
      }[scale=.7]
      =
      0
      & \mspace{70mu} &
      %%%%% CLOCKWISE BUBBLES %%%%%
      \begingroup
        \tikzpic{
          \funzip[0][1]\fzip\node at (.8+.5,-.3+1) {\scriptsize $i$};
        }[scale=.7]
        =
        {Z}
        \mspace{10mu}
        \tikzpic{
          \fdot\node at (.2,-.2) {\scriptsize $i$};
        }[][(0,0)]
        + XYZ
        \mspace{10mu}
        \tikzpic{
          \fdot[0][0][2]\node at (.5,-.2) {\scriptsize $i+1$};
        }[][(0,0)]
      \endgroup
      \\[\ins]
      \text{\footnotesize evaluation of bubbles}
      &&
      \text{\footnotesize evaluation of shaded disks}
    \end{IEEEeqnarraybox}
    \\[\outs]
    %%%%% SHADING RELATED RELATION %%%%%
    \begin{IEEEeqnarraybox}{cCc}
      %%%%% NECK-CUTTING RELATION %%%%%
      \begingroup
        \tikzpic{
          \flst[0][1]\flst\node[left=-3pt] at (0,2) {\scriptsize $i$};
          \frst[1][1]\frst[1][0];
        }[scale=.7][(0,1*.7)]
        \;=\;
        \tikzpic{
          \fdot[.5][1.8]\fcup[0][1]\fcap
          \node[left=-3pt] at (0,2) {\scriptsize $i$};
        }[scale=.7][(0,1*.7)]
        \;+\;
        \tikzpic{
          \fdot[.5][.2]\fcup[0][1]\fcap
          \node[left=-3pt] at (0,2) {\scriptsize $i$};
        }[scale=.7][(0,1*.7)]
      \endgroup
      & \mspace{120mu} &
      %%%%% ISOTOPY RELATION %%%%%
      \begingroup
      \tikzpic{
          \flst[0][3]\flst[0][2]\flst[0][1]\flst[0][0]
          \frst[.5][3][2]\frst[.5][2][2]\frst[.5][1][2]\frst[.5][0][2]
          \flst[1.5][3][2]\flst[1.5][2][2]\flst[1.5][1][2]\flst[1.5][0][2]
          \frst[2][3]\frst[2][2]\frst[2][1]\frst[2][0]
          \node[below=-2pt] at (0,0) {\scriptsize $i$};
          \node[below=-2pt] at (1.5,0) {\scriptsize $i+1$};
        }[scale=.35][(0,2*.35)]
        \;=\;Z^{-1}\;
        \tikzpic{
          \fzip[0][3][2]
          \funzip[0][0][2]
          \begin{scope}[scale=2]
            \fcup[-.25][1]
            \fcap[-.25][0]
          \end{scope}
          \node[below=-2pt] at (-.5,0) {\scriptsize $i$};
          \node[below=-2pt] at (1,0) {\scriptsize $i+1$};
        }[scale=.35][(0,2*.35)]
      \endgroup
      \\[\ins]
      \text{\footnotesize neck-cutting relation}
      &&
      \text{\footnotesize squeezing relation}
    \end{IEEEeqnarraybox}
  \end{gather*}
  \vspace*{-.5cm}

  \caption{Relations in $\pregfoam_d$. We omit the objects labelling the regions of each diagram: this avoids clutter and emphasizes that relations are independent of the ambient object. If no shading is given, the relation holds for all shadings. In the case of the braid-like and pitchfork relations, colours should be so that the crossings exist.}
  \label{fig:rel_diagfoam_graded}
\end{figure}

\subsection{Generic derivations and actions}
\label{subsec:generic_derivation_action}

In this subsection, we define derivations on the graded-2-category $\pregfoam_d$ of graded $\glt$-foams generically, depending on a family of parameters. We then give minimal conditions so that these derivations gather into an action of $\grsl_{2}^{\leq}$ by derivation on $\pregfoam_d$.
We do the same analysis when restricting to the odd case $\presfoam_d$, extending to an action of $\gloo$.

\subsubsection{Graded case}
\label{subsubsec:generic_derivation_action_graded}

\begin{lemma}
  \label{lem:generic_derivations_on_graded_foam}
  Let $\varf\coloneqq\{\varf^i\}_{1\leq i\leq d-1}$, $\delh$ and $\varh\coloneqq\{\varh^i\}_{1\leq i\leq d-1}$ be scalars in $\ringfoam$.
  The graded-2-category $\pregfoam_d$ admits the following graded derivations $\sff_{\varf}$ and $\sfh_{\delh,\varh}$, of degree $(1,1)$ and $(0,0)$ respectively, and defined on the generators (\cref{fig:string_generators_graded_foams}) as zero on crossings and as:
  \begin{center}
    \def\spc{2ex}
    \begin{tabular}{@{}l@{\hskip 5ex}*{4}{r@{\hskip 4ex}}r@{}}
      &
      $\tikzpic{\fdot\node at (.2,-.2) {\scriptsize $i$};}[scale=.5]$
      &
      $\tikzpic{\fcup\node at (1,.4) {\scriptsize $i$};}[scale=.5]$
      &
      $\tikzpic{\fzip\node at (1,.4) {\scriptsize $i$};}[scale=.5]$
      &
      $\tikzpic{\fcap\node at (1,.6) {\scriptsize $i$};}[scale=.5]$
      &
      $\tikzpic{\funzip\node at (1,.6) {\scriptsize $i$};}[scale=.5]$
      \\*[2ex]
      \midrule
      %%%%%%%%%%%%%%%%%%%%%%%%%%%%%%
      $\sff_{\varf}$
      &
      $0$
      &
      $\varf^i\;\tikzpic{\fcup\fdot[.5][.8]\node at (1,.4) {\scriptsize $i$};}[scale=.5]$
      &
      $\varf^i\;\tikzpic{\fzip\fdot[-.3][.7]\node at (1,.4) {\scriptsize $i$};}[scale=.5]$
      &
      $-\varf^iXZ\;\tikzpic{\fcap\fdot[.5][.2]\node at (1,.6) {\scriptsize $i$};}[scale=.5]$
      &
      $-\varf^iYZ\;\tikzpic{\funzip\fdot[-.3][.3]\node at (1,.6) {\scriptsize $i$};}[scale=.5]$
      \\*[\spc]
      %%%%%%%%%%%%%%%%%%%%%%%%%%%%%%
      % $\lieh_1$
      % &
      % $-\;\tikzpic{\fdot\node at (.2,-.2) {\scriptsize $i$};}[scale=.5]$
      % &
      % $\tikzpic{\fcup\node at (1,.4) {\scriptsize $i$};}[scale=.5]$
      % &
      % $0$
      % &
      % $0$
      % &
      % $-\;\tikzpic{\funzip\node at (1,.6) {\scriptsize $i$};}[scale=.5]$
      % \\*[\spc]
      % %%%%%%%%%%%%%%%%%%%%%%%%%%%%%%
      % $\lieh_2$
      % &
      % $\tikzpic{\fdot\node at (.2,-.2) {\scriptsize $i$};}[scale=.5]$
      % &
      % $0$
      % &
      % $\tikzpic{\fzip\node at (1,.4) {\scriptsize $i$};}[scale=.5]$
      % &
      % $-\;\tikzpic{\fcap\node at (1,.6) {\scriptsize $i$};}[scale=.5]$
      % &
      % $0$
      % \\*[\spc]
      % \cmidrule(r){2-6}
      % %%%%%%%%%%%%%%%%%%%%%%%%%%%%%%
      $\sfh_{\delh,\varh}$
      &
      $-\delh\;\tikzpic{\fdot\node at (.2,-.2) {\scriptsize $i$};}[scale=.5]$
      &
      $(\delh-\varh^i)\;\tikzpic{\fcup\node at (1,.4) {\scriptsize $i$};}[scale=.5]$
      &
      $-\varh^i\;\tikzpic{\fzip\node at (1,.4) {\scriptsize $i$};}[scale=.5]$
      &
      $\varh^i\;\tikzpic{\fcap\node at (1,.6) {\scriptsize $i$};}[scale=.5]$
      &
      $-(\delh-\varh^i)\;\tikzpic{\funzip\node at (1,.6) {\scriptsize $i$};}[scale=.5]$
    \end{tabular}
  \end{center}
\end{lemma}

\begin{proof}
  We show that $\sff_{\varf}$ is well-defined.
  Thanks to \cref{rem:derivation_uniquely_defn_on_generators_of_A}, it suffices to check it locally on the defining relations (\cref{fig:rel_diagfoam_graded}).
  It is straightforward for braid-like relations, pitchfork relations, dot annihilation, dot migration, dot slide, and evaluation of dotted bubbles.
  For the other evaluations, we have:
  \def\scl{.5}
  \begin{IEEEeqnarray*}{rCl}
    %%%%% COUNTER-CLOCKWISE BUBBLES %%%%%
    \sff_{\varf}\left(
      \tikzpic{
        \fcap[0][1]\fcup\node at (.8+.5,-.3+1) {\scriptsize $i$};
      }[scale=\scl]
    \right)
    &=&
    -\varf^iXZ \;
    \tikzpic{
        \fdot[.5][1]\node at (.2+.5,-.2+1) {\scriptsize $i$};
        \fcap[0][1]\fcup\node at (.8+.5,-.3+1) {\scriptsize $i$};
      }[scale=\scl]
    +\bilfoam\big((1,1),(-1,0)\big)
    \varf^i\;
    \tikzpic{
        \fdot[.5][1]\node at (.2+.5,-.2+1) {\scriptsize $i$};
        \fcap[0][1]\fcup\node at (.8+.5,-.3+1) {\scriptsize $i$};
      }[scale=\scl]
    =
    0
    \\[2ex]
    %%%%% CLOCKWISE BUBBLES %%%%%
    \sff_{\varf}\left(
      \tikzpic{
        \funzip[0][1]\fzip\node at (.8+.5,-.3+1) {\scriptsize $i$};
      }[scale=\scl]
    \right)
    &=&
    -\varf^iYZ
    \mspace{10mu}
    \tikzpic{
        \fdot[-.5][1]\node at (.2-.5,-.2+1) {\scriptsize $i$};
        \funzip[0][1]\fzip\node at (.8+.5,-.3+1) {\scriptsize $i$};
      }[scale=\scl]
    +\bilfoam\big((1,1),(0,1)\big)
    \varf^i
    \mspace{10mu}
    \tikzpic{
        \fdot[-.5][1]\node at (.2-.5,-.2+1) {\scriptsize $i$};
        \funzip[0][1]\fzip\node at (.8+.5,-.3+1) {\scriptsize $i$};
      }[scale=\scl]
    =0
  \end{IEEEeqnarray*}
  The neck-cutting gives:
  \begin{gather*}
    \sff_{\varf}\left(
      \tikzpic{
        \fdot[.5][1.8]\fcup[0][1]\fcap
        \node[left=-3pt] at (0,2) {\scriptsize $i$};
      }[scale=\scl][(0,1*\scl)]
      \;+\;
      \tikzpic{
        \fdot[.5][.2]\fcup[0][1]\fcap
        \node[left=-3pt] at (0,2) {\scriptsize $i$};
      }[scale=\scl][(0,1*\scl)]
    \right)
    =
    \Big[-\varf^iXZ \bilfoam\big((1,1),(1,0)\big)+\varf^i\Big]\;
    \tikzpic{
        \fdot[.5][.2]
        \fdot[.5][1.8]\fcup[0][1]\fcap
        \node[left=-3pt] at (0,2) {\scriptsize $i$};
      }[scale=\scl][(0,1*\scl)]
    =0
  \end{gather*}
  Finally, the squeezing relation gives:
  \begin{IEEEeqnarray*}{rCl}
    \sff_{\varf}\left(
      \tikzpic{
        \fzip[0][3][2]
        \funzip[0][0][2]
        \begin{scope}[scale=2]
          \fcup[-.25][1]
          \fcap[-.25][0]
        \end{scope}
        \node[below=-2pt] at (-.5,0) {\scriptsize $i$};
        \node[below=-2pt] at (1,0) {\scriptsize $i+1$};
      }[scale=.35][(0,2*.35)]
    \right)
    &=&
    \varf^{i+1}
    \tikzpic{
      \fzip[0][3][2]
      \funzip[0][0][2]
      \begin{scope}[scale=2]
        \fcup[-.25][1]
        \fcap[-.25][0]
      \end{scope}
      \node[below=-2pt] at (-.5,0) {\scriptsize $i$};
      \node[below=-2pt] at (1,0) {\scriptsize $i+1$};
      \fdot[-.2][3.7][2][1.5]
    }[scale=.35][(0,2*.35)]
    +
    \varf^i
    \bilfoam\big((1,1),(1,0)\big)
    \tikzpic{
      \fzip[0][3][2]
      \funzip[0][0][2]
      \begin{scope}[scale=2]
        \fcup[-.25][1]
        \fcap[-.25][0]
      \end{scope}
      \node[below=-2pt] at (-.5,0) {\scriptsize $i$};
      \node[below=-2pt] at (1,0) {\scriptsize $i+1$};
      \fdot[1-.5][3][1][1.5]
    }[scale=.35][(0,2*.35)]
    \\
    &&{}-\varf^iXZ
    \bilfoam\big((1,1),(1,-1)\big)
    \tikzpic{
      \fzip[0][3][2]
      \funzip[0][0][2]
      \begin{scope}[scale=2]
        \fcup[-.25][1]
        \fcap[-.25][0]
      \end{scope}
      \node[below=-2pt] at (-.5,0) {\scriptsize $i$};
      \node[below=-2pt] at (1,0) {\scriptsize $i+1$};
      \fdot[1-.5][1][1][1.5]
    }[scale=.35][(0,2*.35)]
    -\varf^{i+1}YZ
    \bilfoam\big((1,1),(0,-1)\big)
    \tikzpic{
      \fzip[0][3][2]
      \funzip[0][0][2]
      \begin{scope}[scale=2]
        \fcup[-.25][1]
        \fcap[-.25][0]
      \end{scope}
      \node[below=-2pt] at (-.5,0) {\scriptsize $i$};
      \node[below=-2pt] at (1,0) {\scriptsize $i+1$};
      \fdot[-.2][.3][2][1.5]
    }[scale=.35][(0,2*.35)]
    \\[1ex]
    &=&
    \Big[\varf^{i+1}\normafoam{+}{-}\varf^i\normafoam{-}{+}\varf^i-\varf^{i+1}\Big]
    \tikzpic{
      \fzip[0][3][2]
      \funzip[0][0][2]
      \begin{scope}[scale=2]
        \fcup[-.25][1]
        \fcap[-.25][0]
      \end{scope}
      \node[below=-2pt] at (-.5,0) {\scriptsize $i$};
      \node[below=-2pt] at (1,0) {\scriptsize $i+1$};
      \fdot[-.2][3.7][2][1.5]
    }[scale=.35][(0,2*.35)]
    =0
  \end{IEEEeqnarray*}

  We show that $\sfh_{\delh,\varh}$ is well-defined.
  Given that $\sfh_{\delh,\varh}$ has trivial grading and it acts on each generator by multiplication with a certain scalar, its action on a generic diagram amounts to multiplying this diagram with the sum of the scalars associated to each of its generators.
  With this remark, braid-like relations, pitchfork relations, dot annihilation, dot slide and evaluation of undotted bubbles are straightforward, and do not depend on the choice of scalars.
  Zigzag relations force the scalars associated to the cup and unzip (resp.\ the cap and zip) to be opposite of one another.
  Dot migration forces the scalar associated to the dot to be independent of $i$.
  Neck-cutting imposes a linear relation between the scalars associated to the dot, the cup and the cap.
  All the conditions above lead to the choice of scalars given in the lemma.
  One check compatibility with the remaining relations similarly (squeezing, evaluation of dotted bubbles and evaluation of shaded disks).

  This concludes.
\end{proof}

\begin{remark}[unicity of $\sff$ and $\sfh$]
  \label{rem:unicity_generic_graded}
  Recall the front-back composition from \cref{rem:monoidal-2-categorical-structure}.
  It is natural to ask for derivations to satisfy a Leibniz rule with respect to this composition as well. If so, then each derivation of degree $(1,1)$ is of the form $\sff_{\varf}$, where moreover all variables $\varf^i$ are equal.
  % in principle, the action of f could have "far-away" dots
  Similarly, in this case each derivation of degree $(0,0)$ is of the form $\sfh_{\delh,\varh}$, where moreover all variables $\varh^i$ are equal.
\end{remark}

\begin{lemma}
  \label{lem:commutator_graded_case}
  The commutators of the derivations $\sff_{\varf}$ and $\sfh_{\delh,\varh}$ defined in \cref{lem:generic_derivations_on_graded_foam} are
  \[[\sfh_{\delh,\varh},\sff_{\varf}]=-\delh\sff_{\varf}
  \quad\an\quad
  [\sff_{\varf},\sff_{\varf}]=[\sfh_{\delh,\varh},\sfh_{\delh',\varh'}]=0\]
  for any choice of (family of) parameters $\varf$, $(\delh,\varh)$ and $(\delh',\varh')$.
\end{lemma}

\begin{proof}
  Thanks to \cref{rem:derivation_uniquely_defn_on_generators_of_A}, it suffices to check the equalities on generators.
  Checking the claimed equalities amounts to straightforward computation.
  We give another argument for the relation $[\sfh_{\delh,\varh},\sff_{\varf}]=-\delh\sff_{\varf}$.
  Recall from the previous proof that $\sfh_{\delh,\varh}$ acts by multiplying a diagram by the sum of scalars associated to its generators; in particular, for any generator $D$, $\sfh_{\delh,\varh}$ acts by a certain scalar $\lambda_D$.
  On the other hand, the action of $\sff_{\varf}$ on the generator $D$ ``adds a dot'', up to scalar. It follows that 
  \[[\sfh_{\delh,\varh},\sff_{\varf}](D)=\sfh_{\delh,\varh}\sff_{\varf}(D)-\sff_{\varf}\sfh_{\delh,\varh}(D)
  =(\lambda_D-\delh)\sff(D)-\lambda_D\sff(D)=-\delh\sff(D).\]
  This concludes.
\end{proof}

With the help of \cref{rem:action_uniquely_defn_by_generators_of_g}, it follows that:

\begin{corollary}
  \label{cor:generic_action_on_graded_foam}
  For any choice of parameters $\varf$, $\varh$ and $\varh'$ as in \cref{lem:generic_derivations_on_graded_foam},
  The application
  \[\{\lief\mapsto \sff_{\varf},\lieh_1\mapsto\sfh_{1,\varh},\lieh_2\mapsto\sfh_{-1,\varh'}\}\]
  % (resp.\ the application $\{\liee\mapsto \sff_{\varf},\lieh_1\mapsto\sfh_{-1,\varh},\lieh_2\mapsto\sfh_{1,\varh'}\}$)
  defines an action of $\grgl_{2}^{\leq}$
  % (resp.\ $\grgl_{2}^{\geq}$)
  by derivation on the graded-2-category $\pregfoam_d$.
  \hfill\qed
\end{corollary}

We pick a standard choice of action:

\begin{definition}
  \label{defn:action_on_foam_graded_case}
  We view $\pregfoam_d$ as a $\grgl_{2}^{\leq}$-2-category with the action of $\grgl_{2}^{\leq}$ by derivation given in \cref{defn:action_on_foam_super_case} (ignoring the action of $\liee$).
\end{definition}

\subsubsection{Super case}
\label{subsubsec:generic_derivation_action_super}

\begin{lemma}
  \label{lem:generic_derivation_on_super_foam}
  Let $\vare\in\ringfoam$ be a choice of parameter.
  The super-2-category $\presfoam_d$ admits the derivation $\sfe_{\vare}$, defined on the generators (\cref{fig:string_generators_graded_foams}) as zero on crossings and as:
  \begin{center}
    \def\spc{2ex}
    \begin{tabular}{@{}l@{\hskip 5ex}*{4}{r@{\hskip 4ex}}r@{}}
      &
      $\tikzpic{\fdot\node at (.2,-.2) {\scriptsize $i$};}[scale=.5]$
      &
      $\tikzpic{\fcup\node at (1,.4) {\scriptsize $i$};}[scale=.5]$
      &
      $\tikzpic{\fzip\node at (1,.4) {\scriptsize $i$};}[scale=.5]$
      &
      $\tikzpic{\fcap\node at (1,.6) {\scriptsize $i$};}[scale=.5]$
      &
      $\tikzpic{\funzip\node at (1,.6) {\scriptsize $i$};}[scale=.5]$
      \\*[2ex]
      \midrule
      %%%%%%%%%%%%%%%%%%%%%%%%%%%%%%
      $\sfe_{\vare}$
      &
      $\normafoam{\vare}{(-1)^i\vare}\id_\emptyset$
      &
      $0$
      &
      $0$
      &
      $0$
      &
      $0$
    \end{tabular}
  \end{center}
\end{lemma}

\begin{proof}
  \def\scl{.5}
  It suffices to check that $\sfe_{\vare}$ is compatible with the defining relations (\cref{fig:rel_diagfoam_graded}).
  Relations that do not involve dots are straightforward.
  Compatibility with dot annihilation and neck-cutting relation essentially follows from the fact that $\sfe_{\vare}$ is a super derivation:
  \begin{gather*}
    \sfe_{\vare}\left(\;
    \tikzpic{
      \fdot[0][1]\fdot\node at (.2,-.2) {\scriptsize $i$};
    }[scale=\scl]
    \right)
    =
    \vare[1-1]\;\tikzpic{\fdot}
    =0
    \mspace{10mu}\an\mspace{10mu}
    \sfe_{\vare}\left(
      \tikzpic{
        \fdot[.5][1.8]\fcup[0][1]\fcap
        \node[left=-3pt] at (0,2) {\scriptsize $i$};
      }[scale=\scl][(.5,1*\scl)]
      \;+\;
      \tikzpic{
        \fdot[.5][.2]\fcup[0][1]\fcap
        \node[left=-3pt] at (0,2) {\scriptsize $i$};
      }[scale=\scl][(.5,1*\scl)]
    \right)
    =
    \vare[1-1]\;
    \tikzpic{
        \fcup[0][1]\fcap
        \node[left=-3pt] at (0,2) {\scriptsize $i$};
      }[scale=\scl][(.5,1*\scl)]
    =0.
  \end{gather*}
  This explains why $\sfe_{\vare}$ can only be defined in the super case.
  Compatibility with dot migration and evaluation of shaded disks follows from the fact that $\vare$ does not depend on $i$.
  % \begin{gather*}
  %   \sfe\left(\;
  %   \tikzpic{
  %     \fdot\node at (.2,-.2) {\scriptsize $i$};
  %   }
  %   % \mspace{10mu}
  %   -
  %   \mspace{10mu}
  %   \tikzpic{
  %     \fdot[0][0][2]\node at (.5,-.3) {\scriptsize $i+1$};
  %   }
  %   \right)
  %   =0
  % \end{gather*}
  Compatibility with dot slide and evaluation of bubbles is straightforward. This concludes.
\end{proof}

\begin{lemma}
  \label{lem:commutator_super_case}
  The (super) commutators of the super derivation $\sfe_{\vare}$ (\cref{lem:generic_derivation_on_super_foam}) with the (super) derivations $\sff_{\varf}$ and $\sfh_{\delh,\varh}$ (\cref{lem:generic_derivations_on_graded_foam}; restricted to the super case) are
  \begin{gather*}
    [\sfe_{\vare},\sff_{\varf}] = \sfh_{0,(-\normafoam{}{(-1)^i}\vare\varf^i)_i},
    \qquad
    [\sfh_{\delh,\varh},\sfe_{\vare}]=\delh\sfe_{\vare}
    \qquad\an\qquad
    [\sfe_{\vare},\sfe_{\vare}]=0,
  \end{gather*}
  for any choice of (family of) parameters $\varf$, $(\delh,\varh)$ and $\vare$.
\end{lemma}

\begin{proof}
  Thanks to \cref{rem:derivation_uniquely_defn_on_generators_of_A}, it suffices to check the equalities on generators.
  Checking $[\sfe_{\vare},\sfe_{\vare}]=0$ is straightforward, and the case of the commutator $[\sfh_{\delh,\varh},\sfe_{\vare}]$ follows from the equality
  \[[\sfh_{\delh,\varh},\sfe_{\vare}]=-\sfe_{\vare}\sfh_{\delh,\varh}.\]
  The equality $[\sfe_{\vare},\sff_{\varf}](\tikzpic{\fdot})=0$ is straightforward. For $D$ one of the remaining generators, we have $[\sfe_{\vare},\sff_{\varf}](D) = \sfe_{\vare}\sff_{\varf}(D)$, leading to the remaining equality.
\end{proof}

\begin{corollary}
  \label{cor:generic_action_on_super_foam}
  Let $\varf$, $\varh$, $\varh'$ and $\vare$ be choice of (family of) scalars in $\ringfoam$ as in \cref{lem:generic_derivations_on_graded_foam,lem:generic_derivation_on_super_foam}.
  If
  \[\varh^i+\varh^{'i} = \normafoam{-\vare\varf^i}{-(-1)^i\vare\varf^i}\quad\text{ for all }1\leq i\leq d-1,\]
  then the application
  \[\{\lief\mapsto \sff_{\varf},\lieh_1\mapsto\sfh_{1,\varh},\lieh_2\mapsto\sfh_{-1,\varh'},\liee\mapsto \sfe_{\vare}\}\]
  defines an action of $\gloo$ by derivation on $\presfoam_d$.
\end{corollary}

\begin{proof}
  This follows from \cref{cor:generic_action_on_graded_foam} and \cref{lem:commutator_super_case} (with the help of \cref{rem:action_uniquely_defn_by_generators_of_g}), using that
  $\sfh_{1,\varh}+\sfh_{-1,\varh'}=\sfh_{0,\varh+\varh'}$.
\end{proof}

We pick a standard choice of action, corresponding to the choice $\varf^i=\vare=1$, $\varh^i=0$ and $\varh^{'i}=-1$:

\begin{definition}
  \label{defn:action_on_foam_super_case}
  We view $\presfoam_d$ as a $\gloo$-2-category with the action of $\gloo$ by derivation given by:
  \begin{gather}
    \label{eq:simplest_derivation}
    \def\spc{2ex}
    \def\scl{.5}
    \begin{tabular}{@{}l@{\hskip 4ex}r*{4}{@{\hskip 3ex}r}@{}}
      &
      $\tikzpic{\fdot\node at (.2,-.2) {\scriptsize $i$};}[scale=\scl]$
      &
      $\tikzpic{\fcup\node at (1,.4) {\scriptsize $i$};}[scale=\scl]$
      &
      $\tikzpic{\fzip\node at (1,.4) {\scriptsize $i$};}[scale=\scl]$
      &
      $\tikzpic{\fcap\node at (1,.6) {\scriptsize $i$};}[scale=\scl]$
      &
      $\tikzpic{\funzip\node at (1,.6) {\scriptsize $i$};}[scale=\scl]$
      \\*[2ex]
      \midrule
      %%%%%%%%%%%%%%%%%%%%%%%%%%%%%%
      $\lief$
      &
      $0$
      &
      $\normafoam{}{}\tikzpic{\fcup\fdot[.5][.8]\node at (1,.4) {\scriptsize $i$};}[scale=\scl]$
      &
      $\normafoam{}{}\tikzpic{\fzip\fdot[-.3][.7]\node at (1,.4) {\scriptsize $i$};}[scale=\scl]$
      &
      $\normafoam{-\;}{?}\tikzpic{\fcap\fdot[.5][.2]\node at (1,.6) {\scriptsize $i$};}[scale=\scl]$
      &
      $\normafoam{}{}\tikzpic{\funzip\fdot[-.3][.3]\node at (1,.6) {\scriptsize $i$};}[scale=\scl]$
      \\*[\spc]
      %%%%%%%%%%%%%%%%%%%%%%%%%%%%%%
      $\liee$
      &
      $\normafoam{}{(-1)^i}\id_\emptyset$
      &
      $0$
      &
      $0$
      &
      $0$
      &
      $0$
      \\*[\spc]
      %%%%%%%%%%%%%%%%%%%%%%%%%%%%%%
      $\lieh_1$
      &
      $-\;\tikzpic{\fdot\node at (.2,-.2) {\scriptsize $i$};}[scale=\scl]$
      &
      $\;\tikzpic{\fcup\node at (1,.4) {\scriptsize $i$};}[scale=\scl]$
      &
      $0$
      &
      $0$
      &
      $-\;\tikzpic{\funzip\node at (1,.6) {\scriptsize $i$};}[scale=\scl]$
      \\*[\spc]
      %%%%%%%%%%%%%%%%%%%%%%%%%%%%%%
      $\lieh_2$
      &
      $\tikzpic{\fdot\node at (.2,-.2) {\scriptsize $i$};}[scale=\scl]$
      &
      $0$
      &
      $\;\tikzpic{\fzip\node at (1,.4) {\scriptsize $i$};}[scale=\scl]$
      &
      $-\;\tikzpic{\fcap\node at (1,.6) {\scriptsize $i$};}[scale=\scl]$
      &
      $0$
      % \\*[\spc]
      % \cmidrule(r){2-6}
      %%%%%%%%%%%%%%%%%%%%%%%%%%%%%%
      % $\lieh\coloneqq\lieh_1+\lieh_2$
      % $\lieh$
      % &
      % $0$
      % &
      % $-\;\tikzpic{\fcup\node at (1,.4) {\scriptsize $i$};}[scale=\scl]$
      % &
      % $-\;\tikzpic{\fzip\node at (1,.4) {\scriptsize $i$};}[scale=\scl]$
      % &
      % $\;\tikzpic{\fcap\node at (1,.6) {\scriptsize $i$};}[scale=\scl]$
      % &
      % $\;\tikzpic{\funzip\node at (1,.6) {\scriptsize $i$};}[scale=\scl]$
    \end{tabular}
  \end{gather}
  Note that in term of the $\bZ^2$-grading $(a,b)$, we have $h_1(D)=-b D$, $h_2(D)=a D$ and $h(D)=(a-b) D$.
\end{definition}

\begin{remark}
  \label{rem:gloo_is_almost_unique}
  Under certain reasonable assumptions, the $\gloo$-action is almost unique.
  Arguing as in \cref{rem:unicity_generic_graded}, it is reasonable to assume that each family of scalars is independent of $i$. We then view $\varf$, $\varh$ and $\varh'$ as three scalars.
  Any graded derivation on $\pregfoam_d$ of degree $(-1,-1)$ is of the form $\sfe$.
  following \cref{rem:unicity_generic_graded} and \cref{lem:commutator_super_case}, under this assumption any $\gloo$-action by derivation arises as in \cref{cor:generic_action_on_super_foam}.
  Assuming further that $\varf$ and $\vare$ are invertible, one can renormalize the action of $\lief$ and $\liee$, leaving only one parameter $\varh$, having necessarily $\varh'=-1-\varh$.
\end{remark}

\begin{example}
  Another choice compatible with the assumptions of \cref{rem:gloo_is_almost_unique} is $\varf^i=-1$, $\vare=1$, $\varh^i=\varh^{'i}=\frac{1}{2}$, assuming that $2$ is invertible in the ground ring. This gives:
  \begin{center}
    \def\spc{2ex}
    \def\scl{.5}
    \begin{tabular}{@{}l@{\hskip 4ex}r*{4}{@{\hskip 3ex}r}@{}}
      &
      $\tikzpic{\fdot\node at (.2,-.2) {\scriptsize $i$};}[scale=\scl]$
      &
      $\tikzpic{\fcup\node at (1,.4) {\scriptsize $i$};}[scale=\scl]$
      &
      $\tikzpic{\fzip\node at (1,.4) {\scriptsize $i$};}[scale=\scl]$
      &
      $\tikzpic{\fcap\node at (1,.6) {\scriptsize $i$};}[scale=\scl]$
      &
      $\tikzpic{\funzip\node at (1,.6) {\scriptsize $i$};}[scale=\scl]$
      \\*[2ex]
      \midrule
      %%%%%%%%%%%%%%%%%%%%%%%%%%%%%%
      $\lief$
      &
      $0$
      &
      $\normafoam{-\;}{}\tikzpic{\fcup\fdot[.5][.8]\node at (1,.4) {\scriptsize $i$};}[scale=\scl]$
      &
      $\normafoam{-\;}{}\tikzpic{\fzip\fdot[-.3][.7]\node at (1,.4) {\scriptsize $i$};}[scale=\scl]$
      &
      $\normafoam{}{?}\tikzpic{\fcap\fdot[.5][.2]\node at (1,.6) {\scriptsize $i$};}[scale=\scl]$
      &
      $\normafoam{-\;}{}\tikzpic{\funzip\fdot[-.3][.3]\node at (1,.6) {\scriptsize $i$};}[scale=\scl]$
      \\*[\spc]
      %%%%%%%%%%%%%%%%%%%%%%%%%%%%%%
      $\liee$
      &
      $\normafoam{}{(-1)^i}\id_\emptyset$
      &
      $0$
      &
      $0$
      &
      $0$
      &
      $0$
      \\*[\spc]
      %%%%%%%%%%%%%%%%%%%%%%%%%%%%%%
      $\lieh_1$
      &
      $-\;\tikzpic{\fdot\node at (.2,-.2) {\scriptsize $i$};}[scale=\scl]$
      &
      $\frac{1}{2}\;\tikzpic{\fcup\node at (1,.4) {\scriptsize $i$};}[scale=\scl]$
      &
      $-\frac{1}{2}\;\tikzpic{\fzip\node at (1,.4) {\scriptsize $i$};}[scale=\scl]$
      &
      $\frac{1}{2}\;\tikzpic{\fcap\node at (1,.6) {\scriptsize $i$};}[scale=\scl]$
      &
      $-\frac{1}{2}\;\tikzpic{\funzip\node at (1,.6) {\scriptsize $i$};}[scale=\scl]$
      \\*[\spc]
      %%%%%%%%%%%%%%%%%%%%%%%%%%%%%%
      $\lieh_2$
      &
      $\tikzpic{\fdot\node at (.2,-.2) {\scriptsize $i$};}[scale=\scl]$
      &
      $(-1-\frac{1}{2})\;\tikzpic{\fcup\node at (1,.4) {\scriptsize $i$};}[scale=\scl]$
      &
      $-\frac{1}{2}\;\tikzpic{\fzip\node at (1,.4) {\scriptsize $i$};}[scale=\scl]$
      &
      $\frac{1}{2}\;\tikzpic{\fcap\node at (1,.6) {\scriptsize $i$};}[scale=\scl]$
      &
      $-(-1-\frac{1}{2})\;\tikzpic{\funzip\node at (1,.6) {\scriptsize $i$};}[scale=\scl]$
      % \\*[\spc]
      % \cmidrule(r){2-6}
      %%%%%%%%%%%%%%%%%%%%%%%%%%%%%%
      % $\lieh\coloneqq\lieh_1+\lieh_2$
      % $\lieh$
      % &
      % $0$
      % &
      % $-\;\tikzpic{\fcup\node at (1,.4) {\scriptsize $i$};}[scale=\scl]$
      % &
      % $-\;\tikzpic{\fzip\node at (1,.4) {\scriptsize $i$};}[scale=\scl]$
      % &
      % $\;\tikzpic{\fcap\node at (1,.6) {\scriptsize $i$};}[scale=\scl]$
      % &
      % $\;\tikzpic{\funzip\node at (1,.6) {\scriptsize $i$};}[scale=\scl]$
    \end{tabular}
  \end{center}
  % Note that in term of the $\bZ^2$-grading $(a,b)$, we have $h_1(D)=-\frac{1}{2}(a+b) D$, $h_2(D)=\frac{1}{2}(3b-a) D$ and $h(D)=(b-a) D$.
\end{example}

\begin{example}
  A choice which is not compatible with the assumptions of \cref{rem:gloo_is_almost_unique} is $\vare=1$, $\varf^i=0$, $\varh^i=\frac{1}{2}$ and $\varh^{'i}=-\frac{1}{2}$, assuming again that $2$ is invertible in the ground ring. This gives:
  \begin{center}
    \def\spc{2ex}
    \def\scl{.5}
    \begin{tabular}{@{}l@{\hskip 4ex}r*{4}{@{\hskip 3ex}r}@{}}
      &
      $\tikzpic{\fdot\node at (.2,-.2) {\scriptsize $i$};}[scale=\scl]$
      &
      $\tikzpic{\fcup\node at (1,.4) {\scriptsize $i$};}[scale=\scl]$
      &
      $\tikzpic{\fzip\node at (1,.4) {\scriptsize $i$};}[scale=\scl]$
      &
      $\tikzpic{\fcap\node at (1,.6) {\scriptsize $i$};}[scale=\scl]$
      &
      $\tikzpic{\funzip\node at (1,.6) {\scriptsize $i$};}[scale=\scl]$
      \\*[2ex]
      \midrule
      %%%%%%%%%%%%%%%%%%%%%%%%%%%%%%
      $\lief$
      &
      $0$
      &
      $0$
      &
      $0$
      &
      $0$
      &
      $0$
      \\*[\spc]
      %%%%%%%%%%%%%%%%%%%%%%%%%%%%%%
      $\liee$
      &
      $\normafoam{}{(-1)^i}\id_\emptyset$
      &
      $0$
      &
      $0$
      &
      $0$
      &
      $0$
      \\*[\spc]
      %%%%%%%%%%%%%%%%%%%%%%%%%%%%%%
      $\lieh_1$
      &
      $-\;\tikzpic{\fdot\node at (.2,-.2) {\scriptsize $i$};}[scale=\scl]$
      &
      $\frac{1}{2}\;\tikzpic{\fcup\node at (1,.4) {\scriptsize $i$};}[scale=\scl]$
      &
      $-\frac{1}{2}\;\tikzpic{\fzip\node at (1,.4) {\scriptsize $i$};}[scale=\scl]$
      &
      $\frac{1}{2}\;\tikzpic{\fcap\node at (1,.6) {\scriptsize $i$};}[scale=\scl]$
      &
      $-\frac{1}{2}\;\tikzpic{\funzip\node at (1,.6) {\scriptsize $i$};}[scale=\scl]$
      \\*[\spc]
      %%%%%%%%%%%%%%%%%%%%%%%%%%%%%%
      $\lieh_2$
      &
      $\tikzpic{\fdot\node at (.2,-.2) {\scriptsize $i$};}[scale=\scl]$
      &
      $-\frac{1}{2}\;\tikzpic{\fcup\node at (1,.4) {\scriptsize $i$};}[scale=\scl]$
      &
      $\frac{1}{2}\;\tikzpic{\fzip\node at (1,.4) {\scriptsize $i$};}[scale=\scl]$
      &
      $-\frac{1}{2}\;\tikzpic{\fcap\node at (1,.6) {\scriptsize $i$};}[scale=\scl]$
      &
      $\frac{1}{2}\;\tikzpic{\funzip\node at (1,.6) {\scriptsize $i$};}[scale=\scl]$
      % \\*[\spc]
      % \cmidrule(r){2-6}
      %%%%%%%%%%%%%%%%%%%%%%%%%%%%%%
      % $\lieh\coloneqq\lieh_1+\lieh_2$
      % $\lieh$
      % &
      % $0$
      % &
      % $-\;\tikzpic{\fcup\node at (1,.4) {\scriptsize $i$};}[scale=\scl]$
      % &
      % $-\;\tikzpic{\fzip\node at (1,.4) {\scriptsize $i$};}[scale=\scl]$
      % &
      % $\;\tikzpic{\fcap\node at (1,.6) {\scriptsize $i$};}[scale=\scl]$
      % &
      % $\;\tikzpic{\funzip\node at (1,.6) {\scriptsize $i$};}[scale=\scl]$
    \end{tabular}
  \end{center}
  % Note that in term of the $\bZ^2$-grading $(a,b)$, we have $h_1(D)=-\frac{1}{2}(a+b) D$, $h_2(D)=\frac{1}{2}(a+b) D$ and $h(D)=0$.
\end{example}

\subsection{Twist on graded \texorpdfstring{$\glt$}{gl2}-foams}
\label{subsec:twist_foam}

In this subsection, we define webs with green markings and twists on graded $\glt$-foams, following the general framework of \cref{subsubsec:twisting_g2categories} and in analogy with \cite[section~5.1]{QRS+_SymmetriesEquivariantKhovanovRozansky_2023}.
We fix $\fg$ to be either $\fg=\grgl_2^\leq$ or $\fg=\gloo$.
Recall the notation $\fg\foam_d$ after \cref{defn:foam_and_sfoam}, denoting either $\gfoam_d$ or $\sfoam_d$.
Recall that we fixed a structure of $\fg$-2-category on $\fg\prefoam_d$ in \cref{defn:action_on_foam_graded_case} and in \cref{defn:action_on_foam_super_case}, which extends to $\fg\foam_d$.

Below is an example of a web with markings:
\begin{equation*}
  \tikzpic{
    \webid[0][1.5][1]\webs[1][1.5]\webmMarkedT[2][1.5][$(0,3,-3)$]\webid[3][1.5][1]
    \webm\webid[1][0][1]\webid[2][0][1]\websMarkedB[3][0][$(2,-1,-1)$]
  }[xscale=.5,yscale=.4][(0,1.25*.4)]
  =
  \tikzpic{
    \webid[0][1.5][1]\webs[1][1.5]\webm[2][1.5]\webid[3][1.5][1]
    \webm\webid[1][0][1]\webid[2][0][1]\websMarkedB[3][0][$2$]
  }[xscale=.5,yscale=.4][(0,1.25*.4)]
  \;\langle 2,-4\rangle
  \qquad
  \epsilon_{W^{\greenmarking}}(h_1)=2,\quad\epsilon_{W^{\greenmarking}}(h_2)=-4.
\end{equation*}
More formally, a \emph{web with markings $W^{\greenmarking}$} (or \emph{marked web}) is the data of a web $W$ together with markings 
$\greenmarking[2]$ on its edges of width one, each equipped with a triple of scalars in $\ringfoam$, generically denoted $(\alpha,\beta_1,\beta_2)$.
% When $\fg=\gloo$, we further assume that
% \[\alpha=\beta_1+\beta_2.\]
% (This assumption is not necessary for the following few paragraphs, but as we shall see \cref{lem:family_of_twists_foam}, it is necessary for the rest of the paper.)
For a marked web $W^{\greenmarking}$ and $i\in \{1,2\}$, we set $\epsilon_{W^{\greenmarking}}(h_i)$ to be the sum of the $i+1$st entries of all the markings on $W$.
See the example above; the notation of the second web is explained in \cref{rem:notation_bracket_twist}.
If $W_s^{\greenmarking}$ and $W_t^{\greenmarking}$ are two marked webs with $W_s$ and $W_t$ as underlying webs respectively, then any foam $F\colon W_s\to W_t$ defines a foam $F^{{\greenmarking}}\colon W_s^{\greenmarking}\to W_t^{\greenmarking}$.
If $F$ has quantum grading $\qdeg F$, then
\[\qdeg F^{\greenmarking}\coloneqq \qdeg F - (\epsilon_{W_s^{\greenmarking}}(h_2)-\epsilon_{W_s^{\greenmarking}}(h_1)) + (\epsilon_{W_t^{\greenmarking}}(h_2)-\epsilon_{W_t^{\greenmarking}}(h_1)).\]
In other words, adding a twist $\greenmarking$ to a web $W$ shifts it by $q^{\epsilon_{W^{\greenmarking}}(h_2)-\epsilon_{W^{\greenmarking}}(h_1)}$.
Denote $\fg\foam_d^{\text{pre-}\greenmarking}$ the $\fg$-2-category consisting of marked webs and the same Hom-categories as $\fg\foam_d$, restricting to foams preserving the quantum grading. 

We now define a family of twists for the $\fg$-2-category $\gfoam_d^{\text{pre-}\greenmarking}$.

\begin{definition}
  \label{defn:twist_foam}
  Let $\alpha,\beta_1,\beta_2\in\ringfoam$ be three parameters.
  Defines:
  \begin{IEEEeqnarray*}{rClcrClcrCl}
    \lief\left(\;\tikzpic{
      \draw[web1] (-1,0) to (1,0);
      \node[green_mark] at (0,0){};
      \node[below] at (0,0){\scriptsize $(\alpha,\beta_1,\beta_2)$};
    }[scale=.7][(0,0)]\;\right)
    &=&
    \alpha\;
    \tikzpic{
      \draw (-1,0) rectangle (1,1);
      \fdot[0][.5]
    }[scale=.7]
    &\quad&
    \lieh_{i}\left(\;\tikzpic{
      \draw[web1] (-1,0) to (1,0);
      \node[green_mark] at (0,0){};
      \node[below] at (0,0){\scriptsize $(\alpha,\beta_1,\beta_2)$};
    }[scale=.7][(0,0)]\;\right)
    &=&
    \beta_{i}\;
    \tikzpic{
      \draw (-1,0) rectangle (1,1);
    }[scale=.7]
    &\quad&
    \liee\left(\;\tikzpic{
      \draw[web1] (-1,0) to (1,0);
      \node[green_mark] at (0,0){};
      \node[below] at (0,0){\scriptsize $(\alpha,\beta_1,\beta_2)$};
    }[scale=.7][(0,0)]\;\right)
    &=&
    0
    \\
    \tau_{\alpha,\beta_1,\beta_2}(\lief)
    &=&
    \alpha\;
    \tikzpic{
      \draw (-1,0) rectangle (1,1);
      \fdot[0][.5]
    }[scale=.7]
    &&
    \tau_{\alpha,\beta_1,\beta_2}(\lieh)
    &=&
    \beta_{i}\;
    \tikzpic{
      \draw (-1,0) rectangle (1,1);
    }[scale=.7]
    &&
    \tau_{\alpha,\beta_1,\beta_2}(\liee)
    &=&
    0
  \end{IEEEeqnarray*}
  Extending this definition by the Leibniz rule defines for each marked web $W^{\greenmarking}$ a degree-preserving linear map
  \[\tau_{W^{\greenmarking}}\colon\fg\to\End_{\fg\foam_d^{\text{pre-}\greenmarking}}(W^{\greenmarking}).\]
\end{definition}

Note that $\tau_{W^{\greenmarking}}(h_i)=\epsilon_{W^{\greenmarking}}(h_i)\id_W$.

\begin{remark}
  \label{rem:twist-foam-front-back-composition}
  In the definition above, ``extending by the Leibniz rule'' should be understood both with respect to the horizontal composition and with respect to the front-back composition (see \cref{rem:monoidal-2-categorical-structure}).
  Below we sometimes use 2-categorical statements, although we should really be using monoidal 2-categorical statements, and take the front-back composition into account.
  % We shall refer to this remark when encountering such caveat.
\end{remark}

\begin{remark}
  \label{rem:notation_bracket_twist}
  Only the action of $\lief$ depends on the position of the dot. For that reason, we shall use the notation
  \begin{gather*}
    \tikzpic{
      \draw[web1] (-1,0) to (1,0);
      \node[green_mark] (B) at (0,0){};
      \node[below={-3pt} of B] {\scriptsize $(\alpha,\beta_1,\beta_2)$};
    }[scale=.7][(0,0)]
    =
    \tikzpic{
      \draw[web1] (-1,0) to (1,0);
      \node[green_mark] (B) at (0,0){};
      \node[below={-3pt} of B] {\scriptsize $\alpha$};
    }[scale=.7][(0,0)]
    \langle \beta_1,\beta_2 \rangle.
  \end{gather*}
  In particular, marking a green dot with a single scalar $\alpha$ is a notation for marking it with the triple $(\alpha,0,0)$, and the notation $W\langle \beta_1,\beta_2 \rangle$ means ``the web $W$ with an additional marking $(0,\beta_1,\beta_2)$ anywhere''. See the example above.
\end{remark}

Note that $\tau_{W^{\greenmarking}}(f)$ is a sum over the identity foam with a single dot.
We write $\epsilon_{W^{\greenmarking}}(f)$ the sum over all the scalars in front of these dotted identities.

\begin{lemma}
  \label{lem:family_of_twists_foam}
  The family $\tau$ given in \cref{defn:twist_foam} is a family of twists in the sense of \cref{defn:family_of_twists}:
  \begin{enumerate}[(i)]
    \item in the graded case, for any $\grgl_2^{\leq}$-action defined in \cref{cor:generic_action_on_graded_foam};
    \item in the super case, for any $\gloo$-action defined in \cref{cor:generic_action_on_super_foam}, provided that $\epsilon_{W^{\greenmarking}}(f)=\epsilon_{W^{\greenmarking}}(h_1)+\epsilon_{W^{\greenmarking}}(h_2)$.
  \end{enumerate}
\end{lemma}

\begin{proof}
  It is clear that $\tau$ verifies the Leibniz rule and has a graded-commutative image.
  Thanks to \cref{rem:notation_bracket_twist}, we can redistribute twists with respect to $h_1$ and $h_2$, so that if the condition is verified, we may assume it is verified at the level of each twist.
  Following \cref{rem:twist_only_check_on_generators} (bearing \cref{rem:twist-foam-front-back-composition} in mind), it suffices to check flatness locally, that is, a single green marking $\omega=(\alpha,\beta_1,\beta_2)$.
  In the graded case, flatness holds in fact for any $\sff_{\varf}$ and $\sfh_{\delh,\varh}$ defined in \cref{lem:generic_derivations_on_graded_foam}, using \cref{lem:commutator_graded_case} (here we write $\epsilon(\sfh_{\delh,\varh})=\beta$ and $\epsilon(\sfh_{\delh,\varh}')=\beta'$):
  \begin{IEEEeqnarray*}{rCl}
    \IEEEeqnarraymulticol{3}{l}{
      \sfh_{\delh,\varh}\cdot \tau(\sff_{\varf})-\bilfoam(\sfh_{\delh,\varh},\sff_{\varf})\sff_{\varf}\cdot\tau(\sfh_{\delh,\varh})
    }
    \\
    \mspace{80mu}
    &=&
    \sfh_{\delh,\varh}(\alpha\;\tikzpic{\fdot}) - \sff_{\varf}(\beta\id)
    = -\delh\alpha\;\tikzpic{\fdot}
    = \tau(-\delh\sff_{\varf})
    = \tau([\sfh_{\delh,\varh},\sff_{\varf}])
    \\[1ex]
    %%%%%%%%%%%%%%%%%%%%%%%%%%%%%%
    \IEEEeqnarraymulticol{3}{l}{
      \sfh_{\delh,\varh}\cdot \tau(\sfh_{\delh',\varh'})-\bilfoam(\sfh_{\delh,\varh},\sfh_{\delh',\varh'})\sfh_{\delh',\varh'}\cdot\tau(\sfh_{\delh,\varh})
    }
    \\
    &=& \sfh_{\delh,\varh}(\beta')-\sfh_{\delh',\varh'}(\beta\id)=0
    = \tau([\sfh_{\delh,\varh},\sfh_{\delh',\varh'}])
    \\[1ex]
    %%%%%%%%%%%%%%%%%%%%%%%%%%%%%%
    \IEEEeqnarraymulticol{3}{l}{
      \sff_{\varf}\cdot \tau(\sff_{\varf'})-\bilfoam(\sff_{\varf},\sff_{\varf'})\sff_{\varf'}\cdot\tau(\sff_{\varf})
    }
    \\
    &=&
    \sff_{\varf}(\alpha'\;\tikzpic{\fdot})-XY\sff_{\varf'}(\alpha\;\tikzpic{\fdot})=0
    = \tau([\sff_{\varf},\sff_{\varf'}])
  \end{IEEEeqnarray*}
  We do a similar computation in the super case, using \cref{lem:commutator_super_case}:
  \begin{IEEEeqnarray*}{rCl}
    \IEEEeqnarraymulticol{3}{l}{
      \sfh_{\delh,\varh}\cdot \tau(\sfe_{\vare})-\bilfoam(\sfh_{\delh,\varh},\sfe_{\vare})\sfe_{\vare}\cdot\tau(\sfh_{\delh,\varh})
    }
    \\
    \mspace{80mu}
    &=&
    - \sfe_{\vare}(\beta\id) = 0 = \tau(\delh\sfe_{\vare})
    = \tau([\sfh_{\delh,\varh},\sfe_{\vare}])
    \\[1ex]
    %%%%%%%%%%%%%%%%%%%%%%%%%%%%%%
    \IEEEeqnarraymulticol{3}{l}{
      \sfe_{\vare}\cdot \tau(\sff_{\varf})-\bilfoam(\sfe_{\vare},\sff_{\varf})\sff_{\varf}\cdot\tau(\sfe_{\vare})
    }
    \\
    &=&
    \alpha\normafoam{\vare}{(-1)^i\vare}
    \overset{?}{=}
    \beta\id
    = \tau(\sfh_{0,(-\normafoam{}{(-1)^i}\vare\varf^i)_i})
    =
    \tau([\sfe_{\vare},\sff_{\varf}])
    \\[1ex]
    %%%%%%%%%%%%%%%%%%%%%%%%%%%%%%
    \IEEEeqnarraymulticol{3}{l}{
      \sfe_{\vare}\cdot \tau(\sfe_{\vare'})-\bilfoam(\sfe_{\vare},\sfe_{\vare'})\sfe_{\vare'}\cdot\tau(\sfe_{\vare})
    }
    \\
    &=& 0
    = \tau([\sfe_{\vare},\sfe_{\vare'}])
  \end{IEEEeqnarray*}
  The only equality that does not hold formally is the condition coming from the commutation between $\sff$ and $\sfe$, as it requires
  $\alpha
  =
  \beta$, where here $\beta=\beta_1+\beta_2$.
\end{proof}

\begin{definition}
  \label{defn:twisted_foams}
  We denote $\fg\markedfoam[d]\coloneqq\left(\fg\foam_d^{\text{pre-}\greenmarking}\right)^\tau$, where $\tau$ is the family of twists defined in \cref{defn:twist_foam}.
\end{definition}

We conclude this subsection by gathering some properties of twists.

\begin{lemma}[$\lieh$-equivariance]
  \label{lem:h-equivariance}
  Suppose that $G\in \Hom_{\markedgfoam[d]} (W_s^{\greenmarking}, W_t^{\greenmarking}) $ is homogeneous of degree $\deg(G)=(a,b)$. Then the following statements are true:
  \begin{enumerate}[(i)]
    \item $G$ is $h_1$-equivariant if and only if $\tau_{W_t^{\greenmarking}}(h_2)-b-\tau_{W_s^{\greenmarking}}(h_2)=0$;
    \item $G$ is $h_2$-equivariant if and only if $\tau_{W_t^{\greenmarking}}(h_2)+a-\tau_{W_s^{\greenmarking}}(h_2)=0$.
  \end{enumerate}
  In particular, if $G\colon W_s^{\greenmarking}\to W_t^{\greenmarking}$ is $\fg$-equivariant, then $G\colon W_s^{\greenmarking}\langle a,b\rangle\to W_t^{\greenmarking}\langle a,b\rangle$ is $\fg$-equivariant.
\end{lemma}

\begin{lemma}[$\lief$-equivariance]
  \label{lem:f-equivariance}
  Let $W_s^{\greenmarking}$ and $W_t^{\greenmarking}$ be two marked webs with the same underlying web $W$.
  Let $F\colon W_s^{\greenmarking}\to W_t^{\greenmarking}$ be a linear combination where each term is $\id_W$ decorated with a single dot.
  Let $G$ and $H$ be linear combinations of $\lief$-equivariant foams, suitably composable with $F$.
  If
  \[G\circ\tau_{W_s^{\greenmarking}}(f)\circ H=G\circ\tau_{W_t^{\greenmarking}}(f)\circ H,\]
  then $G\circ F\circ H$ is $f$-equivariant.
  In particular, if $\tau_{W_s^{\greenmarking}}(f)=\tau_{W_t^{\greenmarking}}(f)$, then $F$ is $\lief$-equivariant.
\end{lemma}

In some sense, the condition states that ``globally'', i.e.\ when dots are allowed to move in $G\circ F\circ H$, the marked webs $W_s^{\greenmarking}$ and $W_t^{\greenmarking}$ have identical $f$-markings. In practice, one can move $f$-markings along connected components of the underlying unmarked web $W$, and across components if they happen to be connected in $G\circ F\circ H$.

\begin{lemma}[$\liee$-equivariance]
  \label{lem:e-equivariance}
  Let $W_s^{\greenmarking}$ and $W_t^{\greenmarking}$ be two marked webs with the same underlying web $W$.
  Let $F\colon W_s^{\greenmarking}\to W_t^{\greenmarking}$ be a linear combination where each term is $\id_W$ decorated with a single dot.
  Write $\epsilon(F)$ for the sum of coefficients in $F$.
  If
  \[\epsilon(F)=0,\]
  then $F$ is $\liee$-equivariant.
\end{lemma}

Below we sometimes implicitly assume that $2$ is invertible in the ground ring.

\begin{notation}
We use the following notation:
\begin{equation*}
  \tikzpic{\draw (0,0) to (1,0);\node[round_mark] (A) at (.5,0) {};}
  \quad\coloneqq\quad
  \tikzpic{
    \coordinate (base) at (.5,-.5ex);
    \draw (0,0) to (1,0);\node[green_mark] (A) at (.5,0) {};
    \node[below={-3pt} of A] {\scriptsize $(-1,-\frac{1}{2},-\frac{1}{2})$};
  }[baseline=(base)]
\end{equation*}
\end{notation}

\begin{lemma}
  \label{lem:twist_dot_slide_cup_cap}
  \def\webscl{.6}
  The following are $\fg$-equivariant:
  \begin{gather*}
    \tikzpic{\websMarkedT[0][0][$\alpha$]}%
      [scale=\webscl][(.5,.5*\webscl)]
    =
    \tikzpic{\websMarkedB[0][0][$\alpha$]}%
      [scale=\webscl][(.5,.5*\webscl)]
    \qquad\an\qquad
    \tikzpic{\webmMarkedT[0][0][$\alpha$]}%
      [scale=\webscl][(.5,.5*\webscl)]
    =
    \tikzpic{\webmMarkedB[0][0][$\alpha$]}%
      [scale=\webscl][(.5,.5*\webscl)]
  \end{gather*}
\end{lemma}

\begin{lemma}
\label{lem:crossing_complex_equivariant}
  \def\webscl{.6}
  The following are $\fg$-equivariant:
  \begin{IEEEeqnarray*}{CcC}
    \tikzpic{\webid}[scale=\webscl][(.5,.5*\webscl)]
    \quad
    \xrightarrow{\tikzpic{\fzip}[scale=.6]}
    \quad
    \tikzpic{\websRoundMarkedB[0][0][]\webm[1][0]}%
      [scale=\webscl][(.5,.5*\webscl)]
    \;\langle \frac{1}{2},-\frac{1}{2}\rangle
    %%%%%%%%%%%%%%%%%%%%
    &\mspace{80mu}&
    %%%%%%%%%%%%%%%%%%%%
    \tikzpic{\websRoundMarkedB[0][0][]\webm[1][0]}%
      [scale=\webscl][(.5,.5*\webscl)]
    \;\langle -\frac{1}{2},\frac{1}{2}\rangle
    \quad
    \xrightarrow{\tikzpic{\funzip}[scale=.6]}
    \quad
    \tikzpic{\webid}[scale=\webscl][(.5,.5*\webscl)]
    %%%%%%%%%%%%%%%%%%%%
    \\[2ex]
    %%%%%%%%%%%%%%%%%%%%
    \tikzpic{\draw[web2] (-1,0) to (1,0);}%
      [scale=\webscl][(.5,0*\webscl)]
    \quad
    \xrightarrow{\tikzpic{\fcup}[scale=.6]}
    \quad
    \tikzpic{\webmRoundMarkedB[0][0][]\webs[1][0]}%
      [scale=\webscl][(.5,.5*\webscl)]
    \;\langle -\frac{1}{2},\frac{1}{2}\rangle
    %%%%%%%%%%%%%%%%%%%%
    &\mspace{80mu}&
    %%%%%%%%%%%%%%%%%%%%
    \tikzpic{\webmRoundMarkedB[0][0][]\webs[1][0]}%
      [scale=\webscl][(.5,.5*\webscl)]
    \;\langle \frac{1}{2},-\frac{1}{2}\rangle
    \quad
    \xrightarrow{\tikzpic{\fcap}[scale=.6]}
    \quad
    \tikzpic{\draw[web2] (-1,0) to (1,0);}%
      [scale=\webscl][(.5,0*\webscl)]
  \end{IEEEeqnarray*}
\end{lemma}
 
\begin{lemma}
  \label{lem:web_defining_relations_equivariance}
  \def\webscl{.6}
  The following isomorphisms are $\fg$-equivariant:
  \begin{gather*}
    \stackanchor{W_{i,s_1}W_{j,s_2}\cong^\fg W_{j,s_2}W_{i,s_1}}{(\text{for all }s_1,s_2\in\{-,+\}\an \abs{i-j}>1)}
    \\[2ex]
    \tikzpic{
      \websRoundMarkedB[1][1]\webm[2][1]
      \webmRoundMarkedB\webs[3][0]
      \draw[web1] (0,2) to (1,2);\draw[web1] (3,2) to (4,2);
      \draw[web1] (1,0) to (3,0);
    }[scale=.4][(2,1.25*.4)]
    \cong^\fg
    \tikzpic{
      \coordinate (C) at (2,1);
      \draw[web1] (0,2) to (4,2);
      \draw[web2] (0,.5) to (4,.5);
    }[scale=.4][(2,1.25*.4)]
    \quad\an\quad
    \tikzpic{
      \webmRoundMarkedB[0][1]\webs[3][1]
      \websRoundMarkedB[1][0]\webm[2][0]
      \draw[web1] (0,0) to (1,0);\draw[web1] (3,0) to (4,0);
      \draw[web1] (1,2) to (3,2);
    }[scale=.4][(2,.75*.4)]
    \cong^\fg
    \tikzpic{
      \draw[web1] (0,0) to (4,0);
      \draw[web2] (0,1.5) to (4,1.5);
    }[scale=.4][(2,.75*.4)]
  \end{gather*}
  Furthermore, the following is a $\fg$-equivariant split short exact sequence:
  \begin{gather*}
    \tikzpic{
      \draw[web2] (0,0) to (2,0);
    }[scale=\webscl]
    \;\langle \frac{1}{2},-\frac{1}{2}\rangle
    \quad
    \xrightarrow{\tikzpic{\fcup}[scale=.6]}
    \quad
    \tikzpic{\webmRoundMarkedB[0][0][]\webs[1][0]}%
      [scale=\webscl][(0,.5*\webscl)]
    \quad
    \xrightarrow{\tikzpic{\fcap}[scale=.6]}
    \quad
    \tikzpic{
      \draw[web2] (0,0) to (2,0);
    }[scale=\webscl]
    \;\langle -\frac{1}{2},\frac{1}{2}\rangle
  \end{gather*}
\end{lemma}

\section{Local actions on odd and covering Khovanov homology}
\label{sec:action_on_homology}

In this section, we describe actions on odd and covering Khovanov homology.
Following \cite{KR_PositiveHalfWitt_2016,QRS+_Symmetries$mathfrakgl_N$foams_2024a},
\cref{subsec:relative_homotopy_category} introduces the relative homotopy category, which gives the formal framework where the invariant is defined.
Our exposition is slightly different from \emph{op.\ cit.} (beyond using graded structures), as we avoid the use of triangulated categories.
We then define the marked tangle invariant in \cref{subsec:defn_tangle_invariant}.
Finally, \cref{subsec:topological_invariance} and \cref{subsec:marking_slide} show topological invariance and marking slide, respectively.

\subsection{The relative homotopy category}
\label{subsec:relative_homotopy_category}

\begin{convention}
  All chain complexes are bounded chain complexes.
\end{convention}

Let $A$ be a $\fg$-category.
A \emph{$\fg$-equivariant chain complex} is a chain complex in $A$ whose differential has $\fg$-equivariant components.
% In the sequel, we only consider $\fg$-equivariant chain complexes, referred to as ``(chain) complexes''.
A chain morphism is said to be \emph{$\fg$-equivariant} if each of its components is $\fg$-equivariant.
We denote $\Ch(A)$ the category of $\fg$-equivariant chain complexes and $\fg$-equivariant chain morphisms, and $\underline{\Ch}(A)$ the category of $\fg$-equivariant chain complexes and \emph{all} chain morphisms.
There is an embedding $\Ch(A)\hookrightarrow\underline{\Ch}(A)$.

Homotopies have the standard meaning; that is, homotopies for the pre-additive category underlying $A$. A \emph{$\fg$-equivariant homotopy equivalence} is a homotopy equivalence which is also $\fg$-equivariant as a chain morphism.
Note that if $f$ is a $\fg$-equivariant homotopy equivalence, its inverse needs not be $\fg$-equivariant.
A $\fg$-equivariant chain complex $C_\bullet$ is \emph{contractible} if it is contractible in the standard sense, that is, if the (necessarily $\fg$-equivariant) chain morphism $C_\bullet\to 0$ (or equivalently, $0\to C_\bullet$) is a homotopy equivalence.

\begin{definition}
  \label{defn:formal_relative_homotopy_category}
  Let $A$ be a $\fg$-category.
  The \emph{relative homotopy category} $\cK^\fg(A)$ is the localization of $\Ch(A)$ at $\fg$-equivariant homotopy equivalences.
  We denote $\simeq^\fg$ an isomorphism in $\cK^\fg(A)$.
\end{definition}

We unpack the definition; see \cite[section~III.2.2]{GM_MethodsHomologicalAlgebra_2003} for a more thorough review of localization.
Objects of $\cK^\fg(A)$ are the same objects as those in $\Ch(A)$; namely, $\fg$-equivariant chain complexes in $A$.
In this context, a \emph{path} is formal composition of arrows
\[u_0\overset{f_1}{\longrightarrow}u_1\overset{f_2}{\longrightarrow}\ldots\overset{f_n}{\longrightarrow}u_n,\]
where $f_i\colon u_{i-1}\to u_i$ is either a $\fg$-equivariant chain morphism or the inverse of a $\fg$-equivariant homotopy equivalence.
Two paths are \emph{equivalent} if they can be joined by a chain of the following elementary equivalences:
\begin{itemize}
  \item two consecutive arrows are replaced by their composition;
  \item the composition of a $\fg$-equivariant homotopy equivalence with its inverse is replaced by the identity.
\end{itemize}
Morphisms in $\cK^\fg(A)$ are equivalences classes of paths.

Denote $\cK(A)$ (resp.\ $\underline{\cK}(A)$) the homotopy category of $\Ch(A)$ (resp.\ $\underline{\Ch}(A)$), that is, the localization of $\Ch(A)$ (resp.\ $\underline{\Ch}(A)$) at $\fg$-equivariant homotopy equivalences with $\fg$-equivariant inverses (resp.\ at homotopy equivalences).
The relative homotopy category $\cK^\fg(A)$ can be understood as sitting in between $\cK(A)$ and $\underline{\cK}(A)$, inverting homotopy equivalence that are $\fg$-equivariant but may not have $\fg$-equivariant inverses.
Namely, there is a commutative diagram
\begin{center}
  \begin{tikzcd}
    \cK(A) \ar[dr]\ar[rr] && \underline{\cK}(A)
    \\
    & \cK^\fg(A)\ar[ur]
  \end{tikzcd}
\end{center}
with each arrow being the obvious quotient functor.
The category $\cK^\fg(A)$ satisfies the universal property that if $F\colon \Ch(A)\to T$ is a functor sending $\fg$-equivariant homotopy equivalences to isomorphisms, then the functor $F$ factors through the quotient $\Ch(A)\to \cK^\fg(A)$.

\begin{remark}
  As a triangulated category, the category $\cK^\fg(\ring)$ coincides with the relative homotopy category $\cC^\fg(\ring)$ as defined in \cite{QRS+_Symmetries$mathfrakgl_N$foams_2024a}.
\end{remark}

If $C_\bullet$ is a $\fg$-equivariant chain complex in $\gMod$, its homology $H_\bullet(C)$ is canonically endowed with a structure of $\fg$-module, preserving the homological grading.
If $f\colon C_\bullet\to D_\bullet$ is a $\fg$-equivariant chain morphism, it induces a linear map $[f]\colon H_\bullet(C)\to H_\bullet(D)$, preserving both the homological grading and the $\fg$-action.
Furthermore, if $f$ is a homotopy equivalence then $[f]$ an isomorphism.
Recall $\ugMod$ from \cref{ex:ugMod} and write $\cK^\fg(\ring)\coloneqq\cK^\fg(\ugMod)$.
It follows from the above discussion and the universal property of $\cK^\fg(\ring)$ that the homology functor $H_\bullet$ descends to a functor from $\cK^\fg(\ring)$ to $(\gMod)^\bZ$.
This justifies the definition of the relative homotopy category as the category capturing which $\fg$-equivariant complexes have the same homology, with the same induced $\fg$-action.

If $A$ is a $\fg$-category, any choice of object $w$ in $A$ defines a representable functor
\[\Hom_{A}(w,-)\colon A\to\ugMod.\]
It descends to the relative homotopy categories, leading to a homology functor:
\[\cK^\fg(A)\overset{\Hom_{A}(w,-)}{\longrightarrow} \cK^\fg(\ring)\overset{H_\bullet}{\longrightarrow}(\gMod)^\bZ.\]

\subsection{Definition of the marked tangle invariant}
\label{subsec:defn_tangle_invariant}

In this subsection, we adapt the construction of the tangle invariant in \cite{SV_OddKhovanovHomology_2023} to carry a $\fg$-action.
In this article, a tangle diagram always refers to a sliced tangle diagram.

\medbreak

In contrast with \cite{SV_OddKhovanovHomology_2023}, the construction in this article applies to ``marked tangles''.
A \emph{marked tangle diagram} is a tangle diagram with extra markings $\greenmarking[2]$, each labelled with a triple of scalars, as for webs.
For us, a \emph{marked tangle} is an equivalence class of marked tangle diagrams with respect to the standard relations on tangle diagrams, together with the relations (here $\omega=(\alpha,\beta_1,\beta_2)$ is a generic triple of scalars):
\begin{gather}
  \label{eq:marking_slide_through_cup_cap}
  \tikzpic{
    \draw[web1] (1,0) to[out=180,in=-90] (.5,.5) to[out=90,in=180] (1,1);
    \node[green_mark] (B) at (1-.35,1-.15) {};
    \node[above left={-5pt} of B] {\scriptsize $\omega$};
  }[xscale=.8,yscale=.6][(0,.5*.6)]
  \;\leftrightarrow\;
  \tikzpic{
    \draw[web1] (1,0) to[out=180,in=-90] (.5,.5) to[out=90,in=180] (1,1);
    \node[green_mark] (B) at (1-.35,.15) {};
    \node[below left={-5pt} of B] {\scriptsize $\omega$};
  }[xscale=.8,yscale=.6][(0,.5*.6)]
  \quad\an\quad
  \tikzpic{
    \draw[web1] (0,0) to[out=0,in=-90] (.5,.5) to[out=90,in=0] (0,1);
    \node[green_mark] (B) at (.35,1-.15) {};
    \node[above right={-5pt} of B] {\scriptsize $\omega$};
  }[xscale=.8,yscale=.6][(0,.5*.6)]
  \;\leftrightarrow\;
  \tikzpic{
    \draw[web1] (0,0) to[out=0,in=-90] (.5,.5) to[out=90,in=0] (0,1);
    \node[green_mark] (B) at (.35,.15) {};
    \node[below right={-5pt} of B] {\scriptsize $\omega$};
  }[xscale=.8,yscale=.6][(0,.5*.6)]
  \;.
\end{gather}
That is, markings can slide along strands, but (a priori) not over or under crossings.
One could give a topological description, with markings being points where the tangle is ``glued onto the plane'' or ``attached to the point at infinity'', depending on the topological model. We omit the details.

A \emph{(marked) tangled web} is a (marked) web where one may further use the following crossings:
\begin{gather*}
  \def\webscl{.6}
  \tikzpic{
    \webncr
  }[scale=\webscl][(0,.5*\webscl)]
  \qquad\an\qquad
  \tikzpic{
    \webpcr
  }[scale=\webscl][(0,.5*\webscl)]
  \;.
\end{gather*}
If $W$ is a marked tangled web, then $\undcomp(W)$ is a marked tangle diagram. As explained in \cite{SV_OddKhovanovHomology_2023} (and following \cite{LQR_KhovanovHomologySkew_2015}), we have that:

\begin{lemma}
  For any marked tangle diagram $T$, there exists a marked tangled web $W$ such that $\undcomp(W)=T$.
\end{lemma}

To realise the above lemma in practice, it is useful to introduce mixed crossings:
\begin{gather*}
  \tikzpic{\webcrDS}[xscale=.6,yscale=.4]\coloneqq
  \tikzpic{
    \webm[0][1]\draw[web1] (1,2) to (2,2);
    \draw[web1] (0,0) to (1,0);\websRoundMarkedT[1][0]
  }[xscale=.6,yscale=.4]
  \qquad
  \tikzpic{\webcrSD}[xscale=.6,yscale=.4]\coloneqq
  \tikzpic{
    \draw[web1] (0,2) to (1,2);\websRoundMarkedT[1][1]
    \webm\draw[web1] (1,0) to (2,0);
  }[xscale=.6,yscale=.4]
  \qquad\an\qquad
  \tikzpic{\webcrDD}[xscale=.6,yscale=.4]\coloneqq
  \tikzpic{
    \draw[web2] (0,1) to (1,1);
    \draw[web2] (0,0) to (1,0);
  }[xscale=.6,yscale=.4]
  \;.
\end{gather*}
Two different tangled webs $D_1$ and $D_2$ can have the same underlying tangle diagram $\undcomp(D_1) =\undcomp(D_2)$. Indeed, we may have that $D_1\in\fg\foam_{d_1}$ and $D_2\in\fg\foam_{d_2}$ for $d_1\neq d_2$, as we can always add a double line on the top or bottom of a web; and even if $d_1=d_2$, the webs $D_1$ and $D_2$ may not have the same input and output coordinates; as we can always compose a web horizontally with a mixed crossing.
Instead, we would want to think of tangled webs
\begin{itemize}
  \item up to adding double lines on the top or bottom of the web
  \item and up to adding mixed crossings on the right or the left of the web.
\end{itemize}
This is formalize as follows.
On the one hand, adding a double line to a web (resp.\ a 2-facet to a foam) on top (resp.\ on the back) defines a $\fg$-2-functor $\fg\foam_d^{\greenmarking}\to\fg\foam^{\greenmarking}_{d+2}$ (see also the front-back composition from \cref{rem:monoidal-2-categorical-structure}). In fact, it follows from the basis theorem shown in \cite{Schelstraete_RewritingModuloDiagrammatic_2025} that these $\fg$-2-functors are embeddings. We refer to this type of $\fg$-2-functors as ``adding double lines''.
On the other hand, pre- and post-composing with mixed crossings define various $\fg$-2-functors $\fg\foam_d^{\greenmarking}(\lambda,\mu)\to\fg\foam_d^{\greenmarking}(\lambda',\mu')$, where $\lambda$ (resp.\ $\mu$) has the same number of $1$'s as $\lambda'$ (resp.\ $\mu'$). In fact, these $\fg$-2-functors are isomorphisms. We refer to this type of $\fg$-2-functors as ``changing the endpoints''.

\begin{definition}
  \label{defn:colimit_of_foam}
  The $\fg$-2-category $\fg\foam^{\greenmarking}$ is the colimit over ``adding double lines'' and ``changing the endpoints'' $\fg$-2-functors.
  % \begin{center}
  %   \begin{tikzcd}
  %     & \fg\foam^{\greenmarking}_d \ar[dl,hook']\ar[dr,hook]
  %     \\
  %     \fg\foam^{\greenmarking}_{d+2} && \fg\foam^{\greenmarking}_{d+2}
  %   \end{tikzcd}
  % \end{center}
  % where the arrows are embeddings given by adding a 2-facet at the back and at the front, respectively.
\end{definition}

The tangle invariant in \cite{SV_OddKhovanovHomology_2023} is defined as a certain tensor product of chain complexes in $\fg\foam$.
In this context, one needs a new notion of tensor product of chain complexes to account for the graded interchange law, which we review now.

In practice, it means that there exists a graded analogue of the Koszul rule, suitably compatible with homotopy equivalence; this was shown in the second author's master thesis \cite{Schelstraete_SupercategorificationKhovanovlikeTangle_2020}.
Recall that the usual tensor product of two chain complexes looks like a grid, and the Koszul rule is a way to assign signs to edges, such that each square has an odd number of signs; this makes the induced differential squares to zero.
In fact, one does not need to follow the Koszul rule: any two ways of assigning signs to edges such that each square has an odd number of signs lead to isomorphic chain complexes.

In a similar way, if $A_\bullet$ and $B_\bullet$ are two chain complexes with homogeneous differentials, the graded Koszul rule defines $A_\bullet\otimes B_\bullet$ by assigning invertible scalars to edges; denote it $\epsilon$, and view it as 1-cochain on the oriented grid.
For each square
\begin{center}
  \begin{tikzcd}
    \bullet \ar[r,"\alpha\otimes\id"] \ar[d,"\id\otimes\beta"'] & \bullet
    \ar[d,"\id\otimes\beta"]\\
    \bullet \ar[r,"\alpha\otimes\id"] & \bullet
  \end{tikzcd}
\end{center}
in the grid, the assignment $\epsilon$ is such that $\partial\epsilon = \bil(\alpha,\beta)$.
In fact, one does not need to follow the graded Koszul rule: any two ways of assigning invertible scalars to edges such that each square has this property lead to isomorphic chain complexes.
One can proceed inductively and define a tensor product for ``sufficiently homogeneous complexes'', called ``homogeneous polycomplexes''.
We refer to \cite{SV_OddKhovanovHomology_2023} for the precise definition.
As for the classical tensor product, this graded tensor product is suitably compatible with homotopy equivalence.

It is not hard to extend the above construction to the equivariant setting.
We omit the details, and summarize the main points in the following proposition:

\begin{proposition}
  \label{prop:tensor_product_chain_complexes}
  Let $\cA$ be an additive $\fg$-2-category.
  For a certain family of complexes called \emph{$\fg$-equi\-va\-riant homogeneous polycomplexes}, there exists a procedure that, given two $\fg$-equi\-va\-riant homogeneous polycomplexes $A_\bullet$ and $B_\bullet$, defines a $\fg$-equi\-va\-riant homogeneous polycomplex $A_\bullet\otimes B_\bullet$.
  We call it the \emph{graded tensor product of chain complexes}.
  This procedure is such that:
  \begin{equation*}
    A_\bullet\simeq^\fg C_\bullet
    \quad\an\quad
    B_\bullet\simeq^\fg D_\bullet
    \qquad\Rightarrow\qquad
    A_\bullet\otimes B_\bullet\simeq^\fg C_\bullet\otimes D_\bullet,
  \end{equation*} 
  where $\simeq^\fg$ denotes isomorphism in the relative homotopy category $\cK^\fg(\cA)$.
\end{proposition}

We can now define the tangle invariant:

\begin{definition} 
\label{defn:tangle_invariant}
  \def\webscl{.6}
  Assume $2$ is invertible in $\ringfoam$.
  Let $D$ be a diagram of tangled webs with markings.
  The complex $\fg\glKom(D)\in\Ch(\fg\markedfoam[])$ is defined on elementary marked webs as follows (the homological degree zero is underlined):
  \begin{gather*}
    \tikzpic{
      \draw[web1] (-1,0) to (1,0);
      \node[green_mark] (B) at (0,0){};
      \node[below={-3pt} of B] {\scriptsize $(\alpha,\beta_1,\beta_2)$};
    }[scale=.7][(0,0)]
    \mspace{10mu}\mapsto\mspace{10mu}
    \tikzpic{
      \draw[web1] (-1,0) to (1,0);
      \node[green_mark] (B) at (0,0){};
      \node[below={-3pt} of B] {\scriptsize $(\alpha,\beta_1,\beta_2)$};
      \draw[decorate,decoration=snake] (-1,-1) to (1,-1);
    }[scale=.7][(0,0)]
    \mspace{60mu}
    \tikzpic{
      \webm
    }[scale=\webscl][(0,.5*\webscl)]
    \mspace{10mu}\mapsto\mspace{10mu}
    \tikzpic{
      \webm
      \draw[decorate,decoration=snake] (0,-.5) to (1,-.5);
    }[scale=\webscl][(0,.5*\webscl)]
    \mspace{60mu}
    \tikzpic{
      \webs
    }[scale=\webscl][(0,.5*\webscl)]
    \mspace{10mu}\mapsto\mspace{10mu}
    \tikzpic{
      \websRoundMarkedB[0][0][]
      \draw[decorate,decoration=snake] (0,-.5) to (1,-.5);
    }[scale=\webscl][(0,.5*\webscl)]
    \\[2ex]
    %%%%%%%%%%%%%%%%%%%%%%%%%%%%%%
    %%%%%%%%%%%%%%%%%%%%%%%%%%%%%%
    \tikzpic{
      \webncr
    }[scale=\webscl][(0,.5*\webscl)]
    \quad\mapsto\quad
    \tikzpic{
      \websRoundMarkedB[0][0][]
      \webm[1][0]
    }[scale=\webscl][(0,.5*\webscl)]
    \;\langle-\frac{1}{2},\frac{1}{2}\rangle
    \quad
    \xrightarrow{\tikzpic{\funzip}[scale=.6]}
    \quad
    \tikzpic{
      \webid
      \draw[decorate,decoration=snake] (0,-.5) to (2,-.5);
    }[scale=\webscl][(0,.5*\webscl)]
    \\[2ex]
    %%%%%%%%%%%%%%%%%%%%%%%%%%%%%%
    %%%%%%%%%%%%%%%%%%%%%%%%%%%%%%
    \tikzpic{
      \webpcr
    }[scale=\webscl][(0,.5*\webscl)]
    \quad\mapsto\quad
    \tikzpic{
      \webid
    }[scale=\webscl][(0,.5*\webscl)]
    \;\langle-\frac{1}{2},\frac{1}{2}\rangle
    \quad
    \xrightarrow{\tikzpic{\fzip}[scale=.6]}
    \quad
    \tikzpic{
      \websRoundMarkedB[0][0][]
      \webm[1][0]
      \draw[decorate,decoration=snake] (0,-.5) to (2,-.5);
    }[scale=\webscl][(0,.5*\webscl)]
  \end{gather*}
  and extending to $D$ by taking graded tensor product of chain complexes, as given in \cref{prop:tensor_product_chain_complexes}.
\end{definition}

\begin{remark}
  Contrary to \cite{QRS+_SymmetriesEquivariantKhovanovRozansky_2023}, twists do not only arise from crossings. This should be explained by the fact that while both theories are oriented (ie use webs), the setting of \cite{SV_OddKhovanovHomology_2023} uses the more restricted setting of directed webs.
  At the time of writing, there is no non-directed model for odd (or covering) Khovanov homology.
\end{remark}

The next two subsections explore its property, namely topological invariance, and how markings can slide through crossings.

\subsection{Topological invariance}
\label{subsec:topological_invariance}

In this subsection, we prove topological invariance:

\begin{theorem}
  \label{thm:topological_invariance}
  Assume $2$ is invertible in $\ringfoam$.
  Let $T$ be a marked tangle and $D$ a marked tangled web presenting $T$.
  Denote $N_+$ (resp.\ $N_-$) the number of positive (resp.\ negative) crossings in $D$, and $w\coloneqq N_+-N_-$ its writhe.
  In the relative homotopy category $\cK^\fg(\fg\markedfoam[])$, the object
  \begin{gather*}
    \fg\glKh(D)\coloneqq t^{N_+}\fg\glKom(D)\langle\frac{w+N_+}{2},-\frac{w+N_+}{2}\rangle
  \end{gather*}
  only depends on $T$, up to isomorphism.
  We write $\glCKh(D)\coloneqq \fg\glKh(D)$ when $\fg=\grgl_2^\geq$ and $\glOKh(D)\coloneqq \fg\glKh(D)$ when $\fg=\gloo$.
\end{theorem}

Here we remind the reader that as defined in the beginning of \cref{subsec:twist_foam}, every twist $(\alpha,\beta_1,\beta_2)$ carries a shift in the quantum grading by $q^{\beta_2-\beta_1}$.
Note that this theorem does not say anything on how markings can slide through crossings: this is discussed in the next subsection.

\begin{remark}
  \label{rem:local_invariant_2_invertible}
  If one restricts $\grgl^\geq_2$ to $\grsl^\geq_2$ as in \cref{ex:defn_covering_sl2} (resp.\ $\gloo$ to $\sloo$ as in \cref{ex:defn_sloo}), then one can do away with the condition that $2$ is invertible in the ground ring $\ringfoam$.
\end{remark}

The remainder of this subsection is devoted to the proof of \cref{thm:topological_invariance}.
The proof is adapted from the proof of invariance in \cite{SV_OddKhovanovHomology_2023}, incorporating $\fg$-equivariance as in the analogous proof in \cite{QRS+_SymmetriesEquivariantKhovanovRozansky_2023}.
Given our description of the relative homotopy category given in \cref{subsec:relative_homotopy_category}, finding an isomorphism in $\cK^\fg(\fg\foam^{\greenmarking})$ amounts to finding a zigzag of $\fg$-equivariant homotopy equivalences in $\underline{\Ch}(\fg\foam^{\greenmarking})$.

Thanks to the properties of the graded tensor product of complexes (\cref{prop:tensor_product_chain_complexes}), we can work locally.
We first show invariance under planar isotopy in \cref{lem:invariance_planar_isotopies}, where a \emph{planar isotopy} between two marked tangles is a planar isotopy between the underlying tangles, such that markings do not slide through crossings. (In other words, it consists of the usual planar isotopy relations, together with \eqref{eq:marking_slide_through_cup_cap}).
We then show invariance under Reidemeister I, Reidemeister II and Reidemeister III in \cref{lem:invariance_Reidemeister_I}, \cref{lem:invariance_Reidemeister_II} and \cref{lem:invariance_Reidemeister_III}, respectively.

\begin{lemma}
  \label{lem:invariance_planar_isotopies}
  Let $D_1$ and $D_2$ be two marked tangled webs. If $\undcomp(D_1)$ and $\undcomp(D_2)$ are planar isotopic, then there is an equivariant isomorphism between $\fg\glKom(D_1)$ and $\fg\glKom(D_2)$ in $\Ch(\fg\markedfoam[])$.
\end{lemma}

\begin{proof}
  We have already seen in \cref{lem:twist_dot_slide_cup_cap} that markings can slide through cups and caps.
  Hence, it suffices to prove invariance under elementary planar isotopies (see \cite[Figure~3.2]{SV_OddKhovanovHomology_2023}), following the proof of Lemma 3.8 in \cite{SV_OddKhovanovHomology_2023}.
  On the one hand,
  invariance under planar isotopies interchanging two elementary tangles is realised by foam crossings, which are always $\fg$-equivariant;
  on the other hand,
  invariance under zigzags isotopies and pitchfork isotopies essentially use the isomorphisms given in \cref{lem:web_defining_relations_equivariance}, which we showed to be $\fg$-equivariant.
  For instance, the isomorphism for one of the pitchfork isotopies is given as follows:
  \begin{center}
    \begin{tikzpicture}
      \def\vspc{2}\def\hspc{5}
      \def\webxscl{.4}\def\webyscl{.3}
      \node (A1) at (0,\vspc) {
        \tikzpic{
          \websRoundMarkedT[1][0]
          \webpcr[0][1]
          \draw[web1] (0,0) to (1,0);
          \draw[web1] (1,2) to (2,2);
        }[xscale=\webxscl,yscale=\webyscl]
      };
      \node (A2) at (\hspc,\vspc) {
        \tikzpic{
          \websRoundMarkedT[1][0]
          \draw[web1] (1,2) to (2,2);
        }[xscale=\webxscl,yscale=\webyscl]
        $\langle -\frac{1}{2},\frac{1}{2} \rangle$
      };
      \node (A3) at (2*\hspc,\vspc) {
        \tikzpic{
          \websRoundMarkedT[1][0]
          \websRoundMarkedT[-1][1]\webm[0][1]
          \draw[web1] (0,0) to (1,0);
          \draw[web1] (1,2) to (2,2);
          \draw[web1] (-1,0) to (0,0);
        }[xscale=\webxscl,yscale=\webyscl]
      };
      \node (B1) at (0,0) {
        \tikzpic{
          \websRoundMarkedB[1][1]\webm[2][1]
          \webncr[0][0]\websRoundMarkedB[3][0]
          \draw[web1] (0,2) to (1,2);
          \draw[web1] (3,2) to (4,2);
          \draw[web1] (1,0) to (3,0);
        }[xscale=\webxscl,yscale=\webyscl]
      };
      \node (B2) at (\hspc,0) {
        \tikzpic{
          \websRoundMarkedB[1][1]\webm[2][1]
          \websRoundMarkedB[-1][0]\webm[0][0]\websRoundMarkedB[3][0]
          \draw[web1] (-1,2) to (1,2);
          \draw[web1] (3,2) to (4,2);
          \draw[web1] (1,0) to (3,0);
        }[xscale=\webxscl,yscale=\webyscl]$\langle -\frac{1}{2},\frac{1}{2} \rangle$
      };
      \node (B3) at (2*\hspc,0) {
         \tikzpic{
          \websRoundMarkedB[1][1]\webm[2][1]
          \websRoundMarkedB[3][0]
          \draw[web1] (3,2) to (4,2);
          \draw[web1] (1,0) to (3,0);
        }[xscale=\webxscl,yscale=\webyscl]
      };
      \draw[->] (A2) to (A3);
      \node[above] at (1.5*\hspc,\vspc) {
        \tikzpic{\fzip[-.5][0][2]\frst[1][0]}[scale=.4]
      };
      \draw[->] (B2) to (B3);
      \node[above] at (1.5*\hspc,0) {
        \tikzpic{\funzip[-.5][0]\frst[1][0][2]\flst[2][0][2]\frst[2.5][0]}[scale=.4]
      };
      \draw[->] (B2) to (A2);
      \node[left] at (\hspc,.5*\vspc) {$\lambda\;
        \tikzpic{
          \clip (-2,0) rectangle (2,2);
          \funzip[0][0][2]
          \begin{scope}[scale=2]
            \fcap[-.25][0]
          \end{scope}
          \draw[foamdraw1] (-1.5,0) to (-1.5,2);
        }[scale=.4]$
      };
      \draw[->] (B3) to (A3);
      \node[right] at (2*\hspc,.5*\vspc) {$\id$};
      \node[right={5pt} of A1] {:};
      \node[right={5pt} of B1] {:};
    \end{tikzpicture}
  \end{center}
  where $\lambda$ is some scalar that we do not need to compute; here we use the squeezing relation. This concludes.
\end{proof}

Before proving invariance under Reidemeister moves, we recall the following homological fact.

\begin{lemma}
  \label{lem:abstract_homotopy_equivalence}
  Let $A$ be an additive category and let
  \[P_\bullet\overset{f}{\longrightarrow} C_\bullet\overset{g}{\longrightarrow} D_\bullet\overset{h}{\longrightarrow}Q_\bullet\]
  be a chain complex in $A$ which is split exact at $C_\bullet$ and $D_\bullet$. If $P_\bullet$ and $Q_\bullet$ are contractible, then $g$ is a homotopy equivalence with inverse given by the splitting.
\end{lemma}

\begin{proof}
  Let $\ov{f}$ and $\ov{g}$ the maps giving the splitting at $C_\bullet$, so that $f\circ\ov{f}+\ov{g}\circ g=\id_{C_\bullet}$.
  Let $h^P$ be the homotopy between $\id_{P_\bullet}$ and $0$, so that $h^P\circ d_P+d_P\circ h^P=\id_P$.
  The map $f\circ h^P\circ \ov{f}$ defines a homotopy between $\ov{g}\circ g$ and $\id_{C_\bullet}$, as one can check that:
  \begin{align*}
    (f\circ h^P\circ \ov{f})\circ d_{C_\bullet}+d_{C_\bullet}\circ (f\circ h^P\circ \ov{f})
    &= 
    f\circ h^P\circ d_{P_\bullet}\circ \ov{f}+f\circ d_{P_\bullet}\circ h^P\circ \ov{f}
    \\
    &=
    f\circ\ov{f}
    =
    \id_{C_\bullet}-\ov{g}\circ g.
  \end{align*}
  A similar argument gives a homotopy between $g\circ\ov{g}$ and $\id_{D_\bullet}$.
\end{proof}

\begin{lemma}
  \label{lem:invariance_Reidemeister_I}
  Let $D$ be a marked tangled web.
  In the relative homotopy category $\cK^\fg(\fg\foam^{\greenmarking})$, the object $\fg\glKh(D)$ is invariant under Reidemeister I moves, up to isomorphism.
\end{lemma}

\begin{proof}
  We can proceed locally.
  We must check that, in the relative homotopy category:
  \begin{gather*}
    \def\webxscl{.4}\def\webyscl{.3}
    \tikzpic{
      \draw[web1] (0,2) to (1,2);\webncr[1][1]\draw[web1] (2,2) to (3,2);
      \webm[0][0]\draw[web1] (1,0) to (2,0);\websRoundMarkedB[2][0]
    }[xscale=\webxscl,yscale=\webyscl][(0,1.25*\webyscl)]
    \;\simeq^\fg\;
    \tikzpic{
      \draw[web1] (0,2) to (3,2);
      \draw[web2] (0,.5) to (3,.5);
    }[xscale=\webxscl,yscale=\webyscl]
    \langle \frac{1}{2},-\frac{1}{2} \rangle
    \quad\an\quad
    \tikzpic{
      \draw[web1] (0,2) to (1,2);\webpcr[1][1]\draw[web1] (2,2) to (3,2);
      \webm[0][0]\draw[web1] (1,0) to (2,0);\websRoundMarkedB[2][0]
    }[xscale=\webxscl,yscale=\webyscl][(0,1.25*\webyscl)]
    \;\simeq^\fg\;
    t^{-1}
    \tikzpic{
      \draw[web1] (0,2) to (3,2);
      \draw[web2] (0,.5) to (3,.5);
    }[xscale=\webxscl,yscale=\webyscl]
    \langle -1,1 \rangle
    \;.
  \end{gather*}
  Using a split exact sequence in the spirit of \cref{lem:web_defining_relations_equivariance}, we can fit each left-hand side in a sequence which is split exact at the two middle chain complexes (we omit labelling the arrows in the second case):
  \begin{center}
    \begin{tikzpicture}
      \def\vspc{2}\def\hspc{4.5}
      \def\webxscl{.4}\def\webyscl{.3}
      \node (A2) at (\hspc,\vspc) {
        \tikzpic{
          \draw[web1] (0,1.5) to (3,1.5);
          \draw[web2] (0,0) to (3,0);
        }[xscale=\webxscl,yscale=\webyscl]
      };
      \node[left={-8pt} of A2] {\footnotesize $t^{-1}$};
      \node (A3) at (2*\hspc,\vspc) {
        \tikzpic{
          \draw[web1] (0,1.5) to (3,1.5);
          \draw[web2] (0,0) to (3,0);
        }[xscale=\webxscl,yscale=\webyscl]
      };
      \node[right={-8pt} of A3] {\footnotesize $\langle -\frac{1}{2},\frac{1}{2} \rangle$};
      \node (B2) at (\hspc,0) {
        \tikzpic{
          \websRoundMarkedB[1][1]\webm[2][1]
          \webm[0][0]\websRoundMarkedB[3][0]
          \draw[web1] (0,2) to (1,2);
          \draw[web1] (3,2) to (4,2);
          \draw[web1] (1,0) to (3,0);
        }[xscale=\webxscl,yscale=\webyscl]
      };
      \node[left={-8pt} of B2] {\footnotesize $t^{-1}$};
      \node (B3) at (2*\hspc,0) {
        \tikzpic{
          \draw[web1] (0,2) to (1,2);\webid[1][1][1]\draw[web1] (2,2) to (3,2);
          \webm[0][0]\draw[web1] (1,0) to (2,0);\websRoundMarkedB[2][0]
        }[xscale=\webxscl,yscale=\webyscl]
      };
      \node (C3) at (2*\hspc,-\vspc) {
        \tikzpic{
          \draw[web1] (0,1.5) to (3,1.5);
          \draw[web2] (0,0) to (3,0);
        }[xscale=\webxscl,yscale=\webyscl]
      };
      \node[right={-8pt} of C3] {\footnotesize $\langle \frac{1}{2},-\frac{1}{2} \rangle$};
      \node (D3) at (2*\hspc,-2*\vspc) {$0$};
      %
      % EDGES
      \draw[->] (A2) to (A3);
      \node[above] at (1.6*\hspc,\vspc) {\tiny $\id$};
      \draw[->] (B2) to (B3);
      \node[above] at (1.6*\hspc,0) {\tiny 
        \tikzpic{
          \flst[-.5][0]\funzip[0][0][2]\frst[1.5][0]
        }[scale=.3]
      };
      \draw[->] ($(B2.north)+(.1,0)$) to ($(A2.south)+(.1,0)$);
      \draw[<-,dashed] ($(B2.north)-(.1,0)$) to ($(A2.south)-(.1,0)$);
      \node[right] at (\hspc,.5*\vspc) {
        \tikzpic{
          \clip (-1,0) rectangle (2,2);
          \funzip[0][0][2]
          \begin{scope}[scale=2]
            \fcap[-.25][0]
          \end{scope}
        }[scale=.2]
      };
      \node[left] at (\hspc,.5*\vspc) {\tiny $Z^{-1}
        \tikzpic{
          \clip (-1,0) rectangle (2,2);
          \funzip[0][0][2]
          \begin{scope}[scale=2]
            \fcap[-.25][0]
          \end{scope}
        }[scale=.2,yscale=-1]\;$
      };
      \draw[->] ($(B3.north)+(.1,0)$) to ($(A3.south)+(.1,0)$);
      \draw[<-,dashed] ($(B3.north)-(.1,0)$) to ($(A3.south)-(.1,0)$);
      \node[right] at (2*\hspc,.5*\vspc) {
        \tikzpic{\fcap}[scale=.3]
      };
      \node[left] at (2*\hspc,.5*\vspc) {\tiny $
        \tikzpic{\fcup\fdot[.5][.8][2][1.5]}[scale=.3]
        +XY
        \tikzpic{\fcup\fdot[.5][.8][3][1.5]}[scale=.3]\;$
      };
      \draw[->] ($(C3.north)+(.1,0)$) to ($(B3.south)+(.1,0)$);
      \draw[<-,dashed] ($(C3.north)-(.1,0)$) to ($(B3.south)-(.1,0)$);
      \node[right] at (2*\hspc,-.5*\vspc) {
        \tikzpic{\fcup}[scale=.3]
      };
      \node[left] at (2*\hspc,-.5*\vspc) {\tiny $
        \tikzpic{\fcap\fdot[.5][.2][2][1.5]}[scale=.3]
        -
        \tikzpic{\fcap\fdot[.5][.2][3][1.5]}[scale=.3]\;$
      };
      \draw[->] (D3) to (C3);
      \node[right={-8pt} of B2,fill=white,inner sep=1pt] {\footnotesize $\langle -\frac{1}{2},\frac{1}{2} \rangle$};
      \node[right={-8pt} of A2,fill=white,inner sep=1pt] {\footnotesize $\langle -\frac{1}{2},\frac{1}{2} \rangle$};
    \end{tikzpicture}
    \hspace{1cm}
    \begin{tikzpicture}
      \def\vspc{2}\def\hspc{3.5}
      \def\webxscl{.4}\def\webyscl{.3}
      \node (A2) at (\hspc,\vspc) {
        \tikzpic{
          \draw[web1] (0,1.5) to (3,1.5);
          \draw[web2] (0,0) to (3,0);
        }[xscale=\webxscl,yscale=\webyscl]
      };
      \node[left={-8pt} of A2] {\footnotesize $t^{-1}$};
      \node (A3) at (2*\hspc,\vspc) {
        \tikzpic{
          \draw[web1] (0,1.5) to (3,1.5);
          \draw[web2] (0,0) to (3,0);
        }[xscale=\webxscl,yscale=\webyscl]
      };
      \node (B2) at (\hspc,0) {
        \tikzpic{
          \draw[web1] (0,2) to (1,2);\webid[1][1][1]\draw[web1] (2,2) to (3,2);
          \webm[0][0]\draw[web1] (1,0) to (2,0);\websRoundMarkedB[2][0]
        }[xscale=\webxscl,yscale=\webyscl]
      };
      \node[left={-8pt} of B2] {\footnotesize $t^{-1}$};
      \node (B3) at (2*\hspc,0) {
        \tikzpic{
          \websRoundMarkedB[1][1]\webm[2][1]
          \webm[0][0]\websRoundMarkedB[3][0]
          \draw[web1] (0,2) to (1,2);
          \draw[web1] (3,2) to (4,2);
          \draw[web1] (1,0) to (3,0);
        }[xscale=\webxscl,yscale=\webyscl]
      };
      \node (C2) at (\hspc,-\vspc) {
        \tikzpic{
          \draw[web1] (0,1.5) to (3,1.5);
          \draw[web2] (0,0) to (3,0);
        }[xscale=\webxscl,yscale=\webyscl]
      };
      \node[left={-8pt} of C2] {\footnotesize $t^{-1}$};
      \node[right={-8pt} of C2] {\footnotesize $\langle -1,1\rangle$};
      \node (D2) at (\hspc,-2*\vspc) {$0$};
      \draw[->] (A2) to (A3);
      \draw[->] (B2) to (B3);
      \draw[->] (A2) to (B2);
      \draw[->] (B2) to (C2);
      \draw[->] (C2) to (D2);
      \draw[->] (A3) to (B3);
      \node[right={-8pt} of B2,fill=white,inner sep=1pt] {\footnotesize $\langle -\frac{1}{2},\frac{1}{2} \rangle$};
    \end{tikzpicture}
  \end{center}
  Colours $\textdot[1]$, $\textdot[2]$ and $\textdot[3]$ are labels $1$, $2$ and $3$, respectively.
  The top chain complex is the cone of an identity while the bottom chain complex is zero: we are in the situation of \cref{lem:abstract_homotopy_equivalence}.
  Finally, the middle chain morphism is $\fg$-equivariant, so that it defines a $\fg$-equivariant homotopy equivalence.
\end{proof}

\begin{lemma}
  \label{lem:invariance_Reidemeister_II}
  Let $D$ be a marked tangled web.
  In the relative homotopy category $\cK^\fg(\fg\foam^{\greenmarking})$, the object $\fg\glKh(D)$ is invariant under Reidemeister II moves, up to isomorphism.
\end{lemma}

\begin{proof}
  We can proceed locally.
  We must check that, in the relative homotopy category:
  \begin{gather*}
    \tikzpic{\webpcr[0][0]\webncr[1][0]}[scale=.4]
    \simeq^{\fg}
    t^{-1}
    \tikzpic{\webid}[scale=.4]
    \langle -\frac{1}{2},\frac{1}{2}\rangle
    \simeq^{\fg}
    \tikzpic{\webncr[0][0]\webpcr[1][0]}[scale=.4]
    \;.
  \end{gather*}
  We focus on the first isomorphism, the other one being the same up to reordering direct sums.
  Using a split exact sequence in the spirit of \cref{lem:web_defining_relations_equivariance}, we can fit the left-hand side in a sequence which is split exact at the two middle chain complexes:
  \begin{center}
    \begin{tikzpicture}
      \def\vspc{2.5}\def\hspc{6}
      \def\webxscl{.4}\def\webyscl{.3}
      %
      % VERTICES
      \node (Z2) at (\hspc,2*\vspc) {$0$};
      \node (A1) at (0,\vspc) {$0$};
      \node (A2) at (\hspc,\vspc) {
        $\tikzpic{\webid}[xscale=\webxscl,yscale=\webyscl][(0,.5*\webyscl)]
        \;\langle -\frac{1}{2},\frac{1}{2} \rangle$
      };
      \node (A3) at (2*\hspc,\vspc) {$0$};
      \node (B1) at (0,0) {
        $\tikzpic{\websRoundMarkedB\webm[1][0]}[xscale=\webxscl,yscale=\webyscl][(0,.5*\webyscl)]
        \;\langle -1,1 \rangle$
      };
      \node (B2) at (\hspc,0) {
        $\tikzpic{\webid}[xscale=\webxscl,yscale=\webyscl][(0,.5*\webyscl)]
        \;\langle -\frac{1}{2},\frac{1}{2} \rangle
        \;\oplus\;
        \tikzpic{
          \websRoundMarkedB\webm[1][0]\websRoundMarkedB[2][0]\webm[3][0]
        }[xscale=\webxscl,yscale=\webyscl][(0,.5*\webyscl)]
        \;\langle -\frac{1}{2},\frac{1}{2} \rangle$
      };
      \node (B3) at (2*\hspc,0) {
        $\tikzpic{\websRoundMarkedB\webm[1][0]}[xscale=\webxscl,yscale=\webyscl][(0,.5*\webyscl)]$
      };
      \node (C1) at (0,-\vspc) {
        $\tikzpic{\websRoundMarkedB\webm[1][0]}[xscale=\webxscl,yscale=\webyscl][(0,.5*\webyscl)]
        \;\langle -1,1 \rangle$
      };
      \node (C2) at (\hspc,-\vspc) {
        $\tikzpic{
          \websRoundMarkedB\webm[1][0]
        }[xscale=\webxscl,yscale=\webyscl][(0,.5*\webyscl)]
        \;\langle -1,1 \rangle
        \;\oplus\;
        \tikzpic{
          \websRoundMarkedB\webm[1][0]
        }[xscale=\webxscl,yscale=\webyscl][(0,.5*\webyscl)]$
      };
      \node (C3) at (2*\hspc,-\vspc) {
        $\tikzpic{
          \websRoundMarkedB\webm[1][0]
        }[xscale=\webxscl,yscale=\webyscl][(0,.5*\webyscl)]$
      };
      %
      % EDGES
      % horizontal
      \draw[->] (A1) to (A2);
      \draw[->] (A2) to (A3);
      \draw[->] (B1) to (B2);
      \node[above] at (.35*\hspc,0) {\scriptsize
        $\left(\!\begin{array}{c}
          \tikzpic{\funzip}[scale=.2] \\
          \tikzpic{\frst[0][0]\flst[1][0]\fzip[2][0]}[scale=.2]
        \end{array}\!\right)$
      };
      \draw[->] (B2) to (B3);
      \node[above] at (1.7*\hspc,0) {\scriptsize
        $\left(\!\begin{array}{cc}
          \tikzpic{\fzip}[scale=.2] &
          -Z\;\tikzpic{\funzip[0][0]\frst[2][0]\flst[3][0]}[scale=.2] 
        \end{array}\!\right)$
      };
      \draw[->] (C1) to (C2);
      \node[above] at (.5*\hspc,-\vspc) {$
      \left(\!\begin{array}{c}
        1 \\ 0
      \end{array}\!\right)$};
      \draw[->] (C2) to (C3);
      \node[above] at (1.5*\hspc,-\vspc) {$
      \left(\!\begin{array}{cc}
        0 & 1
      \end{array}\!\right)$};
      %
      % vertical
      \draw[<-] (Z2) to (A2);
      \draw[<-] ($(A2.south)+(.1,0)$) to ($(B2.north)+(.1,0)$);
      \draw[->,dashed] ($(A2.south)-(.1,0)$) to ($(B2.north)-(.1,0)$);
      \node[right] at (1*\hspc,.5*\vspc) {\scriptsize
        $\left(\begin{array}{cc}
          1 & -Z^{-3}\;\tikzpic{
          \fcap[0][0][1]
          \draw[foamdraw1] (-.5,0) to (-.5,1);
          \draw[foamdraw1] (1.5,0) to (1.5,1);
          \begin{scope}[scale=2]
            \funzip[-.25][.5]
          \end{scope}
          }[scale=.2]
        \end{array}\right)$
      };
      \node[left] at (1*\hspc,.5*\vspc) {\scriptsize
        $\left(\begin{array}{c}
          1\\
          \tikzpic{
          \fcap[0][0][1]
          \draw[foamdraw1] (-.5,0) to (-.5,1);
          \draw[foamdraw1] (1.5,0) to (1.5,1);
          \begin{scope}[scale=2]
            \funzip[-.25][.5]
          \end{scope}
          }[scale=.2,yscale=-1]
        \end{array}\right)$
      };
      \draw[<-] ($(B1.south)+(.1,0)$) to ($(C1.north)+(.1,0)$);
      \draw[->,dashed] ($(B1.south)-(.1,0)$) to ($(C1.north)-(.1,0)$);
      \node[right] at (0,-.5*\vspc) {$\id$};
      \node[left] at (0,-.5*\vspc) {$\id$};
      \draw[<-] ($(B2.south)+(.1,0)$) to ($(C2.north)+(.1,0)$);
      \draw[->,dashed] ($(B2.south)-(.1,0)$) to ($(C2.north)-(.1,0)$);
      \node[right] at (1*\hspc,-.5*\vspc) {\scriptsize
        $\left(\!\begin{array}{cc}
          \tikzpic{\funzip}[scale=.2] & 0 \\
          \tikzpic{\frst[0][0]\flst[1][0]\fzip[2][0]}[scale=.2] 
          & -Z^{-3}\;\tikzpic{
            \fcap[0][0][1]
            \draw[foamdraw1] (-.5,0) to (-.5,1);
            \draw[foamdraw1] (1.5,0) to (1.5,1);
          }[scale=.2,yscale=-1]
        \end{array}\!\right)$
      };
      \node[left] at (1*\hspc,-.5*\vspc) {\scriptsize
        $\left(\!\begin{array}{cc}
          0 & \tikzpic{
            \fcap[0][0][1]
            \draw[foamdraw1] (-.5,0) to (-.5,1);
            \draw[foamdraw1] (1.5,0) to (1.5,1);
          }[scale=.2]
          \\
          \tikzpic{\fzip}[scale=.2]
          &
          -Z\;\tikzpic{\funzip[0][0]\frst[2][0]\flst[3][0]}[scale=.2] 
        \end{array}\!\right)$
      };
      \draw[<-] ($(B3.south)+(.1,0)$) to ($(C3.north)+(.1,0)$);
      \draw[->,dashed] ($(B3.south)-(.1,0)$) to ($(C3.north)-(.1,0)$);
      \node[right] at (2*\hspc,-.5*\vspc) {$\id$};
      \node[left] at (2*\hspc,-.5*\vspc) {$\id$};
    \end{tikzpicture}
  \end{center}
  % \ls{split thanks to super 4tu (see notes)}
  We omitted the homological degree: the middle column is in homological degree zero.
  The top chain complex is zero while the bottom chain complex is the cone of an identity: we are in the situation of \cref{lem:abstract_homotopy_equivalence}.
  Moreover, the middle chain morphism is $\fg$-equivariant thanks to \cref{lem:crossing_complex_equivariant}, so that it defines a $\fg$-equivariant homotopy equivalence.
\end{proof}

\begin{lemma}
  \label{lem:invariance_Reidemeister_III}
  Let $D$ be a marked tangled web.
  In the relative homotopy category $\cK^\fg(\fg\foam^{\greenmarking})$, the object $\fg\glKh(D)$ is invariant under Reidemeister III moves, up to isomorphism.
\end{lemma}

\begin{proof}
    The proof is an equivariant version of the proof in \cite{SV_OddKhovanovHomology_2023}, following the general strategy of Bar-Natan \cite{Bar-Natan_KhovanovsHomologyTangles_2005}.
\end{proof}

%%%%%%%%%%%%%%%%%%%%%%%%%%%%%%
%%%%%%%%%%%%%%%%%%%%%%%%%%%%%%
%%%%%%%%%%%%%%%%%%%%%%%%%%%%%%
\subsection{Marking slide}
\label{subsec:marking_slide}

In this subsection, we prove the following ``marking slide'' lemma:

\begin{lemma}[marking slide lemma]
  \label{lem:dot_slide_lemma}
  Let $\omega=(\alpha,\beta_1,\beta_2)$ be a generic local twist.
  The identity chain map induces an isomorphism in the relative homotopy category $\cK^{\gloo}(\sfoam^{\greenmarking})$:
  \begin{gather*}
    \tikzpic{
      \webncr
      \node[green_mark] (B) at (.3,.15) {};
      \node[below={-1pt} of B] {\scriptsize $\omega$};
    }[scale=.7][(0,.5*.7)]
    \simeq^{\gloo}
    \tikzpic{
      \webncr
      \node[green_mark] (B) at (1-.3,1-.15) {};
      \node[above={-1pt} of B] {\scriptsize $\omega$};
    }[scale=.7][(0,.5*.7)]
    \;.
  \end{gather*}
  If one considers $\sfoam'$ (see \cref{rem:variant_graded_gl2_foams}) instead, then the roles of the overcrossing and the undercrossing are swapped.
\end{lemma}

\begin{remark}
  \label{rem:marking_cannot_undercross}
  By considering the example 
  \begin{gather*}
  \tikzpic{
  \webpcr[0][0]\webncr[1][0]
  \webs[2][1]\webs[2][-1]
  \webm[-1][1]\webm[-1][-1]
  \draw[web1] (0,2) to (2,2);
  \draw[web1] (0,-1) to (2,-1);
  \node[green_mark] (B) at (1,-1) {};
  \node[above={-1pt} of B] {\scriptsize $-\omega$};
  \node[green_mark] (C) at (1,1) {};
  \node[above={-1pt} of C] {\scriptsize $\omega$};
  }[scale=.5]
  \;,
\end{gather*}
  one can check that indeed, the analogue statement for the other crossing does not hold:
  \begin{gather*}
    \tikzpic{
      \webpcr
      \node[green_mark] (B) at (.3,.15) {};
      \node[below={-1pt} of B] {\scriptsize $\omega$};
    }[scale=.7][(0,.5*.7)]
    \not\simeq^{\gloo}
    \tikzpic{
      \webpcr
      \node[green_mark] (B) at (1-.3,1-.15) {};
      \node[above={-1pt} of B] {\scriptsize $\omega$};
    }[scale=.7][(0,.5*.7)]
    \;,
  \end{gather*}
  and vice-versa if one considers $\sfoam'$. Indeed, if both dot slides did hold, then we would have 
  \begin{gather*}
  \tikzpic{
  %\webpcr[0][0]\webncr[1][0]
  \webs[2][1]\webs[2][-1]
  \webm[-1][1]\webm[-1][-1]
  \draw[web1] (0,2) to (2,2);
  \draw[web1] (0,-1) to (2,-1);
  \draw[web1] (0,1) to (2,1);
  \draw[web1] (0,0) to (2,0);
  }[scale=.5][(0,.5*.5)]
  \simeq^{\gloo}
  \tikzpic{
  \webs[2][1]\webs[2][-1]
  \webm[-1][1]\webm[-1][-1]
  \draw[web1] (0,2) to (2,2);
  \draw[web1] (0,-1) to (2,-1);
  \draw[web1] (0,1) to (2,1);
  \draw[web1] (0,0) to (2,0);
  \node[green_mark] (B) at (1,-1) {};
  \node[above={-1pt} of B] {\scriptsize $-2\omega$};
  \node[green_mark] (C) at (1,2) {};
  \node[above={-1pt} of C] {\scriptsize $2\omega$};
  }[scale=.5][(0,.5*.5)]
  \;.
\end{gather*}
\end{remark}

Before giving the proof, we discuss some consequences.

\begin{lemma}
  \label{lem:csq_dot_slide_lemma}
  Let $\omega=(\alpha,\beta_1,\beta_2)$ be a generic local twist.
  For each of the following cases, the identity chain map induces an isomorphism in the relative homotopy category $\cK^{\gloo}(\sfoam^{\greenmarking})$:
  \begin{gather*}
    \tikzpic{
      \webpcr
      \node[green_mark] (B) at (.3,.15) {};
      \node[below={-1pt} of B] {\scriptsize $\omega$};
    }[scale=.7][(0,.5*.7)]
    \simeq^{\gloo}
    \tikzpic{
      \webpcr
      \node[green_mark] (B) at (1-.3,1-.15) {};
      \node[above={-1pt} of B] {\scriptsize $-\omega$};
      \node[green_mark] (C) at (1-.3,.15) {};
      \node[below={-1pt} of C] {\scriptsize $2\omega$};
    }[scale=.7][(0,.5*.7)]
    \;,\quad
    \tikzpic{
      \webncr
      \node[green_mark] (B) at (.3,1-.15) {};
      \node[above={-1pt} of B] {\scriptsize $\omega$};
    }[scale=.7][(0,.5*.7)]
    \simeq^{\gloo}
    \tikzpic{
      \webncr
      \node[green_mark] (B) at (1-.3,.15) {};
      \node[below={-1pt} of B] {\scriptsize $-\omega$};
      \node[green_mark] (C) at (1-.3,1-.15) {};
      \node[above={-1pt} of C] {\scriptsize $2\omega$};
    }[scale=.7][(0,.5*.7)]
    \quad\an\quad
    \tikzpic{
      \webpcr
      \node[green_mark] (B) at (.3,1-.15) {};
      \node[above={-1pt} of B] {\scriptsize $\omega$};
    }[scale=.7][(0,.5*.7)]
    \simeq^{\gloo}
    \tikzpic{
      \webpcr
      \node[green_mark] (B) at (1-.3,.15) {};
      \node[below={-1pt} of B] {\scriptsize $\omega$};
    }[scale=.7][(0,.5*.7)]
    \;.
  \end{gather*}
  If one considers $\sfoam'$ instead, then the roles of the overcrossing and the undercrossing are swapped.
\end{lemma}

\begin{proof}
  The follows from \cref{lem:dot_slide_lemma} using invariance under planar isotopies \cref{lem:invariance_planar_isotopies} and the fact that in general, the identity chain map induces an isomorphism in $\cK^\gloo(\sfoam)$:
  \begin{gather*}
    \tikzpic{
      \webncr
      \node[green_mark] (B) at (.3,1-.15) {};
      \node[above={-1pt} of B] {\scriptsize $-\omega$};
      \node[green_mark] (B) at (.3,.15) {};
      \node[below={-1pt} of B] {\scriptsize $\omega$};
    }[scale=.7][(0,.5*.7)]
    \simeq^\fg
    \tikzpic{
      \webncr
      \node[green_mark] (B) at (1-.3,.15) {};
      \node[below={-1pt} of B] {\scriptsize $\omega$};
      \node[green_mark] (B) at (1-.3,1-.15) {};
      \node[above={-1pt} of B] {\scriptsize $-\omega$};
    }[scale=.7][(0,.5*.7)]
  \end{gather*}
  This concludes.
\end{proof}

The remainder of this subsection is devoted to the proof of \cref{lem:dot_slide_lemma}.
The proof is inspired by the proof of the analogue result in \cite[lemma~4.4]{QRS+_SymmetriesEquivariantKhovanovRozansky_2023}; see also \cite[Lemma~5.2]{Roz_$mathfraksl_2$ActionLink_2023}. Their proof originates from \cite{KR_PositiveHalfWitt_2016}.

We begin with an outline.
Write $D_\bullet$ for the complex associated to the crossing, and ${}_\omega D_\bullet$ (resp.\ $D_\bullet^\omega$) for the complex with the additional $\omega$-twist at the top left (resp.\ bottom right).
We aim to show that ${}_\omega D_\bullet\simeq^\fg D_\bullet^\omega$.
The main idea is to add a circle to the web, and move the $\omega$-twist to that circle. More formally, we define in \cref{lem:web_resolution_twist} a partial resolution of a generic $\omega$-twist via $\omega$-twisted circles.
Applying this to ${}_\omega D_\bullet$ gives another complex ${}_\omega C_\bullet$ together with a $\gloo$-equivariant homotopy equivalence ${}_\omega C_\bullet\to {}_\omega D_\bullet$; similarly, we find a $\gloo$-equivariant homotopy equivalence $C_\bullet^\omega\to D_\bullet^\omega$.
We are then able to give an $\gloo$-equivariant isomorphism between ${}_\omega C_\bullet$ and $C_\bullet^\omega$.
This leads to a zigzag of $\gloo$-equivariant homotopies between ${}_\omega D_\bullet$ and $D_\bullet^\omega$, and hence an isomorphism in the relative homotopy category.

\medbreak

We now give the details, beginning with some preliminary definitions.
For a web $W$, we shall write $\Phi(W)$ the web obtained by extending $W$ forward with a marked circle, as shown below.
\begin{gather*}
  \Phi(W)
  \coloneqq\quad
  \tikzpic{
    \node[anchor=mid,align=center] at (.4,1.7) {\footnotesize $t(W)$};
    \node[anchor=mid,align=center] at (1.5,1.7) {$W$};
    \node[anchor=mid,align=center] at (2.6,1.7) {\footnotesize $s(W)$};
    \draw[web1,dashed] (1,0) to (2,0);
    \draw[web1,dashed] (1,1) to (2,1);
    \webmRoundMarkedB
    \webs[2][0]
    \draw[dotted] (1,-.3) to (1,2);
    \draw[dotted] (2,-.3) to (2,2);
  }
  \qquad\an\qquad
  \Phi_\omega(W)
  \coloneqq\quad
  \tikzpic{
    \node[anchor=mid,align=center] at (.4,1.7) {\footnotesize $t(W)$};
    \node[anchor=mid,align=center] at (1.5,1.7) {$W$};
    \node[anchor=mid,align=center] at (2.6,1.7) {\footnotesize $s(W)$};
    \draw[web1,dashed] (1,0) to (2,0);
    \draw[web1,dashed] (1,1) to (2,1);
    \webmRoundMarkedB
    \websMarkedB[2][0][$\omega$]
    \draw[dotted] (1,-.3) to (1,2);
    \draw[dotted] (2,-.3) to (2,2);
  }\;.
\end{gather*}
More formally, we define a family of superfunctors $\Phi\colon\sfoam^{\greenmarking}(n,m)\to\sfoam^{\greenmarking}(n,m)$, defined on objects as
\[\Phi(W)\coloneqq(M^{\roundmarking}\mathrel{\square}\id_W)\mathrel{\square} (\id_{(1,1)}\otimes W)\otimes (S\mathrel{\square}\id_W),\]
where $M^{\roundmarking}$ is a merge web with an extra marking $\roundmarking[2]$ and $S$ is a split web, and $\square$ is the front-back composition from \cref{rem:monoidal-2-categorical-structure}.
We also define the variant $\Phi_\omega(W)$ where the circle carries an extra local marking $\omega$.

Consider an identity web $W\in\sfoam^{\greenmarking}$ with a distinguished strand $i$:
\begin{gather*}
  W = \tikzpic{
    \node at (1,.8){};
    \node at (1,-.8){};
    \node at (1,.5) {$\vdots$};
    \draw[web1] (0,0) to (2,0);
    \node at (1,-.3) {$\vdots$};
    \node[right] at (2,0) {\footnotesize $i$};
  }
\end{gather*}
Below we consider $\Phi(W)$ and $\Phi_\omega(W)$, using the color blue ($\textdot$) for the label of the added circle and the color red ($\textdot[2]$) for the label of the distinguished strand $i$ in $W$.

\begin{lemma}
  \label{lem:web_resolution_twist}
  Let $\omega=(\alpha,\beta_1,\beta_2)$ be a generic local  marking.
  The following is a sequence in $\sfoam^{\greenmarking}$, split exact at the two middle vertices, and with each forward (plain) arrow being $\gloo$-equivariant:
  \begin{center} 
  \begin{tikzpicture}
    \def\hsh{3}
    \def\wxscl{.5}
    \def\wyscl{.35}
    % \node (L) at (-\hsh+1*\wxscl,.5) {$\ldots$};
    % \node at (-\hsh+3*\wxscl,1.5*\wyscl) {\footnotesize $\langle -3,3\rangle$};
    % %
    % \begin{scope}[xscale=\wxscl,yscale=\wyscl]
    %   \node[shape=rectangle,anchor=center] at (1,3.7) {$\vdots$};
    %   \draw[web1] (0,2.5) to (2,2.5);
    %   \node[shape=rectangle,anchor=center] at (1,2) {$\vdots$};
    %   \webmRoundMarkedB
    %   \websMarkedB[1][0][$\omega$]
    %   \node at (3,1.5) {\footnotesize $\langle -2,2\rangle$};
    % \end{scope}
    %
    \begin{scope}[shift={(\hsh,0)},xscale=\wxscl,yscale=\wyscl]
      \node[shape=rectangle,anchor=center] at (1,3.7) {$\vdots$};
      \draw[web1] (0,2.5) to (2,2.5);
      \node[shape=rectangle,anchor=center] at (1,2) {$\vdots$};
      \webmRoundMarkedB
      \websMarkedB[1][0][$\omega$]
      \node at (3,1.5) {\footnotesize $\langle -\frac{3}{2},\frac{3}{2}\rangle$};
    \end{scope}
    \begin{scope}[shift={(2*\hsh,0)},xscale=\wxscl,yscale=\wyscl]
      \node[shape=rectangle,anchor=center] at (1,3.7) {$\vdots$};
      \draw[web1] (0,2.5) to (2,2.5);
      \node[shape=rectangle,anchor=center] at (1,2) {$\vdots$};
      \webmRoundMarkedB
      \websMarkedB[1][0][$\omega$]
      \node at (3,1.5) {\footnotesize $\langle -\frac{1}{2},\frac{1}{2}\rangle$};
    \end{scope}
    \begin{scope}[shift={(3*\hsh,0)},xscale=\wxscl,yscale=\wyscl]
      \node[shape=rectangle,anchor=center] at (1,3.7) {$\vdots$};
      \draw[web1] (0,2.5) to (2,2.5);
      \node[green_mark=2.5pt] (M) at (1,2.5) {};
      \node[below right={-1pt} of M] {\scriptsize $\omega$};
      \node[shape=rectangle,anchor=center] at (1,1.8) {$\vdots$};
      \draw[web2] (0,.5) to (2,.5);
    \end{scope}
    \node (R) at (3.7*\hsh,.5) {$0$};
    % Arrow
    % \draw[->,thick] (-\hsh+2*\wxscl,1.5) to[out=45,in=135] (0,1.5);
    % \draw[->,thick] (2*\wxscl,1.5) to[out=45,in=135] (\hsh,1.5);
    \draw[->,thick] (\hsh+2*\wxscl,1.5) to[out=45,in=135] (2*\hsh,1.5);
    \draw[->,thick] (2*\hsh+2*\wxscl,1.5) to[out=45,in=135] (3*\hsh,1.5);
    \draw[->,thick] (3*\hsh+2*\wxscl+.2,.5) -- (R);
    %
    % \node[above] at (-.5*\hsh+\wxscl,2) {%
    %   $\tikzpic{
    %     \flst\frst[1][0]\fdot[.5][.5]
    %   }[scale=.5,xscale=.6]
    %   -
    %   \tikzpic{
    %     \flst\frst[1][0]\fdot[.5][.5][2]
    %   }[scale=.5,xscale=.6]
    %   $};
    % \node[above] at (.5*\hsh+\wxscl,2) {%
    %   $\tikzpic{
    %     \flst\frst[1][0]\fdot[.5][.5]
    %   }[scale=.5,xscale=.6]
    %   -
    %   \tikzpic{
    %     \flst\frst[1][0]\fdot[.5][.5][2]
    %   }[scale=.5,xscale=.6]
    %   $};
    \node[above] at (1.5*\hsh+\wxscl,2) {%
      $\tikzpic{
        \flst\frst[1][0]\fdot[.5][.5]
      }[scale=.5,xscale=.6]
      -
      \tikzpic{
        \flst\frst[1][0]\fdot[.5][.5][2]
      }[scale=.5,xscale=.6]
      $};
    \node[above] at (2.5*\hsh+\wxscl,2) {
      $\tikzpic{
        \fcap\fdot[.5][.3]
      }[scale=.5,xscale=.6]
      -
      \tikzpic{
        \fcap\fdot[.5][.3][2]
      }[scale=.5,xscale=.6]
      $};
    %
    % \draw[<-,thick,dashed] (-\hsh+2*\wxscl,-.5) to[out=-45,in=-135] (0,-.5);
    % \draw[<-,thick,dashed] (2*\wxscl,-.5) to[out=-45,in=-135] (\hsh,-.5);
    \draw[<-,thick,dashed] (\hsh+2*\wxscl,-.5) to[out=-45,in=-135] (2*\hsh,-.5);
    \draw[<-,thick,dashed] (2*\hsh+2*\wxscl,-.5) to[out=-45,in=-135] (3*\hsh,-.5);
    %
    % \node[below=15pt] at (-.5*\hsh+\wxscl,-.5) {%
    %   $\tikzpic{
    %     \fcap\fcup[0][1]
    %   }[scale=.3,xscale=1]
    %   $};
    % \node[below=15pt] at (.5*\hsh+\wxscl,-.5) {%
    %   $\tikzpic{
    %     \fcap\fcup[0][1]
    %   }[scale=.3,xscale=1]
    %   $};
    \node[below=15pt] at (1.5*\hsh+\wxscl,-.5) {%
      $\tikzpic{
        \fcap\fcup[0][1]
      }[scale=.3,xscale=1]
      $};
    \node[below=15pt] at (2.5*\hsh+\wxscl,-.5) {%
      $\tikzpic{
        \fcup
      }[scale=.3,xscale=1]
      $};
  \end{tikzpicture}
  \end{center}
\end{lemma}

\begin{proof}
  The fact that the sequence split is a direct computation.
  Equivariance with respect to $\lieh_1$ and $\lieh_2$ can be checked using \cref{lem:h-equivariance}.
  Equivariance with respect to $\liee$ and $\lief$ follows from
  \cref{lem:crossing_complex_equivariant} and respectively \cref{lem:e-equivariance} and \cref{lem:f-equivariance}.
\end{proof}

Using this partial resolution, we construct a partial resolution in $\underline{\Ch}(\sfoam^{\greenmarking})$ of ${}_\omega D_\bullet$, as pictured in \cref{fig:proof_dot_slide}.
The fact that the complexes are $\gloo$-equivariant was checked already in \cref{lem:web_resolution_twist}. Note that up to scalar, we have ${}_\omega C_\bullet\cong\Cone(F)$ for $F\colon\Phi_\omega(C_\bullet)\to\Phi_\omega(C_\bullet)$ a $\gloo$-equivariant chain map consisting of dots.
Note moreover that $P_\bullet=\Cone(\id_{\Phi_\omega(D_\bullet)}\langle-\frac{3}{2},\frac{3}{2}\rangle)$. In particular, the complex $P_\bullet$ is contractible.
We are in the situation of \cref{lem:abstract_homotopy_equivalence}, and conclude that the chain map ${}_\omega C_\bullet\to{}_\omega D_\bullet$ is a ($\gloo$-equivariant) homotopy equivalence.

Use the colour brown ($\textdot[3]$) for  the label of the backmost strand amongst to two strands involved in $D_\bullet$.
The very same argument applies to $D_\bullet^\omega$, only replacing red dots ($\textdot[{2}]$) with brown dots ($\textdot[{3}]$), and swapping the dots from left to right in the two vertical arrows in the middle of the diagram.
We get a ($\gloo$-equivariant) homotopy equivalence $C_\bullet^\omega\to D_\bullet^\omega$.

\begin{figure}[p]
  \begin{tikzpicture}
    \def\hsh{5}\def\vsh{5}\def\hoplussh{.5}\def\voplussh{1}
    \def\wxscl{.35}\def\wyscl{.35}
    %%%%%%%%%%%%%%%%%%%%%%%%%%%%%%
    % SOME DIAGRAMS
    \def\tempwebA{
      \webid[0][1.5][1]\websRoundMarkedB[1][1.5]\webm[2][1.5]\webid[3][1.5][1]
      \webmRoundMarkedB\webid[1][0][1]\webid[2][0][1]\websMarkedB[3][0][$\omega$]
    }
    \def\tempwebB{
      \webid[0][1.5][1]\webid[1][1.5]\webid[3][1.5][1]
      \webmRoundMarkedB\webid[1][0][1]\webid[2][0][1]\websMarkedB[3][0][$\omega$]
    }
    \newcommand{\tempfoamid}{
    \tikzpic{
        \flst\frst[1][0]\fdot[.5][.5]
      }[scale=.7,xscale=.5]
      -
      \tikzpic{
        \flst\frst[1][0]\fdot[.5][.5][2]
      }[scale=.7,xscale=.5]
    }
    \newcommand{\tempfoamdbid}{
      \tikzpic{
        \draw[fill,foamshade2] (1-.3,0) rectangle (1+.3,2);
        \draw[foamdraw2] (1-.3,0) to (1-.3,2);
        \draw[foamdraw2] (1+.3,0) to (1+.3,2);
        \draw[fill,foamshade1] (-.5,0) rectangle (0,2);
        \draw[foamdraw1] (0,0) to (0,2);
        \draw[fill,foamshade1] (2.5,0) rectangle (2,2);
        \draw[foamdraw1] (2,0) to (2,2);
        \fdot[.4][1][1]
      }[scale=.35]
      -
      \tikzpic{
        \draw[fill,foamshade2] (1-.3,0) rectangle (1+.3,2);
        \draw[foamdraw2] (1-.3,0) to (1-.3,2);
        \draw[foamdraw2] (1+.3,0) to (1+.3,2);
        \draw[fill,foamshade1] (-.5,0) rectangle (0,2);
        \draw[foamdraw1] (0,0) to (0,2);
        \draw[fill,foamshade1] (2.5,0) rectangle (2,2);
        \draw[foamdraw1] (2,0) to (2,2);
        \fdot[.4][1][2]
      }[scale=.35]
    }
    \newcommand{\tempfoamdbunzip}{
      \tikzpic{
        \funzip[.5][0][2]
        \draw[fill,foamshade1] (-.5,0) rectangle (0,2);
        \draw[foamdraw1] (0,0) to (0,2);
        \draw[fill,foamshade1] (2.5,0) rectangle (2,2);
        \draw[foamdraw1] (2,0) to (2,2);
      }[scale=.35]
    }
    %%%%%%%%%%%%%%%%%%%%%%%%%%%%%%
    % VERTICES
    %
    % top layer
    \node[] (T0) at (\hsh-\hoplussh,\vsh) {$\tikzpic{
      \webid[0][1.5][1]\websRoundMarkedB[1][1.5]\webm[2][1.5]\webid[3][1.5][1]
      \node[green_mark] (M) at (.5,1.5) {};
      \node[below={-3pt} of M] {\scriptsize $\omega$};
      \draw[web2] (0,.5) to (4,.5);
    }[xscale=\wxscl,yscale=\wyscl]$};
    \node[] (R0) at (2*\hsh,\vsh) {$\tikzpic{
      \webid[0][1.5][4]
      \node[green_mark] (M) at (.5,1.5) {};
      \node[below={-3pt} of M] {\scriptsize $\omega$};
      \draw[web2] (0,.5) to (4,.5);
    }[xscale=\wxscl,yscale=\wyscl]$};
    %
    % middle layer
    \node[] (L1) at (0,0) {$\tikzpic{\tempwebA}[xscale=\wxscl,yscale=\wyscl]$};
    \node[] (T1) at (\hsh-\hoplussh,\voplussh) {$\tikzpic{\tempwebA}[xscale=\wxscl,yscale=\wyscl]$};
    \node[anchor=mid] at (\hsh,.1) {$\bigoplus$};
    \node[] (B1) at (\hsh+\hoplussh,-\voplussh) {$\tikzpic{\tempwebB}[xscale=\wxscl,yscale=\wyscl]$};
    \node[] (R1) at (2*\hsh,0) {$\tikzpic{\tempwebB}[xscale=\wxscl,yscale=\wyscl]$};
    %
    % bottom layer
    \node[] (L2) at (0,-\vsh) {$\tikzpic{\tempwebA}[xscale=\wxscl,yscale=\wyscl]$};
    \node[] (T2) at (\hsh-\hoplussh,\voplussh-\vsh) {$\tikzpic{\tempwebA}[xscale=\wxscl,yscale=\wyscl]$};
    \node[anchor=mid] at (\hsh,.1-\vsh) {$\bigoplus$};
    \node[] (B2) at (\hsh+\hoplussh,-\voplussh-\vsh) {$\tikzpic{\tempwebB}[xscale=\wxscl,yscale=\wyscl]$};
    \node[] (R2) at (2*\hsh,-\vsh) {$\tikzpic{\tempwebB}[xscale=\wxscl,yscale=\wyscl]$};
    %
    %%%%%%%%%%%%%%%%%%%%%%%%%%%%%%
    % EDGES
    % top layer
    \draw[->] (T0) to[out=0,in=180] (R0);
    \node[above] at (1.5*\hsh-.5*\hoplussh,\vsh) {\tikzpic{\funzip[0][0][2]}[scale=.7]};
    %
    % mid-top
    \draw[->] (T1) to (T0);
    \draw[->] (R1) to (R0);
    \node[left=15pt] at (\hsh,.6*\vsh) {$
      \tikzpic{
        \clip (-.5,0) rectangle (2.5,2);
        \begin{scope}[scale=2]
          \fcap
        \end{scope}
        \draw[fill,foamshade2] (1-.3,0) rectangle (1+.3,2);
        \draw[foamdraw2] (1-.3,0) to (1-.3,2);
        \draw[foamdraw2] (1+.3,0) to (1+.3,2);
        \fdot[.4][.5][1]
      }[xscale=1,scale=.4]
      -
      \tikzpic{
        \clip (-.5,0) rectangle (2.5,2);
        \begin{scope}[scale=2]
          \fcap
        \end{scope}
        \draw[fill,foamshade2] (1-.3,0) rectangle (1+.3,2);
        \draw[foamdraw2] (1-.3,0) to (1-.3,2);
        \draw[foamdraw2] (1+.3,0) to (1+.3,2);
        \fdot[.4][.5][2]
      }[xscale=1,scale=.4]
    $};
    \node[right] at (2*\hsh,.5*\vsh) {$
      \tikzpic{
        \fcap\fdot[.5][.3]
      }[scale=.7,xscale=.5]
      -
      \tikzpic{
        \fcap\fdot[.5][.3][2]
      }[scale=.7,xscale=.5]
    $};
    %
    % T2 to T1
    \draw[->] ($(T2.north)+(-.2,0)$) to ($(T1.south)+(-.2,0)$);
    %
    % middle layer
    \draw[->] (L1) to[out=20,in=180] (T1);
    \draw[->] (T1) to[out=0,in=160] (R1);
    \draw[->] (B1) to[out=0,in=-160] (R1);
    \draw[{preaction={draw,white,line width=5pt}},->] (L1) to[out=-20,in=180] (B1);
    \node[fill=white,inner sep=1pt] at (.4*\hsh-.5*\hoplussh,.25*\vsh) {$
      \tempfoamdbid$};
    \node[fill=white,inner sep=1pt] at (.4*\hsh+.5*\hoplussh,-.2*\vsh) {$-\;
      \tempfoamdbunzip$};
    \node[fill=white,inner sep=1pt] at (1.5*\hsh-.5*\hoplussh,.25*\vsh) {$-\;
      \tempfoamdbunzip$};
    \node[fill=white,inner sep=1pt] at (1.6*\hsh+.5*\hoplussh,-.2*\vsh) {$\tempfoamid$};
    %
    % bot-mid
    \draw[->] (L2) to (L1);
    \draw[->] (R2) to (R1);
    \node[right] at (2*\hsh,-.5*\vsh) {$\tempfoamid$};
    \node[left,fill=white,inner sep=1pt] at (1*\hsh,-.45*\vsh) {$\tempfoamdbid$};
    \node[left] at (0,-.5*\vsh) {$\id$};
    %
    % bottom layer
    \draw[->] (L2) to[out=20,in=180] (T2);
    \draw[->,{preaction={draw,white,line width=5pt}}] (L2) to[out=-20,in=180] (B2);
    \draw[->] (T2) to[out=0,in=160] (R2);
    \draw[->] (B2) to[out=0,in=-160] (R2);
    %
    % SHIFTS
    \node[right={-3pt} of T0,fill=white,inner sep=1pt] {\scriptsize $\langle -\frac{1}{2},\frac{1}{2}\rangle$};
    \node[right={-9pt} of L1,fill=white,inner sep=1pt] {\scriptsize $\langle -2,2\rangle$};
    \node[right={-12pt} of T1,fill=white,inner sep=1pt] {\scriptsize $\langle -1,1\rangle$};
    \node[right={-12pt} of B1,fill=white,inner sep=1pt] {\scriptsize $\langle -\frac{3}{2},\frac{3}{2}\rangle$};
    \node[right={-9pt} of R1,fill=white,inner sep=1pt] {\scriptsize $\langle -\frac{1}{2},\frac{1}{2}\rangle$};
    \node[right={-12pt} of L2] {\scriptsize $\langle -2,2\rangle$};
    \node[right={-9pt} of T2,fill=white,inner sep=1pt] {\scriptsize $\langle -2,2\rangle$};
    \node[right={-12pt} of B2,fill=white,inner sep=1pt] {\scriptsize $\langle -\frac{3}{2},\frac{3}{2}\rangle$};
    \node[right={-9pt} of R2] {\scriptsize $\langle -\frac{3}{2},\frac{3}{2}\rangle$};
    \node[fill=white,inner sep=1pt] at (.4*\hsh-.5*\hoplussh,.2*\vsh-\vsh) {$
      \id$};
    \node[fill=white,inner sep=1pt] at (.4*\hsh+.5*\hoplussh,-.25*\vsh-\vsh) {$
      -\;\tempfoamdbunzip$};
    \node[fill=white,inner sep=1pt] at (1.5*\hsh-.5*\hoplussh,.25*\vsh-\vsh) {$
      \tempfoamdbunzip$};
    \node[fill=white,inner sep=1pt] at (1.6*\hsh+.5*\hoplussh,-.2*\vsh-\vsh) {$\id$};
    %
    % B2 to B1
    \draw[{preaction={draw,white,line width=5pt}},->] ($(B2.north)+(.2,0)$) to ($(B1.south)+(.2,0)$);
    \node[right=-8pt] at (\hsh+\voplussh,-.4*\vsh-\voplussh) {$\id$};
    %
    % below-bot
    % \draw[->] ($(L2.south)-(0,.3*\vsh)$) to (L2);
    % \draw[dashed] ($(L2.south)-(0,.5*\vsh)$) to[] ($(L2.south)-(0,.3*\vsh)$);
    % \draw[->] ($(B2.south)+(.2,-.3*\vsh)$) to ($(B2.south)+(.2,0)$);
    % \draw[dashed] ($(B2.south)+(.2,-.5*\vsh)$) to[] ($(B2.south)+(.2,-.3*\vsh)$);
    % \draw[dashed] ($(T2.south)+(-.2,-.5*\vsh)$) to[] ($(T2.south)+(-.2,-.3*\vsh)$);
    % \draw[->] ($(R2.south)-(0,.3*\vsh)$) to (R2);
    % \draw[dashed] ($(R2.south)-(0,.5*\vsh)$) to[] ($(R2.south)-(0,.3*\vsh)$);
    %
    %%%%%%%%%%%%%%%%%%%%%%%%%%%%%%
    % RIGHT BAR
    \node (zero) at (2.6*\hsh,1.5*\vsh) {$0$};
    \node (C) at (2.6*\hsh,\vsh) {${}_\omega D_\bullet$};
    \node (P0) at (2.6*\hsh,0) {${}_\omega C_\bullet$};
    \node (P1) at (2.6*\hsh,-\vsh) {$P_\bullet$};
    \draw[->] (C) to (zero);
    \draw[->] (P0) to (C);
    \draw[->] (P1) to (P0);
    % \draw[->] ($(P1.south)-(0,.3*\vsh)$) to (P1);
    % \draw[dashed] ($(P1.south)-(0,.5*\vsh)$) to[] ($(P1.south)-(0,.3*\vsh)$);
    \draw[dotted] (R0) to (C);
    \draw[dotted] (R1) to (P0);
    \draw[dotted] ($(R2.east)+(.7,0)$) to (P1);
  \end{tikzpicture}
  \caption{Partial resolution of the marked crossing ${}_\omega D_\bullet$. Colour blue ($\textdot[1]$) corresponds to the label of the foremost strand on the circle and colour red ($\textdot[2]$) corresponds to the label of the foremost strand amongst to two strands involved in $D_\bullet$.}
  \label{fig:proof_dot_slide}
\end{figure}
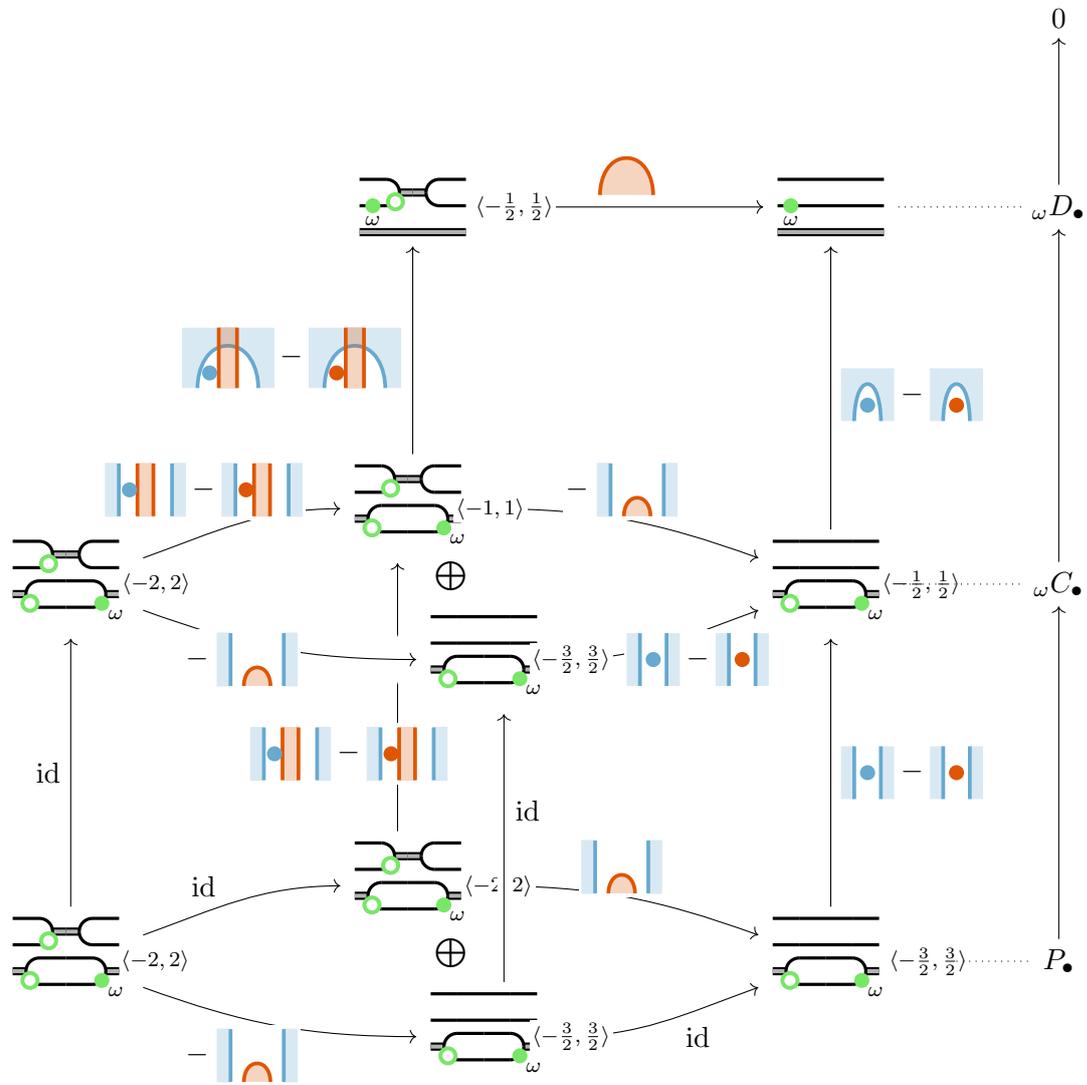

Finally, we construct a $\gloo$-equivariant isomorphism ${}_\omega C_\bullet\to C_\bullet^\omega$, as follows:
\begin{equation}
  \label{eq:proof_dot_slide_isomorphism}
  \begin{tikzpicture}
    \def\hsh{5}\def\vsh{5}\def\hoplussh{.5}\def\voplussh{1}
    \def\wxscl{.35}\def\wyscl{.35}
    %%%%%%%%%%%%%%%%%%%%%%%%%%%%%%
    % SOME DIAGRAMS
    \def\tempwebA{
      \webid[0][1.5][1]\websRoundMarkedB[1][1.5]\webm[2][1.5]\webid[3][1.5][1]
      \webmRoundMarkedB\webid[1][0][1]\webid[2][0][1]\websMarkedB[3][0][$\omega$]
    }
    \def\tempwebB{
      \webid[0][1.5][1]\webid[1][1.5]\webid[3][1.5][1]
      \webmRoundMarkedB\webid[1][0][1]\webid[2][0][1]\websMarkedB[3][0][$\omega$]
    }
    \newcommand{\tempfoamid}{
    \tikzpic{
        \flst\frst[1][0]\fdot[.5][.5]
      }[scale=.7,xscale=.5]
      -
      \tikzpic{
        \flst\frst[1][0]\fdot[.5][.5][2]
      }[scale=.7,xscale=.5]
    }
    \newcommand{\tempfoamdbid}{
      \tikzpic{
        \draw[fill,foamshade2] (1-.3,0) rectangle (1+.3,2);
        \draw[foamdraw2] (1-.3,0) to (1-.3,2);
        \draw[foamdraw2] (1+.3,0) to (1+.3,2);
        \draw[fill,foamshade1] (-.5,0) rectangle (0,2);
        \draw[foamdraw1] (0,0) to (0,2);
        \draw[fill,foamshade1] (2.5,0) rectangle (2,2);
        \draw[foamdraw1] (2,0) to (2,2);
        \fdot[.4][1][1]
      }[scale=.35]
      -
      \tikzpic{
        \draw[fill,foamshade2] (1-.3,0) rectangle (1+.3,2);
        \draw[foamdraw2] (1-.3,0) to (1-.3,2);
        \draw[foamdraw2] (1+.3,0) to (1+.3,2);
        \draw[fill,foamshade1] (-.5,0) rectangle (0,2);
        \draw[foamdraw1] (0,0) to (0,2);
        \draw[fill,foamshade1] (2.5,0) rectangle (2,2);
        \draw[foamdraw1] (2,0) to (2,2);
        \fdot[.4][1][2]
      }[scale=.35]
    }
    \newcommand{\tempfoamdbunzip}{
      \tikzpic{
        \funzip[.5][0][2]
        \draw[fill,foamshade1] (-.5,0) rectangle (0,2);
        \draw[foamdraw1] (0,0) to (0,2);
        \draw[fill,foamshade1] (2.5,0) rectangle (2,2);
        \draw[foamdraw1] (2,0) to (2,2);
      }[scale=.35]
    }
    %%%%%%%%%%%%%%%%%%%%%%%%%%%%%%
    % VERTICES
    %
    % top layer
    \node[] (L1) at (0,0) {$\tikzpic{\tempwebA}[xscale=\wxscl,yscale=\wyscl]$};
    \node[] (T1) at (\hsh-\hoplussh,\voplussh) {$\tikzpic{\tempwebA}[xscale=\wxscl,yscale=\wyscl]$};
    \node[anchor=mid] at (\hsh,.1) {$\bigoplus$};
    \node[] (B1) at (\hsh+\hoplussh,-\voplussh) {$\tikzpic{\tempwebB}[xscale=\wxscl,yscale=\wyscl]$};
    \node[] (R1) at (2*\hsh,0) {$\tikzpic{\tempwebB}[xscale=\wxscl,yscale=\wyscl]$};
    %
    % bottom layer
    \node[] (L2) at (0,-\vsh) {$\tikzpic{\tempwebA}[xscale=\wxscl,yscale=\wyscl]$};
    \node[] (T2) at (\hsh-\hoplussh,\voplussh-\vsh) {$\tikzpic{\tempwebA}[xscale=\wxscl,yscale=\wyscl]$};
    \node[anchor=mid] at (\hsh,.1-\vsh) {$\bigoplus$};
    \node[] (B2) at (\hsh+\hoplussh,-\voplussh-\vsh) {$\tikzpic{\tempwebB}[xscale=\wxscl,yscale=\wyscl]$};
    \node[] (R2) at (2*\hsh,-\vsh) {$\tikzpic{\tempwebB}[xscale=\wxscl,yscale=\wyscl]$};
    %
    %%%%%%%%%%%%%%%%%%%%%%%%%%%%%%
    % EDGES
    %
    % T2 to T1
    \draw[->] ($(T2.north)+(-.3,0)$) to ($(T1.south)+(-.3,0)$);
    \node[,fill=white,inner sep=1pt] at (.78*\hsh,-.4*\vsh) {$\id$};
    %
    % top layer
    \draw[->] (L1) to[out=20,in=180] (T1);
    \draw[->] (T1) to[out=0,in=160] (R1);
    \draw[->] (B1) to[out=0,in=-160] (R1);
    \draw[{preaction={draw,white,line width=5pt}},->] (L1) to[out=-20,in=180] (B1);
    \node[right={-12pt} of L1] {\scriptsize $\langle -2,2\rangle$};
    \node[right={-12pt} of T1,fill=white,inner sep=1pt] {\scriptsize $\langle -1,1\rangle$};
    \node[right={-12pt} of B1,fill=white,inner sep=1pt] {\scriptsize $\langle -\frac{3}{2},\frac{3}{2}\rangle$};
    \node[right={-12pt} of R1,fill=white,inner sep=1pt] {\scriptsize $\langle -\frac{1}{2},\frac{1}{2}\rangle$};
    \node[fill=white,inner sep=1pt] at (.4*\hsh-.5*\hoplussh,.25*\vsh) {$
      \tikzpic{
        \draw[fill,foamshade2] (1-.3,0) rectangle (1+.3,2);
        \draw[foamdraw2] (1-.3,0) to (1-.3,2);
        \draw[foamdraw2] (1+.3,0) to (1+.3,2);
        \draw[fill,foamshade1] (-.5,0) rectangle (0,2);
        \draw[foamdraw1] (0,0) to (0,2);
        \draw[fill,foamshade1] (2.5,0) rectangle (2,2);
        \draw[foamdraw1] (2,0) to (2,2);
        \fdot[1.6][1][1]
      }[scale=.35]
      -
      \tikzpic{
        \draw[fill,foamshade2] (1-.3,0) rectangle (1+.3,2);
        \draw[foamdraw2] (1-.3,0) to (1-.3,2);
        \draw[foamdraw2] (1+.3,0) to (1+.3,2);
        \draw[fill,foamshade1] (-.5,0) rectangle (0,2);
        \draw[foamdraw1] (0,0) to (0,2);
        \draw[fill,foamshade1] (2.5,0) rectangle (2,2);
        \draw[foamdraw1] (2,0) to (2,2);
        \fdot[1.6][1][3]
      }[scale=.35]$};
    \node[fill=white,inner sep=1pt] at (.4*\hsh+.5*\hoplussh,-.2*\vsh) {$-\;
      \tempfoamdbunzip$};
    \node[fill=white,inner sep=1pt] at (1.5*\hsh-.5*\hoplussh,.25*\vsh) {$-\;
      \tempfoamdbunzip$};
    \node[fill=white,inner sep=1pt] at (1.6*\hsh+.5*\hoplussh,-.2*\vsh) {$
      \tikzpic{
        \flst\frst[1][0]\fdot[.5][.5]
      }[scale=.7,xscale=.5]
      -
      \tikzpic{
        \flst\frst[1][0]\fdot[.5][.5][3]
      }[scale=.7,xscale=.5]
    $};
    %
    % bot-top
    \draw[->] (L2) to (L1);
    \draw[->] (R2) to (R1);
    \node[right] at (2*\hsh,-.5*\vsh) {$\id$};
    \node[left] at (0,-.5*\vsh) {$\id$};
    %
    % bottom layer
    \draw[->] (L2) to[out=20,in=180] (T2);
    \draw[->,{preaction={draw,white,line width=5pt}}] (L2) to[out=-20,in=180] (B2);
    \draw[->] (T2) to[out=0,in=160] (R2);
    \draw[->] (B2) to[out=0,in=-160] (R2);
    \node[right={-12pt} of L2] {\scriptsize $\langle -2,2\rangle$};
    \node[right={-12pt} of T2,fill=white,inner sep=1pt] {\scriptsize $\langle -1,1\rangle$};
    \node[right={-12pt} of B2,fill=white,inner sep=1pt] {\scriptsize $\langle -\frac{3}{2},\frac{3}{2}\rangle$};
    \node[right={-12pt} of R2,fill=white,inner sep=1pt] {\scriptsize $\langle -\frac{1}{2},\frac{1}{2}\rangle$};
    \node[fill=white,inner sep=1pt] at (.4*\hsh-.5*\hoplussh,.25*\vsh-\vsh) {$
      \tempfoamdbid$};
    \node[fill=white,inner sep=1pt] at (.4*\hsh+.5*\hoplussh,-.25*\vsh-\vsh) {$
      -\;\tempfoamdbunzip$};
    \node[fill=white,inner sep=1pt] at (1.5*\hsh-.5*\hoplussh,.25*\vsh-\vsh) {$
      -\;\tempfoamdbunzip$};
    \node[fill=white,inner sep=1pt] at (1.6*\hsh+.5*\hoplussh,-.2*\vsh-\vsh) {$\tempfoamid$};
    %
    % B2 to B1
    \draw[{preaction={draw,white,line width=5pt}},->] ($(B2.north)+(.3,0)$) to ($(B1.south)+(.3,0)$);
    \node[right=-8pt] at (\hsh+\voplussh,-.4*\vsh-\voplussh) {$\id$};
    \draw ($(B2.north)+(.3,1)$) to[out=90,in=-90] ($(T1.south)+(-.3,-1.4)$);
    \node[fill=white,inner sep=1pt] at (\hsh,-.5*\vsh) {$\tikzpic{
        \fzip[.5][1][2]
        \draw[fill,foamshade1] (-.5,0) rectangle (0,2);
        \draw[foamdraw1] (0,0) to (0,2);
        \draw[fill,foamshade1] (2.5,0) rectangle (2,2);
        \draw[foamdraw1] (2,0) to (2,2);
      }[scale=.3]$};
  \end{tikzpicture}
\end{equation}
Equivariance follows from \cref{lem:crossing_complex_equivariant}, and the fact that this is indeed an isomorphism of complexes follows from the following two computations (in the first case, we additionally use dot migration to change from $\textdot[2]$ to $\textdot[3]$):
\begin{IEEEeqnarray*}{rCll}
  \tikzpic{
    \draw[fill,foamshade2] (1-.3,0) rectangle (1+.3,2);
    \draw[foamdraw2] (1-.3,0) to (1-.3,2);
    \draw[foamdraw2] (1+.3,0) to (1+.3,2);
    \draw[fill,foamshade1] (-.5,0) rectangle (0,2);
    \draw[foamdraw1] (0,0) to (0,2);
    \draw[fill,foamshade1] (2.5,0) rectangle (2,2);
    \draw[foamdraw1] (2,0) to (2,2);
    \fdot[1.6][1][1]
  }[scale=.35]
  -
  \tikzpic{
    \draw[fill,foamshade2] (1-.3,0) rectangle (1+.3,2);
    \draw[foamdraw2] (1-.3,0) to (1-.3,2);
    \draw[foamdraw2] (1+.3,0) to (1+.3,2);
    \draw[fill,foamshade1] (-.5,0) rectangle (0,2);
    \draw[foamdraw1] (0,0) to (0,2);
    \draw[fill,foamshade1] (2.5,0) rectangle (2,2);
    \draw[foamdraw1] (2,0) to (2,2);
    \fdot[1.6][1][3]
  }[scale=.35]
  &=&
  \tikzpic{
    \draw[fill,foamshade2] (1-.3,0) rectangle (1+.3,2);
    \draw[foamdraw2] (1-.3,0) to (1-.3,2);
    \draw[foamdraw2] (1+.3,0) to (1+.3,2);
    \draw[fill,foamshade1] (-.5,0) rectangle (0,2);
    \draw[foamdraw1] (0,0) to (0,2);
    \draw[fill,foamshade1] (2.5,0) rectangle (2,2);
    \draw[foamdraw1] (2,0) to (2,2);
    \fdot[.4][1][1]
  }[scale=.35]
  -
  \tikzpic{
    \draw[fill,foamshade2] (1-.3,0) rectangle (1+.3,2);
    \draw[foamdraw2] (1-.3,0) to (1-.3,2);
    \draw[foamdraw2] (1+.3,0) to (1+.3,2);
    \draw[fill,foamshade1] (-.5,0) rectangle (0,2);
    \draw[foamdraw1] (0,0) to (0,2);
    \draw[fill,foamshade1] (2.5,0) rectangle (2,2);
    \draw[foamdraw1] (2,0) to (2,2);
    \fdot[.4][1][2]
  }[scale=.35]
  \mspace{10mu}-\mspace{10mu}
  \tikzpic{
    \funzip[.5][0][2]
    \fzip[.5][1][2]
    \draw[fill,foamshade1] (-.5,0) rectangle (0,2);
    \draw[foamdraw1] (0,0) to (0,2);
    \draw[fill,foamshade1] (2.5,0) rectangle (2,2);
    \draw[foamdraw1] (2,0) to (2,2);
  }[scale=.35]
  %%%%%%%%%%%%%%%%%%%%
  &
  %%%%%%%%%%%%%%%%%%%%
  \quad
  \text{thanks to}
  \quad
  \tikzpic{
    \funzip[.5][0][2]
    \fzip[.5][1][2]
  }[scale=.35]
  =
  \tikzpic{
    \draw[fill,foamshade2] (0,0) rectangle (1,2);
    \draw[foamdraw2] (0,0) to (0,2);
    \draw[foamdraw2] (1,0) to (1,2);
    \fdot[1.5][1][2]
  }[scale=.35]
  \;-\;
  \tikzpic{
    \draw[fill,foamshade2] (0,0) rectangle (1,2);
    \draw[foamdraw2] (0,0) to (0,2);
    \draw[foamdraw2] (1,0) to (1,2);
    \fdot[-.5][1][2]
  }[scale=.35]
  %%%%%%%%%%%%%%%%%%%%
  %%%%%%%%%%%%%%%%%%%%
  %%%%%%%%%%%%%%%%%%%%
  \\
  \tikzpic{
    \flst\frst[1][0]\fdot[.5][.5]
  }[scale=.7,xscale=.5]
  -
  \tikzpic{
    \flst\frst[1][0]\fdot[.5][.5][2]
  }[scale=.7,xscale=.5]
  &=&
  \tikzpic{
    \flst\frst[1][0]\fdot[.5][.5]
  }[scale=.7,xscale=.5]
  -
  \tikzpic{
    \flst\frst[1][0]\fdot[.5][.5][3]
  }[scale=.7,xscale=.5]
  \mspace{10mu}-\mspace{10mu}
  \tikzpic{
    \funzip[.5][1][2]
    \fzip[.5][0][2]
    \draw[fill,foamshade1] (-.5,0) rectangle (0,2);
    \draw[foamdraw1] (0,0) to (0,2);
    \draw[fill,foamshade1] (2.5,0) rectangle (2,2);
    \draw[foamdraw1] (2,0) to (2,2);
  }[scale=.35]
  %%%%%%%%%%%%%%%%%%%%
  &
  %%%%%%%%%%%%%%%%%%%%
  \quad
  \text{thanks to}
  \quad
  \tikzpic{
    \funzip[.5][2][2]
    \fzip[.5][1][2]
  }[scale=.35]
  =
  \tikzpic{
    \fdot[0][0][2]
  }[scale=.35]
  \;-\;
  \tikzpic{
    \fdot[0][0][3]
  }[scale=.35]
\end{IEEEeqnarray*}
This gives a zigzag of $\gloo$-equivariant homotopy equivalences between ${}_\omega D_\bullet$ and $D_\bullet^\omega$. One checks that their composition (or their inverse, using the splitting given in \cref{lem:web_resolution_twist}) is the identity.
The last statement in \cref{lem:dot_slide_lemma} is discussed in the following remark.\hfill\qed

\begin{remark}
  Let us try to prove the understrand variant of \cref{lem:dot_slide_lemma}, namely that:
  \[\tikzpic{
      \webncr
      \node[green_mark] (B) at (.3,1-.15) {};
      \node[above={-1pt} of B] {\scriptsize $\omega$};
    }[scale=.5][(0,.5*.5)]
  \;\simeq\;
  \tikzpic{
    \webncr
    \node[green_mark] (B) at (1-.3,.15) {};
    \node[below={-1pt} of B] {\scriptsize $\omega$};
  }[scale=.5][(0,.5*.5)]
  \;.\]
  The beginning of the proof would go through, and we could try to build an isomorphism as in \eqref{eq:proof_dot_slide_isomorphism}.
  The only difference would be that red dots ($\textdot[2]$) are swapped with brown dots ($\textdot[3]$).
  To get a chain map, we would need the following relations, with $\lambda$ some invertible scalar:
  \begin{IEEEeqnarray*}{rCl}
    \tikzpic{
      \draw[fill,foamshade2] (1-.3,0) rectangle (1+.3,2);
      \draw[foamdraw2] (1-.3,0) to (1-.3,2);
      \draw[foamdraw2] (1+.3,0) to (1+.3,2);
      \draw[fill,foamshade1] (-.5,0) rectangle (0,2);
      \draw[foamdraw1] (0,0) to (0,2);
      \draw[fill,foamshade1] (2.5,0) rectangle (2,2);
      \draw[foamdraw1] (2,0) to (2,2);
      \fdot[1.6][1][1]
    }[scale=.35]
    -
    \tikzpic{
      \draw[fill,foamshade2] (1-.3,0) rectangle (1+.3,2);
      \draw[foamdraw2] (1-.3,0) to (1-.3,2);
      \draw[foamdraw2] (1+.3,0) to (1+.3,2);
      \draw[fill,foamshade1] (-.5,0) rectangle (0,2);
      \draw[foamdraw1] (0,0) to (0,2);
      \draw[fill,foamshade1] (2.5,0) rectangle (2,2);
      \draw[foamdraw1] (2,0) to (2,2);
      \fdot[1.6][1][2]
    }[scale=.35]
    &\overset{?}{=}&
    \tikzpic{
      \draw[fill,foamshade2] (1-.3,0) rectangle (1+.3,2);
      \draw[foamdraw2] (1-.3,0) to (1-.3,2);
      \draw[foamdraw2] (1+.3,0) to (1+.3,2);
      \draw[fill,foamshade1] (-.5,0) rectangle (0,2);
      \draw[foamdraw1] (0,0) to (0,2);
      \draw[fill,foamshade1] (2.5,0) rectangle (2,2);
      \draw[foamdraw1] (2,0) to (2,2);
      \fdot[.4][1][1]
    }[scale=.35]
    -
    \tikzpic{
      \draw[fill,foamshade2] (1-.3,0) rectangle (1+.3,2);
      \draw[foamdraw2] (1-.3,0) to (1-.3,2);
      \draw[foamdraw2] (1+.3,0) to (1+.3,2);
      \draw[fill,foamshade1] (-.5,0) rectangle (0,2);
      \draw[foamdraw1] (0,0) to (0,2);
      \draw[fill,foamshade1] (2.5,0) rectangle (2,2);
      \draw[foamdraw1] (2,0) to (2,2);
      \fdot[.4][1][3]
    }[scale=.35]
    \mspace{10mu}-\mspace{10mu}\lambda\;
    \tikzpic{
      \funzip[.5][0][2]
      \fzip[.5][1][2]
      \draw[fill,foamshade1] (-.5,0) rectangle (0,2);
      \draw[foamdraw1] (0,0) to (0,2);
      \draw[fill,foamshade1] (2.5,0) rectangle (2,2);
      \draw[foamdraw1] (2,0) to (2,2);
    }[scale=.35]
    %%%%%%%%%%%%%%%%%%%%
    %%%%%%%%%%%%%%%%%%%%
    %%%%%%%%%%%%%%%%%%%%
    \\
    \tikzpic{
      \flst\frst[1][0]\fdot[.5][.5]
    }[scale=.7,xscale=.5]
    -
    \tikzpic{
      \flst\frst[1][0]\fdot[.5][.5][3]
    }[scale=.7,xscale=.5]
    &\overset{?}{=}&
    \tikzpic{
      \flst\frst[1][0]\fdot[.5][.5]
    }[scale=.7,xscale=.5]
    -
    \tikzpic{
      \flst\frst[1][0]\fdot[.5][.5][2]
    }[scale=.7,xscale=.5]
    \mspace{10mu}-\mspace{10mu}\lambda\;
    \tikzpic{
      \funzip[.5][1][2]
      \fzip[.5][0][2]
      \draw[fill,foamshade1] (-.5,0) rectangle (0,2);
      \draw[foamdraw1] (0,0) to (0,2);
      \draw[fill,foamshade1] (2.5,0) rectangle (2,2);
      \draw[foamdraw1] (2,0) to (2,2);
    }[scale=.35]
  \end{IEEEeqnarray*}
  The first identity imposes $\lambda=1$, while the second imposes $\lambda=-1$: we do not have a chain map.
  However, if we instead work with $\sfoam'$, then the bubble evaluation is replaced by the relation
  \begin{gather*}
    \tikzpic{
      \funzip[.5][2][2]
      \fzip[.5][1][2]
    }[scale=.35]
    =
    \;-\;
    \tikzpic{
      \fdot[0][0][2]
    }[scale=.35]
    \;+\;
    \tikzpic{
      \fdot[0][0][3]
    }[scale=.35]
    \;,
  \end{gather*}
  so that setting $\lambda=1$ works.
  In other words, if we work with $\sfoam$, \cref{lem:dot_slide_lemma} (overstrand) works but not its understrand variant; while if we work with $\sfoam'$, the understrand variant of \cref{lem:dot_slide_lemma} holds, but not \cref{lem:dot_slide_lemma} itself.
\end{remark}

\section{A global \texorpdfstring{$\gloo$}{gl11}-action on odd Khovanov homology}
\label{sec:global_action_odd_khovanov}

In this section, we describe a \texorpdfstring{$\gloo$}{gl11}-action on odd Khovanov homology using the original definition of odd Khovanov homology \cite{ORS_OddKhovanovHomology_2013}, and show that it coincides with the local $\gloo$-action defined in the previous section, restricted to links.
This can be seen as an equivariant version of \cite[Theorem 3.4]{SV_OddKhovanovHomology_2023}.

For this section, we refer to the original construction as \emph{odd $\slt$-Khovanov homology}, and denote $\slOKh^X(D)$ (resp.\ $\slOKh^Y(D)$) the construction using type X (resp.\ type Y).
In contrast, the construction in \cref{defn:tangle_invariant} is referred to as \emph{odd $\glt$-Khovanov homology}.
We review $\gloo$-representations in \cref{subsec:gloo_basic_definitions} and define a $\gloo$-representation on the exterior algebra in \cref{subsec:gloo_action_exterior_alg}.
Then, with the $\gloo$-action on $\slOKh^X(D)$ and $\slOKh^Y(D)$ defined in \cref{subsec:defn_global_action_odd_Kh}, we show in \cref{subsec:comparison_local_global} that:

\begin{theorem}
  \label{thm:equivalence_local_global}
  Let $W$ be a marked closed tangled web and $D=\undcomp(W)$ its underlying marked link diagram.
  Denote $\emptyset$ the empty web in $\sfoam^{\greenmarking}$.
  There is an $\gloo$-equivariant isomorphism of complexes
  \begin{equation*}
    H_\bullet\Hom_{\sfoam^{\greenmarking}}(\emptyset,\glOKh(W))\cong^{\gloo}\slOKh^Y(D)
  \end{equation*}
  and similarly when working with $\sfoam'$ and type X.
\end{theorem}

Here a link diagram $D$ is \emph{marked} if it is marked as a tangle diagram; see \cref{subsec:defn_tangle_invariant}.
Note that we used the homology functor described at the end of \cref{subsec:relative_homotopy_category}.

\medbreak

This $\gloo$-action has appeared in various guised in the literature; we discuss this in the following remarks.
The reader may wish to come back to them after reading the main definitions of the section.

\begin{remark}
  \label{rem:literature_comparison_Shumakovitch}
  Over the field $\bZ/2\bZ$,
  the action of $\liee$ recovers Shumakovitch's operation \cite{Shumakovitch_TorsionKhovanovHomology_2014} and when the marking is only a base point, the action of $\lief$ recovers Khovanov's differential \cite{Khovanov_PatternsKnotCohomology_2003}.
  We note that Shumakovitch's operation was recently extended to equivariant Khovanov homology over $\bZ$ \cite{KS_SymmetriesEquivariantKhovanov_2025}.
\end{remark}

\begin{remark}
  \label{rem:literature_comparison_Migdail_Wehrli}
  The $\lief$-part of the global action was already studied by Manion \cite{Manion_SignAssignmentTotally_2014}, although with a different perspective. This action is furthered studied in Migdail's PhD thesis \cite{Migdail_FunctorialityOddKhovanov_2025}, who realize it as an action of the coloring module or, when restricting to reduced odd Khovanov homology, as an action of the first homology of the branched double cover of the link. In particular, they point out that contrary to what is claimed in \cite{Manion_SignAssignmentTotally_2014}, markings (or dot action, in their perspective) cannot both overslide and underslide; via \cref{thm:equivalence_local_global}, this is in agreement with our result (\cref{lem:dot_slide_lemma}).
  This action is further studied in work in progress by Migdail and Wehrli \cite{MW_ModuleStructureOdd_}.

  As noted in the introduction, we learned about their work while working on this manu\-script; at the ``Conference on Modern Developments in Low-Dimensional Topology'' (Trieste, June 2025) and through private communication following that.
  This motivated us to precisely compare our action with the original definition of odd Khovanov homology, and hence compare with their work.
  Furthermore, as we learned in private communication, at least part of the work in progress of Migdail and Wehrli appeared already in Migdail's PhD thesis; while unpublished at that time, it was posted on the arXiv \cite{Migdail_FunctorialityOddKhovanov_2025} (see e.g.\ Theorem 5 for the definition of the action) at about the same time we posted our article.
\end{remark}

\begin{remark}
  \label{rem:literature_comparison_Grigsby_Wehrli}
  The $\gloo$-action on odd annular Khovanov homology defined by Grigsby and Wehrli \cite{GW_Action$mathfrakgl1|1$Odd_2020} closely resembles ours.
  Writing $x$ both for a circle and its associated variable, the action of $\lief$ in \cite{GW_Action$mathfrakgl1|1$Odd_2020} is an alternating sum of $(-)\wedge x$, and the action of $\liee$ is a sum over $(-)\mathbin{\llcorner}x$, where each sum is over essential circles.
  Apart from the difference of convention that their action is on the right, our definition does not allow twisting the action of $\liee$, so that the sum is always over all circles.
  If we did however, their action should be a special case of ours, with the annular structure inducing a canonical choice of markings, and hence a canonical $\gloo$-action.
\end{remark}

Throughout we assume $2$ is invertible in $\ringfoam$, although the analogue of \cref{rem:local_invariant_2_invertible} applies.

\subsection{Review of \texorpdfstring{$\gloo$}{gl11}-representations}

\label{subsec:gloo_basic_definitions}

% A \emph{super vector space} is a $\bZ/2\bZ$-graded vector space $V=V_{\ov{0}}\oplus V_{\ov{1}}$.
% For $v$ a non-zero homogeneous vector in $V$, we write $\abs{v}\in\bZ/2\bZ$ its $\bZ/2\bZ$-degree, and called it its \emph{parity}.
% Given two super vector spaces $V$ and $W$, the space $\HOM_\ringfoam(V,W)$ of all linear maps has a canonical $\bZ/2\bZ$-grading, where $\deg(f) = g$ if and only if $f(V_{\ov{i}})\subset V_{\ov{i}+g}$.

% A \emph{Lie superalgebra} is a super vector space $\fg$ endowed with a bilinear degree-preserving map $[-,-]\colon\fg\otimes \fg\to\fg$, satisfying the following axioms:
% \begin{IEEEeqnarray*}{Cl}
%   [v,w] = -(-1)^{\abs{v}\abs{w}}[w,v]&\text{graded symmetry}\\{}
%   [u,[v,w]] + (-1)^{\abs{u}(\abs{v}+\abs{w})}
%   [v,[w,u]] + (-1)^{\abs{w}(\abs{u}+\abs{v})}[w,[u,v]] 
%   = 0
%   \qquad&\text{graded Jacobi identity}
% \end{IEEEeqnarray*}
% Given a super vector space $V$, the super vector space $\END_\ringfoam(V)$ is canonically as Lie superalgebra, setting $[f,g] \coloneqq f\circ g - (-1)^{\abs{f}\abs{g}}g\circ f$.
% A \emph{representation} of a Lie superalgebra $\fg$ is the data of a Lie superalgebra morphism $\fg\to\END_\ringfoam(V)$ for some super vector space $V$.
% Given two representations $V$ and $W$ of $\fg$, their tensor product $V\otimes W$ is endowed with a representation of $\fg$ as follows:
% \begin{gather*}
%   x\cdot (v\otimes w) = (x\cdot v)\otimes w + (-1)^{\abs{x}\abs{v}} v\otimes (x\cdot w).
% \end{gather*}

Recall the notion of a super Lie algebra from \cref{ex:Lie_algebra_as_graded_algebra} and the example of $\gloo$ from \cref{ex:defn_gloo}.
Recall that a \emph{weight $\gloo$-representation} is a $\gloo$-representation whose underlying super vector space splits as a direct sum over the simultaneous eigenspaces of $h_1$ and $h_2$.
Elements of these eigenspaces are called \emph{weight vectors}, and their pair of $h_1$- and $h_2$-eigenvalue their \emph{weight}.

\medbreak

We partially follow \cite{GW_Action$mathfrakgl1|1$Odd_2020}.
Fix $\ring$ a generic commutative ring.
One-dimensional $\gloo$-representations are parame\-trized by $\nu\in\ring$, with $e$ and $f$ acting as zero and $h_1$ and $h_2$ acting as multiplication by $\nu$ and $-\nu$, respectively. We denote $L^{(1)}(\nu)$ this representation.
For $(r,s)\in\ring^2$, we define two-dimensional $\gloo$-representations $P^{(2)}(r,s)$ and $I^{(2)}(r,s)$ whose underlying super vector space has basis $\{v_1,v_x\}$ with $p(v_1)=\ov{0}$ and $p(v_x)=\ov{1}$, and actions:
\begin{gather*}
  P^{(2)}(r,s)\coloneqq
  \qquad
  \begin{gathered}
    \liee\cdot v_1 = 0,\quad \liee\cdot v_x = (r+s)v_1,\quad \lief\cdot v_1 = v_x,\quad\lief\cdot v_x = 0,\\
    \lieh_1\cdot v_1 = rv_1,\quad\lieh_1\cdot v_x = (r-1)v_x,\quad\lieh_2\cdot v_1 = s v_1,\quad\lieh_2\cdot v_x = (s+1)v_x
  \end{gathered}
\end{gather*}
and
\begin{gather*}
  I^{(2)}(r,s)\coloneqq
  \qquad
  \begin{gathered}
    \liee\cdot v_1 = 0,\quad \liee\cdot v_x = v_1,\quad \lief\cdot v_1 = (r+s)v_x,\quad\lief\cdot v_x = 0,\\
    \lieh_1\cdot v_1 = rv_1,\quad\lieh_1\cdot v_x = (r-1)v_x,\quad\lieh_2\cdot v_1 = s v_1,\quad\lieh_2\cdot v_x = (s+1)v_x,
  \end{gathered}
\end{gather*}
respectively.
These representations are irreducible if and only if $r+s\neq 0$, and in which case once has $L^{(2)}(r,s)\coloneqq P^{(2)}(r,s)\cong I^{(2)}(r,s)$.
We summarize these representations as follows:
\begin{IEEEeqnarray*}{CcCcCcC}
  \begin{tikzcd}[ampersand replacement=\&]
    v_x \arrow[loop above,"h_1"]\arrow[loop below,"h_2"]
    \arrow[r,bend left,"e"]
    \&
    v_1 \arrow[loop above,"h_1"]\arrow[loop below,"h_2"] \arrow[l,bend left,"f"]
  \end{tikzcd}
  \coloneqq
  &&
  \begin{tikzcd}[ampersand replacement=\&]
    \bullet \arrow[loop above,"\nu"]\arrow[loop below,"-\nu"]
  \end{tikzcd}
  &&
  \begin{tikzcd}[ampersand replacement=\&]
    v_x \arrow[loop above,"r-1"]\arrow[loop below,"s+1"]
    \arrow[r,bend left,"r+s"]
    \&
    v_1 \arrow[loop above,"r"]\arrow[loop below,"s"] \arrow[l,bend left,"1"]
  \end{tikzcd}
  &&
  \begin{tikzcd}[ampersand replacement=\&]
    v_x \arrow[loop above,"r-1"]\arrow[loop below,"s+1"]
    \arrow[r,bend left,"1"]
    \&
    v_1 \arrow[loop above,"r"]\arrow[loop below,"s"] \arrow[l,bend left,"r+s"]
  \end{tikzcd}
  \\
  &\mspace{50mu}&L^{(1)}(\nu)&\qquad&P^{(2)}(r,s)&\qquad&I^{(2)}(r,s)
\end{IEEEeqnarray*}

\subsection{\texorpdfstring{$\gloo$}{gl11}-action on the exterior algebra}
\label{subsec:gloo_action_exterior_alg}

Let $n\in\bN$.
Denote $\wedge_{\ringfoam}(x_1,\ldots,x_n)$ the exterior algebra on $n$ generators $x_1,\ldots,x_n$ over ${\ringfoam}$.
In other words, $\wedge_{\ringfoam}(x_1,\ldots,x_n)$ is the quotient of the free ${\ringfoam}$-algebra on generators $x_1,\ldots,x_n$ by the relations
\[
x^2_i = 0\quad 1\leq i \leq n
\qquad\an\qquad
x_ix_j= -x_jx_i\quad 1\leq i,j\leq n.
\]
A \emph{word} is a formal wedge product $x_{i_1}\wedge\ldots\wedge x_{i_k}$ where the indices $1\leq i_1,\ldots,i_k\leq n$ are pairwise distinct; words generate $\wedge_{\ringfoam}(x_1,\ldots,x_n)$ as a ${\ringfoam}$-module.
We equip $\wedge_{\ringfoam}(x_1,\ldots,x_n)$ with a $\bZ$-grading, setting $\abs{x_{i_1}\wedge\ldots\wedge x_{i_k}}=k$. It descends to a $\bZ/2\bZ$-grading viewing the $\bZ$-grading modulo two.
We write
\[\epsilon(\lambda_1x_1+\ldots+\lambda_nx_n) =\lambda_1+\ldots+\lambda_n.\]

Each choice of index $1\leq i\leq n$ defines a ${\ringfoam}$-linear map
\[x_i\inprod (-)\colon \wedge_{\ringfoam}(x_1,\ldots,x_n)\to \wedge_{\ringfoam}(x_1,\ldots,x_n),\]
called the \emph{inner product}, and defined on words as
\begin{gather*}
  x_i\inprod (x_{i_1}\wedge\ldots\wedge x_{i_k}) =
  \begin{cases}
    (-1)^{j-1}x_{i_1}\wedge\ldots \widehat{x}_{i_{j}} \ldots\wedge x_{i_k} & \text{if }i_j =i\text{ for some }1\leq j\leq k,\\
    0 & \text{else.}
  \end{cases}
\end{gather*}
Here the notation $\widehat{x}_{i_{j}}$ indicates that the letter $x_{i_j}$ is omitted.

\begin{lemma}
  \label{lem:properties_inner_product}
  Let $n\in\bN$. For each $v,w\in \wedge_{\ringfoam}(x_1,\ldots,x_n)$, we have:
  \begin{gather*}
    x_i\inprod (x_i \inprod v) = 0
    \quad\an\quad
    x_i\inprod (x_j\inprod v) = - x_j\inprod (x_i\inprod v),
    \\
    x_i\inprod(v\wedge w)=(x_i\inprod v)\wedge w+(-1)^{\abs{v}}v\wedge(x_i\inprod w),
    \\
    x_i\inprod(x_i\wedge v)+x_i\wedge(x_i\inprod v)=v
    \quad\an\quad
    x_i\inprod(x_j\wedge v)+x_j\wedge(x_i\inprod v)=0
    \quad\text{ if }i\neq j.
  \end{gather*}
\end{lemma}

\begin{definition}
  \label{defn:gloo_rep_exterior_alg}
  Given a linear combination of element $z=\lambda_1x_1+\ldots+\lambda_nx_n$ and a choice of scalar $\nu\in{\ringfoam}$, we endow $\wedge_{\ringfoam}(x_1,\ldots,x_n)$ with a structure of a weight $\gloo$-representation, denoted $V^{\nu;z}(x_1,\ldots,x_n)$, as follows:
  \begin{gather*}
    \lief(\ul{x}) = z\wedge \ul{x},
    \quad
    \liee(\ul{x}) = \sum_{j=1}^n (x_j\inprod \ul{x}), 
    \\
    \lieh_1(\ul{x}) = (\epsilon(z)-\abs{\ul{x}}-\nu)\ul{x}
    \quad\an\quad
    \lieh_2(\ul{x}) = (\abs{\ul{x}}+\nu) \ul{x}.
  \end{gather*}
\end{definition}

\begin{proof}
  We have $\liee(z)=\epsilon(z)$, and using that $x_i\inprod{}$ acts as a derivation (\cref{lem:properties_inner_product}), we have
  \[\sum_{j=1}^nx_j\inprod(z\wedge w) = \epsilon(z)v-\sum_{j=1}^nz\wedge(x_j\inprod w).\]
  It follows that $[\liee,\lief](w) = \epsilon(w)$.
\end{proof}

\begin{lemma}
  Write $z=\lambda_1x_1+\ldots+\lambda_nx_n$.
  In the terminology of \cref{subsec:gloo_basic_definitions}, we have
  \begin{gather*}
    V^{\nu;z}(x_1,\ldots,x_n)\cong I^{(2)}(\lambda_1,0)\otimes\ldots\otimes I^{(2)}(\lambda_n,0)\otimes L^{(1)}(-\nu).
  \end{gather*} 
  The weight of a word $\ul{x}$ is $(\epsilon(z)-\abs{\ul{x}}-\nu,\abs{\ul{x}}+\nu)$.\hfill\qed
\end{lemma}

The following extends \cite[Lemma~2.1]{Manion_SignAssignmentTotally_2014}:

\begin{lemma}
  \label{lem:merge_split_equivariant}
  Fix $n\in\bN$, scalar $\nu\in{\ringfoam}$ and elements $z\in\wedge (y_1,y_2,x_1,\ldots,x_n)$ and $z'\in \wedge(y,x_1,\ldots,x_n)$ of homogeneous degree $1$.
  \begin{enumerate}[(i)]
    \item Consider the $\bZ$-linear map
    \begin{align*}
      M_{y_1,y_2;y}\colon V^{\nu;z}(y_1,y_2,x_1,\ldots,x_n)&\to
      V^{\nu;z'}(y,x_1,\ldots,x_n)
      \\
      \ul{x}&\mapsto \ul{x}\vert_{y_1,y_2\mapsto y}
    \end{align*}
    Here $y_1,y_2\mapsto y$ means that we replace each instance of $y_1$ and $y_2$ by $y$ in $\ul{x}$.
    If $z\vert_{y_1,y_2\mapsto y}=z'$, then $M_{y_1,y_2;y}$ is a morphism of $\gloo$-representations.

    \item Consider the $\bZ$-linear map
    \begin{align*}
      S_{y;y_1,y_2}\colon V^{\nu+1;z'}(y,x_1,\ldots,x_n)&\to
      V^{\nu;z}(y_1,y_2,x_1,\ldots,x_n)
      \\
      \ul{x} &\mapsto (y_1- y_2)\wedge\ul{x}\vert_{y\mapsto y_1} =  (y_1- y_2)\wedge\ul{x}\vert_{y\mapsto y_2}
    \end{align*}
    If $z\vert_{y_1,y_2\mapsto y}=z'$, then $S_{y;y_1,y_2}$ commutes with the action of $\lieh_1$ and $\lieh_2$, and anti-commutes with the action of $\lief$ and $\liee$.
  \end{enumerate}
\end{lemma}

\begin{proof}
  Equivariance (up to sign) with respect to $\lieh_1$, $\lieh_2$ and $\lief$ is clear.
  Consider case (i).
  It is clear that $M_{y_1,y_2;y}$ is equivariant with respect to $x_i\inprod{}$.
  Equivariance (up to sign) with respect to $\liee$ then follows from the identity
  \begin{gather*}
    y\inprod (\ul{x}\vert_{y_1,y_2\mapsto y}) = (y_1\inprod\ul{x})\vert_{y_1,y_2\mapsto y} + (y_2\inprod\ul{x})\vert_{y_1,y_2\mapsto y}.
  \end{gather*}
  Similarly, case (ii) reduces to the identity
  \begin{align*}
    (y_1- y_2)\wedge(y\inprod \ul{x})\vert_{y\mapsto y_1} 
    &=
    (y_1- y_2)\wedge((y_1\inprod{}+y_2\inprod{}) \ul{x}\vert_{y\mapsto y_1})
    \\
    &= -(y_1\inprod{} + y_2\inprod{})\big((y_1- y_2)\wedge\ul{x}\vert_{y\mapsto y_1}\big),
  \end{align*}
  using distributivity and $(y_1\inprod{}+y_2\inprod{})(y_1-y_2)=0$.
  % note that these two identities do not hold when the inner product is replaced by the wedge product, taking eg x=1.
  % \begin{IEEEeqnarray*}{rCl}
  %   \IEEEeqnarraymulticol{3}{l}{
  %     \sum_{t=y_1,y_2,x_1,\ldots,x_n}t\inprod \big(\ul{x}\vert_{y\mapsto y_1}\wedge (y_1-y_2)\big)
  %   }
  %   \\
  %   \qquad&=&
  %   \sum_{t=y_1,y_2,x_1,\ldots,x_n}(t\inprod \ul{x}\vert_{y\mapsto y_1})\wedge (y_1-y_2)
  %   +
  %   \ul{x}\wedge\sum_{t=y_1,y_2,x_1,\ldots,x_n}t\inprod(y_1-y_2)
  %   \\
  %   &=&
  %   \sum_{t=y,x_1,\ldots,x_n}(t\inprod \ul{x})\vert_{y\mapsto y_1}
  %   \wedge (y_1-y_2)
  %   \IEEEQEDhere
  % \end{IEEEeqnarray*}
\end{proof}

\subsection{\texorpdfstring{$\gloo$}{gl11}-action on odd Khovanov homology}
\label{subsec:defn_global_action_odd_Kh}

We sketch how the dot action from \cite{Manion_SignAssignmentTotally_2014} extends to a $\gloo$-action on the original definition of odd Khovanov homology \cite{ORS_OddKhovanovHomology_2013}.
To get a proper invariant of marked oriented link, one should further shift and twist using the orientation, as in \cref{thm:topological_invariance}; we ignore that.

Let $D$ be a marked link diagram with $N$ crossings.
For a resolution $r\in\{0,1\}^N$ of $D$, denote $c(r)$ the number of connected components in $r$, and let
\[\nu(D;r)\coloneqq\frac{1}{2}(N-\abs{r}-c(r)).
% \frac{1}{2}(2N_- - N_+-\abs{r}-c(r)) when considering orientation
\]
We associate to $r$ the state space
\[
V(D;r)\coloneqq V^{\nu(D;r);z(D;r)}(x_1,\ldots,x_{c(r)}).
\]
Here $z(D;r)=\epsilon_1(f)x_i+\ldots+\epsilon_n(f)x_n$, where $x_i$ is the variable associated to the $i$-circle and $\epsilon_i(f)$ is the sum of all the $\lief$-scalars associated to marked points on the $i$-circle.
One then constructs a complex $\slOKh^Y(D)$
using the ${\ringfoam}$-linear maps $M_{y_1,y_2;y}$ and $S_{y;y_1,y_2}$, respectively corresponding to a ``merge cobordism'' and to a ``split cobordism''.
Finally, one fixes the signs, either using a type X or a type Y sign assignment; here we use type Y sign assignment.
By \cref{lem:merge_split_equivariant}, it carries an action of $\gloo$, up to some signs.
These signs can be fixed in an essentially unique way, following \cite[Proposition~2.2]{Manion_SignAssignmentTotally_2014}.
This defines a chain complex $\slOKh^Y(D)$ endowed with a $\gloo$-action.
Note that the quantum grading is precisely twice the eigenvalue of $\lieh_2$: $\lieh_2(v) = \frac{1}{2}\qdeg(v)\;v$.

% \begin{corollary}
%   \label{cor:global_invariance}
%   Let $L$ be an oriented link and $D$ an oriented link diagram of $L$.
%   Let $\cM$ be a choice of markings on $D$.
%   Then the $\gloo$-action on $\slOKh^Y(D,\cM)$ only depends on the position of markings, which furthermore satisfy the following:
%   \begin{gather*}
%     \tikzpic{
%       \coordinate (base) at (.5,.5);
%       \webpcr
%       \node[green_mark] (B) at (.3,.15) {};
%       \node[below={-1pt} of B] {\scriptsize $\lambda$};
%     }[baseline=(base),scale=.7]
%     \simeq
%     \tikzpic{
%       \coordinate (base) at (.5,.5);
%       \webpcr
%       \node[green_mark] (B) at (1-.3,1-.15) {};
%       \node[above={-1pt} of B] {\scriptsize $\lambda$};
%     }[baseline=(base),scale=.7]
%     \quad\an\quad
%     \tikzpic{
%       \coordinate (base) at (.5,.5);
%       \webncr
%       \node[green_mark] (B) at (.3,.15) {};
%       \node[below={-1pt} of B] {\scriptsize $\lambda$};
%     }[baseline=(base),scale=.7]
%     \simeq
%     \tikzpic{
%       \coordinate (base) at (.5,.5);
%       \webncr
%       \node[green_mark] (B) at (1-.3,1-.15) {};
%       \node[above={-1pt} of B] {\scriptsize $-\lambda$};
%       \node[green_mark] (C) at (1-.3,.15) {};
%       \node[below={-1pt} of C] {\scriptsize $2\lambda$};
%     }[baseline=(base),scale=.7]
%     \;.
%   \end{gather*}
%   If one considers type X instead, then the roles of the overcrossing and the undercrossing are swapped.
% \end{corollary}

\begin{remark}
  We work over a ring where $2$ is invertible.
  One could avoid this condition by either restricting to an action of $\sloo$, or by adding $\frac{1}{2}c(L)$ to $\nu(D;r)$, where $c(L)$ denotes the number of components of $L$.
\end{remark}

\begin{lemma}
  \label{lem:action_on_reduced}
  Let $D$ be a marked oriented link diagram.
  Reduced odd Khovanov homology can be identified with the kernel (or image) of $\liee$:
  \[\redOKh_\slt^Y(D)\cong \ker e=\im e.\]
  Write $\epsilon(f)$ the sum of scalars over all markings.
  Furthermore, if $\epsilon(f)=0$, then the $\gloo$-action on $\OKh(D)$ descends to a $\fgl_{1|1}^{\leq 0}$-action on $\redOKh(D)$.
\end{lemma}

Comparing with the work in progress of Migdail and Wehrli (see \cref{rem:literature_comparison_Migdail_Wehrli}), this is the same statement that the action of the coloring module descends to an action of the reduced coloring module on reduced odd Khovanov homology. 

\begin{proof}
  By definition, the following holds on $\slOKh^Y(D)$:
  \[\liee\circ\lief+\lief\circ\liee = \epsilon(f)\id.\]
  In particular, if $\epsilon(f)=1$, then $(\liee\circ\lief+\lief\circ\liee)(v) =v $ for all $v$; this shows that $\im e=\ker e$.
  This also shows that if $\epsilon(f)=0$, then $\lief(\ker e)\subset\ker e$, so that the $\fgl_{1|1}^{\leq 0}$-action restricts to $\ker e$.

  The reduced state space $\widetilde{\wedge}\subset\wedge(x_1,\ldots,x_n)$
  is defined in \cite[section~4]{ORS_OddKhovanovHomology_2013} as the subalgebra generated by $\ker\epsilon$.
  On homogeneous elements of degree one, we have $\liee=\epsilon$, so that $\widetilde{\wedge}\subset\ker\liee$.
  Moreover:
  \[\liee(x_{i_1}\ldots x_{i_k})=\sum_{j=1}^k(-1)^jx_{i_1}\ldots\widehat{x}_{i_j}\ldots x_{i_k} = (x_{i_2}-x_{i_1})\ldots(x_{i_k}-x_{i_1})\in\widetilde{\wedge},\]
  so $\im e\subset\widetilde{\wedge}$. The fact that $\im e=\ker e$ concludes.
  % \cite[Lemma~3.10]{Vaz_NotEvenKhovanov_2020}.
\end{proof}

% \begin{remark}
%   \ls{to clean:}
%   Consider the action of $\lief$ on reduced odd Khovanov homology as above, assuming $\epsilon(\cM)=0$.
%   To be non-trivial, we must have that for some homological degree $i$, there exists a quantum degree $j$ such that $\redOKh_{i,j}(L,{\ringfoam})\neq 0$ and $\redOKh_{i,j+2}(L,{\ringfoam})\neq 0$.
%   In particular,the link $L$ must be \emph{${\ringfoam} OH$-thick} (see \cite{Shumakovitch_PatternsOddKhovanov_2011}).
%   Quasi-alternating links, and in particular non-split alternating links like $T(2,n)$, are ${\ringfoam} OH$-thin for any ring ${\ringfoam}$.
%   In particular, the DG-structure is trivial.
%   The first $\bZ OH$-thick knot is $8_{19}$. The knot $15^n_{41127}$ is $\bQ OH$-thick; this seems to be the first such knot.
% \ls{to see if the action is non-trivial, we would need to compute (at least) for this 15 crossings knot...}
% \end{remark}

% Are there computations of either Szabo's homology, and the Heegaard-Floer homology of the branched double cover, for $8_{19}$ and $15^n_{41127}$?

\subsection{Comparison with the local action}
\label{subsec:comparison_local_global}

In this subsection, we prove \cref{thm:equivalence_local_global}.
We begin with the isomorphism at the level of state spaces.
Write $\ov{V}^{\nu,z}(x_1,\ldots,x_n)$ the $\gloo$-representation identical to $V^{\nu,z}(x_1,\ldots,x_n)$, except that the action of $\lief$ and $\liee$ is multiplied by $-1$.

\begin{lemma}
  \label{lem:equivalence_local_global_state_space}
  Let $W^{\greenmarking}$ be a marked closed web.
  Order the components of $\undcomp(W)$ from $1$ to $n$.
  Denote $\tau_i(\lief)$ the total $\lief$-marking on the $i$th component, $\#_i\text{split}$ the number of split webs in the $i$th component and $\#\text{split}$ the total number of splits.
  Let
  \[z=\epsilon_1x_1+\ldots+\epsilon_nx_n\]
  for $\epsilon_i = \tau_i(\lief)+\#_i\text{split}$, and
  \[\nu=\tau_{W^{\greenmarking}}(\lieh_2)+\frac{1}{2}\#\text{split}-\frac{n}{2}.\]
  Then:
  \begin{gather*}
    \Hom_{\sfoam^{\greenmarking}}(\emptyset,W^{\greenmarking}) 
    \cong \begin{cases}
      V^{\nu,z}(x_1,\ldots,x_n) & \text{if $n$ is even},\\
      \ov{V}^{\nu,z}(x_1,\ldots,x_n) & \text{if $n$ is odd},\\
    \end{cases}
  \end{gather*}
  as $\gloo$-representations, where the isomorphism is the one used in the proof of Theorem 3.4 in \cite[subsubsection~3.3.3]{SV_OddKhovanovHomology_2023}, which shows the isomorphism between odd $\slt$- and $\glt$-Khovanov homology.
\end{lemma}

\begin{proof}
  We first verify the lemma when $W^{\greenmarking}=\widetilde{W}^{\greenmarking}$ is of the form
  \begin{gather*}
    \widetilde{W}^{\greenmarking} = 
    \tikzpic{
      \webmMarkedB[0][0][$\omega_1$]\webs[1][0]
      \begin{scope}[shift={(2,0)}]
        \webmMarkedB[0][0][$\omega_2$]\webs[1][0]
      \end{scope}
      \node at (5,.5) {\ldots};
      \begin{scope}[shift={(6,0)}]
        \webmMarkedB[0][0][$\omega_n$]\webs[1][0]
      \end{scope}
    }[scale=.5][(0,.5*.5)]
    \;.
  \end{gather*}
  To describe the isomorphism mentioned in the statement, one arbitrarily chooses (i) a ``cup foam'' $\beta^{W}\colon\emptyset\to W$, whose underlying surface is a union of disks, and (ii) an ordering on the components of $\undcomp(W^{\greenmarking})$.
  For $\widetilde{W}$, we choose $\beta^{\widetilde{W}}$ as
  \begin{gather*}
    \beta^{\widetilde{W}} \;\coloneqq\;
    \tikzpic{
      \fill[foamshade1] (-.5,2) to 
        (0,2) to[out=-90,in=180] ++(.5,-.7) to[out=0,in=-90] ++(.5,.7)
        to ++(1,0) to[out=-90,in=180] ++(.5,-1) to[out=0,in=-90] ++(.5,1)
        to ++(2,0) to[out=-90,in=180] ++(.5,-1.3) to[out=0,in=-90] ++(.5,1.3)
        to ++(.5,0) to ++(0,-2) to (-.5,0) to (-.5,2);
      \draw[foamdraw1] (0,2) to[out=-90,in=180] ++(.5,-.7) to[out=0,in=-90] ++(.5,.7);
      \draw[foamdraw1] (2,2) to[out=-90,in=180] ++(.5,-1) to[out=0,in=-90] ++(.5,1);
      \draw[foamdraw1] (5,2) to[out=-90,in=180] ++(.5,-1.3) to[out=0,in=-90] ++(.5,1.3);
      \node at (4,1.5) {$\ldots$};
    }[scale=.6]\;,
  \end{gather*}
  and the ordering from left to right when reading $\widetilde{W}$.
  For $\epsilon\in\{0,1\}$, write $\tikzpic{\fhollowdot}_{\epsilon}$ for $\tikzpic{\fdot}$ if $\epsilon=1$, and nothing otherwise.
  Explicitly, the isomorphism is given on basis elements by (here $\delta\in\{0,1\}^n$ and $\abs{\delta}=\delta_1+\ldots+\delta_n$)
  \begin{gather*}
    \tikzpic{
      \fill[foamshade1] (-.5,2) to 
        (0,2) to[out=-90,in=180] ++(.5,-.7) to[out=0,in=-90] ++(.5,.7)
        to ++(1,0) to[out=-90,in=180] ++(.5,-1) to[out=0,in=-90] ++(.5,1)
        to ++(2,0) to[out=-90,in=180] ++(.5,-1.3) to[out=0,in=-90] ++(.5,1.3)
        to ++(.5,0) to ++(0,-2) to (-.5,0) to (-.5,2);
      \draw[foamdraw1] (0,2) to[out=-90,in=180] ++(.5,-.7) to[out=0,in=-90] ++(.5,.7);
      \draw[foamdraw1] (2,2) to[out=-90,in=180] ++(.5,-1) to[out=0,in=-90] ++(.5,1);
      \draw[foamdraw1] (5,2) to[out=-90,in=180] ++(.5,-1.3) to[out=0,in=-90] ++(.5,1.3);
      \node at (4,1.5) {$\ldots$};
      \fill[foamshade1] (-.5,2) rectangle (0,3);
      \draw[foamdraw1] (0,2) to (0,3);
      \draw[foamdraw1] (1,2) to (1,3);
      \fill[foamshade1] (1,2) rectangle (2,3);
      \draw[foamdraw1] (2,2) to (2,3);
      \draw[foamdraw1] (3,2) to (3,3);
      \fill[foamshade1] (3,2) rectangle (5,3);
      \draw[foamdraw1] (5,2) to (5,3);
      \draw[foamdraw1] (6,2) to (6,3);
      \fill[foamshade1] (6,2) rectangle (6.5,3);
      \fhollowdot[.5][2.2]\node[below right=-3pt] at (.5,2.2) {\footnotesize $\delta_1$};
      \fhollowdot[2.5][2.4]\node[below right=-3pt] at (2.5,2.5) {\footnotesize $\delta_2$};
      \fhollowdot[5.5][2.8]\node[below right=-3pt] at (5.5,2.8) {\footnotesize $\delta_n$};
    }[scale=.6]
    \;\mapsto\;
    (-1)^{\abs{\delta}n}
    x_n^{\delta_n}\ldots x_2^{\delta_2}x_1^{\delta_1}\;.
  \end{gather*}
  One checks that the $\gloo$-action on super $\glt$-foams coincides with the $\gloo$-action in $V^{\nu,z}$, up to an extra sign $(-1)^n$ for $\liee$ and $\lief$.
  For instance, assume $\omega_i=\roundmarking$ for each $i$, or in other words that $\tau_i(f)=-\#_i\text{split}=-1$ and $\tau(h_2)=-\frac{1}{2}\#\text{split}=-\frac{n}{2}$.
  Then the action of $\lief$ on foams is zero, in agreement with $z=0$; and the action of $\lieh_2(\beta^{\widetilde{W}^{\greenmarking}})=-\frac{n}{2}\beta^{\widetilde{W}^{\greenmarking}}$, in agreement with $\nu=-\frac{n}{2}$.

  We now show the lemma for generic $W^{\greenmarking}$.
  First, note that there is a $\gloo$-equivariant isomorphism $W^{\greenmarking}\cong^{\gloo}\widetilde{W}^{\greenmarking}$ in $\sfoam^{\greenmarking}$.
  Without the equivariance and ignoring markings, this statement was shown in \cite[subsubsection~6.3.3]{Schelstraete_OddKhovanovHomology_2024} (see also \cite[Lemma~2.13]{SV_OddKhovanovHomology_2023}).
  To lift it to include equivariance and markings, we use the $\gloo$-equivariant isomorphisms from \cref{lem:twist_dot_slide_cup_cap} and \cref{lem:web_defining_relations_equivariance} (note that $z$ and $\nu$ do not change under these isomorphisms), together with the following lemma:

  \begin{lemma}
    Let $\omega=(\alpha,\beta_1,\beta_2)$ be a generic local twist.
    In $\sfoam^{\greenmarking}$, there exists a $\gloo$-equivariant isomorphism
    \begin{gather*}
      \tikzpic{
        \webmMarkedB[0][1][$\omega$]\webs[1][1]\webm[2][1]\draw[web1] (3,2) to (4,2);
        \draw[web1] (0,0) to (3,0);\webs[3][0]
        % \draw[web2] (4,.5) to (6,.5);
      }[xscale=.4,yscale=.4][(0,1*.4)]
      \;\cong^{\gloo}
      \tikzpic{
        % \draw[web2] (0,1.5) to (2,1.5);
        \webm[2][1]\draw[web1] (3,2) to (6,2);
        \draw[web1] (2,0) to (3,0);\webs[3][0]\webmMarkedB[4][0][$\omega$]\webs[5][0]
      }[xscale=.4,yscale=.4][(0,1*.4)]
      \;.
    \end{gather*}
  \end{lemma}

  \begin{proof}
    The $\gloo$-equivariant isomorphism is given by the linear combination
    \begin{gather*}
      \tikzpic{
        \fill[foamshade2] (-.5,0) to (-.5,3) to (0,3) to[out=-90,in=90] (2,0) to (1,0) to[out=90,in=0] (.5,.7) to[out=180,in=90] (0,0);
        \draw[foamdraw2] (0,3) to[out=-90,in=90] (2,0);
        \draw[foamdraw2] (1,0) to[out=90,in=0] (.5,.7) to[out=180,in=90] (0,0);
        \fill[foamshade1] (3,0) to[out=90,in=-90] (1,3) to (2,3) to[out=-90,in=180] (2.5,2.3) to[out=0,in=-90] (3,3) to (3.5,3) to (3.5,0) to (3,0);
        \draw[foamdraw1] (3,0) to[out=90,in=-90] (1,3);
        \draw[foamdraw1] (2,3) to[out=-90,in=180] (2.5,2.3) to[out=0,in=-90] (3,3);
        \fdot[.5][.2][2]
      }[scale=.4]
      \;+\;
      \tikzpic{
        \fill[foamshade2] (-.5,0) to (-.5,3) to (0,3) to[out=-90,in=90] (2,0) to (1,0) to[out=90,in=0] (.5,.7) to[out=180,in=90] (0,0);
        \draw[foamdraw2] (0,3) to[out=-90,in=90] (2,0);
        \draw[foamdraw2] (1,0) to[out=90,in=0] (.5,.7) to[out=180,in=90] (0,0);
        \fill[foamshade1] (3,0) to[out=90,in=-90] (1,3) to (2,3) to[out=-90,in=180] (2.5,2.3) to[out=0,in=-90] (3,3) to (3.5,3) to (3.5,0) to (3,0);
        \draw[foamdraw1] (3,0) to[out=90,in=-90] (1,3);
        \draw[foamdraw1] (2,3) to[out=-90,in=180] (2.5,2.3) to[out=0,in=-90] (3,3);
        \fdot[2.5][2.8][1]
      }[scale=.4]
      \;,
      \quad\text{ whose inverse is }\quad
      \tikzpic{
        \fill[foamshade2] (-.5,0) to (-.5,3) to (0,3) to[out=-90,in=90] (2,0) to (1,0) to[out=90,in=0] (.5,.7) to[out=180,in=90] (0,0);
        \draw[foamdraw2] (0,3) to[out=-90,in=90] (2,0);
        \draw[foamdraw2] (1,0) to[out=90,in=0] (.5,.7) to[out=180,in=90] (0,0);
        \fill[foamshade1] (3,0) to[out=90,in=-90] (1,3) to (2,3) to[out=-90,in=180] (2.5,2.3) to[out=0,in=-90] (3,3) to (3.5,3) to (3.5,0) to (3,0);
        \draw[foamdraw1] (3,0) to[out=90,in=-90] (1,3);
        \draw[foamdraw1] (2,3) to[out=-90,in=180] (2.5,2.3) to[out=0,in=-90] (3,3);
        \fdot[.5][.2][2]
      }[scale=.4,yscale=-1]
      \;+\;
      \tikzpic{
        \fill[foamshade2] (-.5,0) to (-.5,3) to (0,3) to[out=-90,in=90] (2,0) to (1,0) to[out=90,in=0] (.5,.7) to[out=180,in=90] (0,0);
        \draw[foamdraw2] (0,3) to[out=-90,in=90] (2,0);
        \draw[foamdraw2] (1,0) to[out=90,in=0] (.5,.7) to[out=180,in=90] (0,0);
        \fill[foamshade1] (3,0) to[out=90,in=-90] (1,3) to (2,3) to[out=-90,in=180] (2.5,2.3) to[out=0,in=-90] (3,3) to (3.5,3) to (3.5,0) to (3,0);
        \draw[foamdraw1] (3,0) to[out=90,in=-90] (1,3);
        \draw[foamdraw1] (2,3) to[out=-90,in=180] (2.5,2.3) to[out=0,in=-90] (3,3);
        \fdot[2.5][2.8][1]
      }[scale=.4,yscale=-1]
      \;.
    \end{gather*}
    One checks that both 2-morphisms are $\gloo$-equivariant. For instance:
    \def\fscl{.3}
    \begin{gather*}
      \lief\cdot
      \left(
        \tikzpic{
          \fill[foamshade2] (-.5,0) to (-.5,3) to (0,3) to[out=-90,in=90] (2,0) to (1,0) to[out=90,in=0] (.5,.7) to[out=180,in=90] (0,0);
          \draw[foamdraw2] (0,3) to[out=-90,in=90] (2,0);
          \draw[foamdraw2] (1,0) to[out=90,in=0] (.5,.7) to[out=180,in=90] (0,0);
          \fill[foamshade1] (3,0) to[out=90,in=-90] (1,3) to (2,3) to[out=-90,in=180] (2.5,2.3) to[out=0,in=-90] (3,3) to (3.5,3) to (3.5,0) to (3,0);
          \draw[foamdraw1] (3,0) to[out=90,in=-90] (1,3);
          \draw[foamdraw1] (2,3) to[out=-90,in=180] (2.5,2.3) to[out=0,in=-90] (3,3);
          \fdot[.5][.2][2]
        }[scale=\fscl]
        \;+\;
        \tikzpic{
          \fill[foamshade2] (-.5,0) to (-.5,3) to (0,3) to[out=-90,in=90] (2,0) to (1,0) to[out=90,in=0] (.5,.7) to[out=180,in=90] (0,0);
          \draw[foamdraw2] (0,3) to[out=-90,in=90] (2,0);
          \draw[foamdraw2] (1,0) to[out=90,in=0] (.5,.7) to[out=180,in=90] (0,0);
          \fill[foamshade1] (3,0) to[out=90,in=-90] (1,3) to (2,3) to[out=-90,in=180] (2.5,2.3) to[out=0,in=-90] (3,3) to (3.5,3) to (3.5,0) to (3,0);
          \draw[foamdraw1] (3,0) to[out=90,in=-90] (1,3);
          \draw[foamdraw1] (2,3) to[out=-90,in=180] (2.5,2.3) to[out=0,in=-90] (3,3);
          \fdot[2.5][2.8][1]
        }[scale=\fscl]
      \right)
      \;=\;
      \\
      \alpha
      \left(
        \tikzpic{
          \fill[foamshade2] (-.5,0) to (-.5,3) to (0,3) to[out=-90,in=90] (2,0) to (1,0) to[out=90,in=0] (.5,.7) to[out=180,in=90] (0,0);
          \draw[foamdraw2] (0,3) to[out=-90,in=90] (2,0);
          \draw[foamdraw2] (1,0) to[out=90,in=0] (.5,.7) to[out=180,in=90] (0,0);
          \fill[foamshade1] (3,0) to[out=90,in=-90] (1,3) to (2,3) to[out=-90,in=180] (2.5,2.3) to[out=0,in=-90] (3,3) to (3.5,3) to (3.5,0) to (3,0);
          \draw[foamdraw1] (3,0) to[out=90,in=-90] (1,3);
          \draw[foamdraw1] (2,3) to[out=-90,in=180] (2.5,2.3) to[out=0,in=-90] (3,3);
          \fdot[.5][.2][2]
          \fdot[2.5][3.4][1]
        }[scale=\fscl][(0,1.5*\fscl)]
        \;+\;
        \tikzpic{
          \fill[foamshade2] (-.5,0) to (-.5,3) to (0,3) to[out=-90,in=90] (2,0) to (1,0) to[out=90,in=0] (.5,.7) to[out=180,in=90] (0,0);
          \draw[foamdraw2] (0,3) to[out=-90,in=90] (2,0);
          \draw[foamdraw2] (1,0) to[out=90,in=0] (.5,.7) to[out=180,in=90] (0,0);
          \fill[foamshade1] (3,0) to[out=90,in=-90] (1,3) to (2,3) to[out=-90,in=180] (2.5,2.3) to[out=0,in=-90] (3,3) to (3.5,3) to (3.5,0) to (3,0);
          \draw[foamdraw1] (3,0) to[out=90,in=-90] (1,3);
          \draw[foamdraw1] (2,3) to[out=-90,in=180] (2.5,2.3) to[out=0,in=-90] (3,3);
          \fdot[2.5][2.8][1]
          \fdot[2.5][3.4][1]
        }[scale=\fscl][(0,1.5*\fscl)]
      \right)
      +
      \left(
        {}-
        \tikzpic{
          \fill[foamshade2] (-.5,0) to (-.5,3) to (0,3) to[out=-90,in=90] (2,0) to (1,0) to[out=90,in=0] (.5,.7) to[out=180,in=90] (0,0);
          \draw[foamdraw2] (0,3) to[out=-90,in=90] (2,0);
          \draw[foamdraw2] (1,0) to[out=90,in=0] (.5,.7) to[out=180,in=90] (0,0);
          \fill[foamshade1] (3,0) to[out=90,in=-90] (1,3) to (2,3) to[out=-90,in=180] (2.5,2.3) to[out=0,in=-90] (3,3) to (3.5,3) to (3.5,0) to (3,0);
          \draw[foamdraw1] (3,0) to[out=90,in=-90] (1,3);
          \draw[foamdraw1] (2,3) to[out=-90,in=180] (2.5,2.3) to[out=0,in=-90] (3,3);
          \fdot[.5][.2][2]
          \fdot[2.5][2.8][1]
        }[scale=\fscl][(0,1.5*\fscl)]
        \;+\;
        \tikzpic{
          \fill[foamshade2] (-.5,0) to (-.5,3) to (0,3) to[out=-90,in=90] (2,0) to (1,0) to[out=90,in=0] (.5,.7) to[out=180,in=90] (0,0);
          \draw[foamdraw2] (0,3) to[out=-90,in=90] (2,0);
          \draw[foamdraw2] (1,0) to[out=90,in=0] (.5,.7) to[out=180,in=90] (0,0);
          \fill[foamshade1] (3,0) to[out=90,in=-90] (1,3) to (2,3) to[out=-90,in=180] (2.5,2.3) to[out=0,in=-90] (3,3) to (3.5,3) to (3.5,0) to (3,0);
          \draw[foamdraw1] (3,0) to[out=90,in=-90] (1,3);
          \draw[foamdraw1] (2,3) to[out=-90,in=180] (2.5,2.3) to[out=0,in=-90] (3,3);
          \fdot[2.5][2.8][1]
          \fdot[.5][.2][2]
        }[scale=\fscl][(0,1.5*\fscl)]
      \right)
      -
      \alpha
      \left(
        \tikzpic{
          \fill[foamshade2] (-.5,0) to (-.5,3) to (0,3) to[out=-90,in=90] (2,0) to (1,0) to[out=90,in=0] (.5,.7) to[out=180,in=90] (0,0);
          \draw[foamdraw2] (0,3) to[out=-90,in=90] (2,0);
          \draw[foamdraw2] (1,0) to[out=90,in=0] (.5,.7) to[out=180,in=90] (0,0);
          \fill[foamshade1] (3,0) to[out=90,in=-90] (1,3) to (2,3) to[out=-90,in=180] (2.5,2.3) to[out=0,in=-90] (3,3) to (3.5,3) to (3.5,0) to (3,0);
          \draw[foamdraw1] (3,0) to[out=90,in=-90] (1,3);
          \draw[foamdraw1] (2,3) to[out=-90,in=180] (2.5,2.3) to[out=0,in=-90] (3,3);
          \fdot[.5][.2][2]
          \fdot[.5][-.4][2]
        }[scale=\fscl][(0,1.5*\fscl)]
        \;+\;
        \tikzpic{
          \fill[foamshade2] (-.5,0) to (-.5,3) to (0,3) to[out=-90,in=90] (2,0) to (1,0) to[out=90,in=0] (.5,.7) to[out=180,in=90] (0,0);
          \draw[foamdraw2] (0,3) to[out=-90,in=90] (2,0);
          \draw[foamdraw2] (1,0) to[out=90,in=0] (.5,.7) to[out=180,in=90] (0,0);
          \fill[foamshade1] (3,0) to[out=90,in=-90] (1,3) to (2,3) to[out=-90,in=180] (2.5,2.3) to[out=0,in=-90] (3,3) to (3.5,3) to (3.5,0) to (3,0);
          \draw[foamdraw1] (3,0) to[out=90,in=-90] (1,3);
          \draw[foamdraw1] (2,3) to[out=-90,in=180] (2.5,2.3) to[out=0,in=-90] (3,3);
          \fdot[2.5][2.8][1]
          \fdot[.5][-.4][2]
        }[scale=\fscl][(0,1.5*\fscl)]
      \right)
      = \;0\;.
    \end{gather*}
    This concludes.
  \end{proof}

  Denote $\gamma\colon \widetilde{W}^{\greenmarking}\to W^{\greenmarking}$ this $\gloo$-equivariant isomorphism.
  By composition, it induces an isomorphism with either $V^{\nu,z}$ or $\ov{V}^{\nu,z}$, depending on $n$.
  Finally, this isomorphism has the expected form, as the composition $\gamma\circ\beta^{\widetilde{W}}$ gives a choice of ``cup foam'' for $W$.
\end{proof}

We can now prove \cref{thm:equivalence_local_global}:

\begin{proof}[Proof of \cref{thm:equivalence_local_global}]
  Let $W$ be a marked closed tangled web with $N$ crossings and $D=\undcomp(W)$ its underlying marked link diagram.
  In the construction of $\glKom(W)$, one associates to $r\in\{0,1\}^N$ a certain resolution of $W$, denoted $\langle W;r\rangle$, with extra marking $\roundmarking[2]$ for each split and extra shift $\langle-\frac{N-\abs{r}}{2},\frac{N-\abs{r}}{2}\rangle$.
  It follows from \cref{lem:equivalence_local_global_state_space} that $\Hom_{\sfoam^{\greenmarking}}(\emptyset,\langle W;r\rangle)$ is isomorphic as a $\gloo$-representation to $V(D;r)$, up to some additional sign on $\lief$ and $\liee$.

  Recall that in the definition of $\slOKh^Y(D)$, one must add signs to the action of $\lief$ and $\liee$ to get equivariance, and doing so is essentially unique.
  We choose these signs so that the isomorphism defined by \cref{lem:equivalence_local_global_state_space} becomes $\gloo$-equivariant.
  Finally, it was shown in \cite{SV_OddKhovanovHomology_2023} how one can add global signs to these isomorphisms an isomorphism of complexes; this does not affect $\gloo$-equivariance. Considering $\sfoam'$ instead, one gets an isomorphism with type X. This concludes.
\end{proof}

\newpage

\section{Torsion in pretzel links}
\label{sec:computation}

\definecolor{pretzelcol1}{rgb}{0.8,0,0.3}
\definecolor{pretzelcol2}{rgb}{0.3,0,0.8}

\NewDocumentCommand{\respretzel}{O{0}O{0}O{1}O{1}O{1}}{
  \ifnum#3=0
    \webid[#1-1][#2-2.5]
  \else
    \webs[#1-1][#2-2.5]\webm[#1][#2-2.5]
  \fi
  \ifnum#4=0
    \webid[#1-1][#2-.5]
  \else
    \webs[#1-1][#2-.5]\webm[#1][#2-.5]
  \fi
  \ifnum#5=0
    \webid[#1-1][#2+1.5]
  \else
    \webs[#1-1][#2+1.5]\webm[#1][#2+1.5]
  \fi
  \draw[web1] (#1-1,#2-1.5) to[out=180,in=180] (#1-1,#2-.5);
  \draw[web1] (#1-1,#2+.5) to[out=180,in=180] (#1-1,#2+1.5);
  \draw[web1] (#1-1,#2-2.5) to[out=180,in=180] (#1-1,#2+2.5);
  \draw[web1] (#1+1,#2-1.5) to[out=0,in=0] (#1+1,#2-.5);
  \draw[web1] (#1+1,#2+.5) to[out=0,in=0] (#1+1,#2+1.5);
  \draw[web1] (#1+1,#2-2.5) to[out=0,in=0] (#1+1,#2+2.5);
}

\tikzset{gaussmark/.style={circle,fill=white,draw=black,line width=.5pt,inner sep=1pt}}

In this section, we partially compute the odd Khovanov homology of the pretzel links $P(n,n,-n)$ and prove \cref{thm:main_pretzel}.

First, we take advantage of the extension to tangles, and compute the $n$th crossing twist:

\begin{lemma}
  \label{lem:n-crossing-twist}
  \def\webscl{.3}
  \def\intspc{5mu}
  Let $n\in\bN$.
  The following are isomorphisms in the relative homology category $\cK^{\gloo}(\sfoam^{\greenmarking})$ (the wiggly lines indicate homological degree zero):
  \begin{gather*}
    \scriptstyle
    \underbrace{
    \tikzpic{
      \webncr
      \node at (2,.5) {$\ldots$};
      \webncr[3][0]
    }[scale=\webscl][(0,.5*\webscl)]
    }_n
    \quad\simeq^{\gloo}\quad
    \tikzpic{
      \websRoundMarkedB[0][0][]
      \webm[1][0]
    }[scale=\webscl][(0,.5*\webscl)]
    \;\langle-\frac{n}{2},\frac{n}{2}\rangle
    \mspace{\intspc}
    \xrightarrow{
      \tikzpic{\frst\flst[1][0]\fdot[-.7][.5]}[scale=.4,xscale=.5]
      \;-\;
      \tikzpic{\frst\flst[1][0]\fdot[1.7][.5]}[scale=.4,xscale=.5]
    }
    \mspace{\intspc}
    \tikzpic{
      \websRoundMarkedB[0][0][]
      \webm[1][0]
    }[scale=\webscl][(0,.5*\webscl)]
    \;\langle-\frac{(n-1)}{2},\frac{(n-1)}{2}\rangle
    \mspace{\intspc}
    {\textstyle \ldots}
    \mspace{\intspc}
    \xrightarrow{
      \tikzpic{\frst\flst[1][0]\fdot[-.7][.5]}[scale=.4,xscale=.5]
      \;-\;
      \tikzpic{\frst\flst[1][0]\fdot[1.7][.5]}[scale=.4,xscale=.5]
    }
    \mspace{\intspc}
    \tikzpic{
      \websRoundMarkedB[0][0][]
      \webm[1][0]
    }[scale=\webscl][(0,.5*\webscl)]
    \;\langle-\frac{1}{2},\frac{1}{2}\rangle
    \mspace{\intspc}
    \xrightarrow{\tikzpic{\funzip}[scale=.4,xscale=.7]}
    \mspace{\intspc}
    \tikzpic{
      \webid
      \draw[decorate,decoration=snake] (0,-.5) to (2,-.5);
    }[scale=\webscl][(0,.5*\webscl)]
    \\[2ex]
    %%%%%%%%%%%%%%%%%%%%%%%%%%%%%%
    %%%%%%%%%%%%%%%%%%%%%%%%%%%%%%
    \scriptstyle
    \underbrace{
    \tikzpic{
      \webpcr
      \node at (2,.5) {$\ldots$};
      \webpcr[3][0]
    }[scale=\webscl][(0,.5*\webscl)]
    }_n
    \quad\simeq^{\gloo}\quad
    \tikzpic{
      \webid
    }[scale=\webscl][(0,.5*\webscl)]
    \;\langle-\frac{n}{2},\frac{n}{2}\rangle
    \mspace{\intspc}
    \xrightarrow{\tikzpic{\fzip}[scale=.4,xscale=.7]}
    \mspace{\intspc}
    \tikzpic{
      \websRoundMarkedB[0][0][]
      \webm[1][0]
    }[scale=\webscl][(0,.5*\webscl)]
    \;\langle-\frac{(n-1)}{2},\frac{(n-1)}{2}\rangle
    \mspace{\intspc}
    \xrightarrow{
      \tikzpic{\frst\flst[1][0]\fdot[-.7][.5]}[scale=.4,xscale=.5]
      \;-\;
      \tikzpic{\frst\flst[1][0]\fdot[1.7][.5]}[scale=.4,xscale=.5]
    }
    \mspace{\intspc}
    {\textstyle \ldots}
    \mspace{\intspc}
    \tikzpic{
      \websRoundMarkedB[0][0][]
      \webm[1][0]
    }[scale=\webscl][(0,.5*\webscl)]
    \;\langle-\frac{1}{2},\frac{1}{2}\rangle
    \mspace{\intspc}
    \xrightarrow{
      \tikzpic{\frst\flst[1][0]\fdot[-.7][.5]}[scale=.4,xscale=.5]
      \;-\;
      \tikzpic{\frst\flst[1][0]\fdot[1.7][.5]}[scale=.4,xscale=.5]
    }
    \mspace{\intspc}
    \tikzpic{
      \websRoundMarkedB[0][0][]
      \webm[1][0]
      \draw[decorate,decoration=snake] (0,-.5) to (2,-.5);
    }[scale=\webscl][(0,.5*\webscl)]
  \end{gather*}
\end{lemma}

\begin{proof}
  This can be shown by induction.
  Resolving the last crossing expresses the right-hand side as a cone over an equivariant chain morphism $f$.
  On the one hand, using the deformation retracts given in the proof of invariance under Reidemeister I (\cref{lem:invariance_Reidemeister_I}), one can simplify the source (resp.\ the target) of $f$;
  on the other hand, the target (resp.\ source) can be simplified using induction.
  This concludes.
\end{proof}

From now on, we ignore gradings.

The preztel link $P(n,n,-n)$ has three ``crossing bridges''; given \cref{lem:n-crossing-twist}, the associated complex can be identified with an $(n+1)\times (n+1)\times (n+1)$ hypercube.
Two slices of this hypercube are depicted in \cref{subfig:preztel_3_square_resolutions}, in the case $n=3$.
They correspond to fixing a resolution for the bottom crossing bridge.

To proceed, we work with state spaces; that is, we apply a representable functor as in \cref{thm:equivalence_local_global}. Furthermore, we restrict to reduced odd Khovanov homology; see \cref{lem:action_on_reduced}.
This means that if a state has $k$ circles, its associated state space is $2^{k-1}$-dimensional.
\Cref{subfig:pretzel_3_schematic} gives a schematic of the $4\times 4\times 4$ hypercube associated to $P(3,3,-3)$, for a suitable choice of basis elements in each state space. One recognizes the four slices of the hypercube, given by fixing a resolution for the bottom crossing bridge; the first slice corresponds to the first slice in \cref{subfig:preztel_3_square_resolutions}, while the other slices correspond to the second slice in \cref{subfig:preztel_3_square_resolutions}.
Each point is a basis element, or rather the copy of $\bZ$ it generates, and each (dashed) line is (minus) an identity, always reading from left to right.
For instance, here are the basis elements for the following state space (here $\omega=(1,\frac{1}{2},\frac{1}{2})$):
\begin{gather*}
  \Hom_{\sfoam^{\greenmarking}}\left(\emptyset,
   \tikzpic{
    \respretzel[0][0][1][1][1]
    \node[green_mark] at (-2,2) {};
    \node[left=-1pt] at (-2,2) {\scriptsize $\omega$};
    \node[green_mark] at (2,2) {};
    \node[right=-1pt] at (2,2) {\scriptsize $-\omega$};
  }[scale=.18]\right)\cong\langle 1,t\rangle_{\ringfoam}
  \quad\text{depicted as}\quad
  \tikzpic{
    \node[fill=black,circle,inner sep=1pt] at (0,0) {};
    \node[left] at (0,0) {$1$};
    \node[fill=black,circle,inner sep=1pt] at (0,1) {};
    \node[left] at (0,1) {$t$};
    \draw[thick,red,->] (0,0) to[out=0,in=-90] (.5,.5) to[out=90,in=0] (0,1);
  }\;.
\end{gather*}
More explicitly, the isomorphism is given as follows.
Let $W$ be the state and $\beta^W$ an undotted cup foam for $W$ (see \cref{subsec:comparison_local_global}).
Let $\id_W^L$ (resp.\ $\id_W^R$) be the identity of $W$ with an additional dot on the thickening of the left (resp.\ right) circle. Then the isomorphism identifies $\beta^W\leftrightarrow 1$ and $(\id_W^L-\id_W^R)\circ\beta^W\leftrightarrow t$.
The red arrow denotes the action of $\lief$; one of them is depicted in \cref{subfig:pretzel_3_schematic}.
Note that up to a homological shift, the $\lief$-action coincides with the dot multiplication map $d$ appearing in \cref{lem:n-crossing-twist}; in the hypercube, dot maps appearing is the last three slices (see the second slice in \cref{subfig:preztel_3_square_resolutions}) and in between the last three slices, as some of the almost-horizontal gray lines in \cref{subfig:pretzel_3_schematic}.
One can choose basis elements for the remaining state space and the get the full schematic of \cref{subfig:pretzel_3_schematic}.

Two connected components are highlighted in \cref{subfig:pretzel_3_schematic}; we compute their contribution to homology using gaussian elimination.
As we shall see, the red connected component contributes with $\bZ\oplus\bZ/3\bZ$, while the blue component contributes with $\bZ\oplus\bZ$.
Moreover, we can identify these copies with specific copies in the chain complex, as pictured in \cref{subfig:pretzel_3_schematic}.
Importantly, one copy of $\bZ$ lies ``below'' the copy of $\bZ/3\bZ$, with the $\lief$-action pictured as a red arrow; the action survives in homology.

\begin{figure}
\tikzset{pretzelshading1/.style={draw=pretzelcol1!20,line width=3pt}}
\tikzset{pretzelshading2/.style={draw=pretzelcol2!20,line width=3pt}}
\begin{subfigure}{\textwidth}
  \centering
  \begin{tikzpicture}[scale=.15]
    \def\str{6}
    \foreach \x in {0,\str,2*\str}{
      \foreach \y in {0,\str,2*\str}{
        \respretzel[-\x+\y][2*\x+2*\y][0][1][1]
      }
    }
    \foreach \t in {0,\str,2*\str}{
        \respretzel[-3*\str+\t][6*\str+2*\t][0][0][1]
    }
    \foreach \t in {0,\str,2*\str}{
        \respretzel[3*\str-\t][6*\str+2*\t][0][1][0]
    }
    \respretzel[0][6*2*\str][0][0][0]
    \foreach \t in {0,\str,2*\str,3*\str}{
      \draw[,shorten >=20pt,shorten <=20pt,->] (-3*\str+\t,6*\str+2*\t) 
        to node[above right=-3pt]{$m$} (-2*\str+\t,2*2*\str+2*\t);
    }
    \foreach \t in {0,\str,2*\str,3*\str}{
      \draw[shorten >=20pt,shorten <=20pt,<-] (3*\str-\t,6*\str+2*\t) 
        to node[above left=-3pt]{$s$} (2*\str-\t,2*2*\str+2*\t);
    }
  \end{tikzpicture}
  \hspace*{1cm}
  \begin{tikzpicture}[scale=.15]
    \def\str{6}
    \foreach \x in {0,\str,2*\str}{
      \foreach \y in {0,\str,2*\str}{
        \respretzel[-\x+\y][2*\x+2*\y][1][1][1]
      }
    }
    \foreach \t in {0,\str,2*\str}{
        \respretzel[-3*\str+\t][6*\str+2*\t][1][0][1]
    }
    \foreach \t in {0,\str,2*\str}{
        \respretzel[3*\str-\t][6*\str+2*\t][1][1][0]
    }
    \respretzel[0][6*2*\str][1][0][0]
    \foreach \t in {0,\str,2*\str}{
      \foreach \s in {0,\str}{
        \draw[,shorten >=20pt,shorten <=20pt,->] 
        (-\s+\t-\str,2*\s+2*\t+2*\str) 
        to node[above right=-3pt]{$d$}
        (-\s+\t,2*\s+2*\t);
      }
    }
    \foreach \t in {0,\str}{
      \foreach \s in {0,\str,2*\str}{
        \draw[,shorten >=20pt,shorten <=20pt,<-] 
        (-\s+\t+\str,2*\s+2*\t+2*\str) 
        to node[above left=-3pt]{$d$}
        (-\s+\t,2*\s+2*\t);
      }
    }
    \foreach \t in {0,\str,2*\str}{
      \draw[,shorten >=20pt,shorten <=20pt,->] 
        (-3*\str+\t,6*\str+2*\t) 
        to node[above right=-3pt]{$s$} 
        (-2*\str+\t,2*2*\str+2*\t);
    }
    \draw[,shorten >=20pt,shorten <=20pt,->] 
      (-\str,10*\str) 
      to node[above left=-3pt]{$s$} 
      (0,12*\str);
    \foreach \t in {0,\str,2*\str}{
      \draw[shorten >=20pt,shorten <=20pt,<-]
        (3*\str-\t,6*\str+2*\t) 
        to node[above left=-3pt]{$m$} 
        (2*\str-\t,2*2*\str+2*\t);
    }
    \draw[,shorten >=20pt,shorten <=20pt,->] 
      (0,12*\str) 
      to node[above right=-3pt]{$m$} 
      (\str,10*\str);
  \end{tikzpicture}
  \caption{Two slices in the hypercube associated to $P(3,3,-3)$; they correspond to taking the 0- or 1- resolution for the bottom crossing bridge in $P(3,3,-3)$. The labels $m$, $s$ and $d$ refer to a merge, a split or a dot multiplication maps, respectively.}
  \label{subfig:preztel_3_square_resolutions}
\end{subfigure}

\vspace*{1cm}

\begin{subfigure}{\textwidth}
  \centering
  \begin{tikzpicture}[scale=.45]
    % SHADED PART 1
    \foreach \z/\x/\y in {0/5/4, 8/4/2,8/4/6, 16/3/0,16/3/4,16/3/8, 24/2/2,24/2/6,24/2/10}
      {\draw[pretzelshading1] (\x+\z,\y) to (\x+\z+1,\y+3);}
    \foreach \z/\x/\y in {8/5/5, 16/4/3,16/4/7,16/4/11, 24/3/1,24/3/5,24/3/9}
      {\draw[pretzelshading1] (\x+\z,\y) to (\x+\z-1,\y+1);}
    \foreach \z/\x/\y in {8/5/5,8/5/9, 16/4/3,16/4/7, 24/3/1,24/3/5,24/3/9}
      {\draw[pretzelshading1] (\x+\z,\y) to (\x+\z-8,\y-1);}
    % SHADED PART 2
    \foreach \z/\x/\y in {8/5/4, 16/4/2,16/4/6, 24/3/0,24/3/4,24/3/8}
      {\draw[pretzelcol2!20,line width=3pt] (\x+\z,\y) to (\x+\z+1,\y+3);}
    \foreach \z/\x/\y in {16/5/5, 24/4/3,24/4/7,24/4/11}
      {\draw[pretzelcol2!20,line width=3pt] (\x+\z,\y) to (\x+\z-1,\y+1);}
    \foreach \z/\x/\y in {8/6/7, 16/5/5, 24/4/3,24/4/7}
      {\draw[pretzelcol2!20,line width=3pt] (\x+\z,\y) to (\x+\z-8,\y-1);}
    %
    % LINES
    % first square
    \foreach \x/\y in {5/4,4/6,3/8}
      {\draw (\x,\y) to (\x+1,\y+3);}
    \foreach \x/\y in {1/4,2/6,3/8}
      {\draw[dashed] (\x,\y) to (\x-1,\y+2);}
    \draw (2,10) to (3,12);
    \draw (2,11) to (2.5,12.5);
    \draw (4,11) to (3,12);
    \draw[dashed,out=45,in=-45] (4,10) to (3.5,12.5);
    % other squares
    \foreach \z in {8,16,24} {
      \foreach \x/\y in {3/0,2/2,4/2,1/4,3/4,5/4,2/6,4/6,3/8,2/10}
      {\draw (\x+\z,\y) to (\x+\z+1,\y+3);}
      \foreach \x/\y in {3/1,2/3,4/3,1/5,3/5,5/5,2/7,4/7,3/9,4/11}
      {\draw (\x+\z,\y) to (\x+\z-1,\y+1);}
    }
    %
    % INTERSQUARES
    % first intersquare
    \foreach \x/\y in {0/6,1/8,2/10}
      {\draw[draw=black!20,out=-30,in=-180] (\x,\y) to (\x+8,\y);}
    \foreach \x/\y in {6/6,5/8,4/10}
      {\draw[draw=black!20] (\x,\y) to (\x+8,\y+1);}
    \draw[draw=black!20] (3.5,12.5) to ++(7.5,-.5);
    \draw[draw=black!20] (3,13) to ++(8,0);
    \draw[dashed,draw=black!20] (3,12) to ++(8,1);
    % other intersquares
    \foreach \z in {8,16,24} {
      \foreach \x/\y in {3/1,2/3,4/3,1/5,3/5,5/5,2/7,4/7,3/9}
      {\draw[draw=black!20] (\x+\z,\y) to (\x+\z-8,\y-1);}
    }
    %
    % POINTS
    % first square
    \foreach \x/\y in {
      3/0,2/2,4/2,1/4,3/4,5/4,2/6,4/6,3/8,
      0/6,0/7,1/8,1/9,2/10,2/11,
      6/6,6/7,5/8,5/9,4/10,4/11,
      3/12,2.5/12.5,3.5/12.5,3/13
    }
    {\node[fill=black,circle,inner sep=.5pt] at (\x,\y){};}
    % other squares
    \foreach \z in {8,16,24} {
      \foreach \x/\y in {3/1,2/3,4/3,1/5,3/5,5/5,2/7,4/7,3/9,4/11}
      {\node[fill=black,circle,inner sep=.5pt] at (\x+\z,\y){};
      \node[fill=black,circle,inner sep=.5pt] at (\x+\z-1,\y+1){};}
      \foreach \x/\y in {3/0,4/2,5/4,3/1,4/3,5/5,6/7,5/9,3/13}
      {\node[fill=black,circle,inner sep=.5pt] at (\x+\z,\y){};}
    }
    \node[below] at (6,6) {\tiny $\bZ$};
    \node[below] at (12,6) {\tiny $\bZ$};
    \node[below] at (24+3,8) {\tiny $\bZ$};
    \node[above] at (24+3,9) {\tiny $\bZ_3$};
    \draw[red,->] (24+3,8) to[out=0,in=-90] (24+4,8.5) to[out=90,in=0] (24+3,9);
  \end{tikzpicture}
  \caption{A schematic for the hypercube associated to $P(3,3,-3)$. (Dashed) lines are (resp.\ minus) identities between copies of $\bZ$; homological degree goes from left to right. Two connected components are highlighted; labels ``$\bZ$'' and ``$\bZ_3$'' indicate how they contribute to the homology. A red arrow indicates the $\lief$-action induced on homology.}
  \label{subfig:pretzel_3_schematic}
\end{subfigure}

\caption{The proof of \cref{thm:main_pretzel} in the case of $P(3,3,-3)$.}
\label{fig:pretzel_3}
\end{figure}
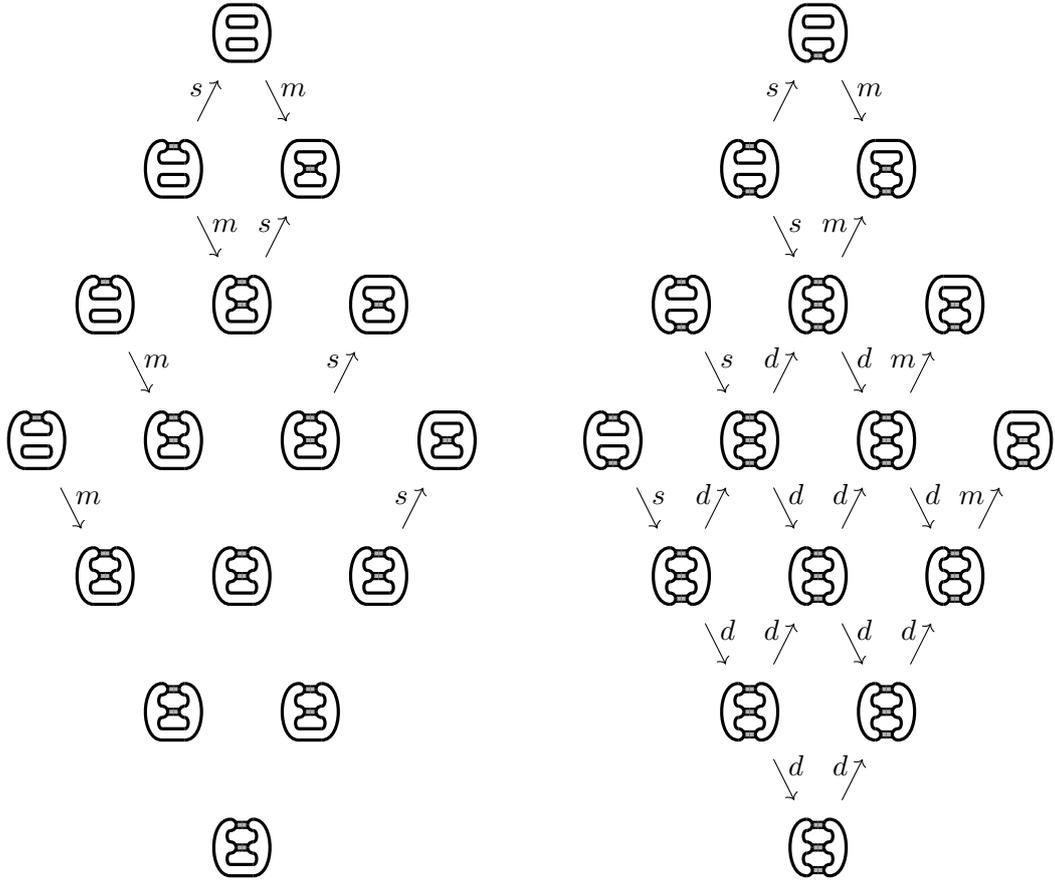
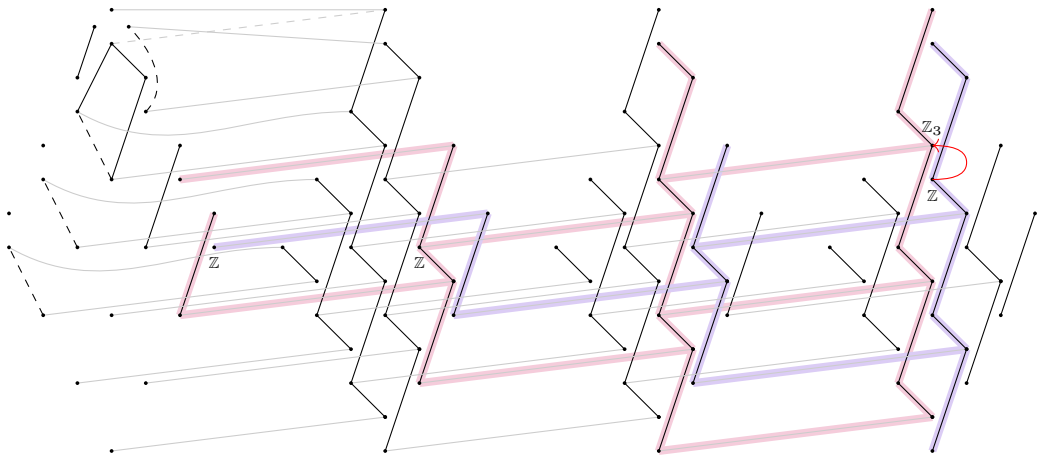

We aim to show these claims for generic $n\in\bN$.
It is not hard to extend the schematic; the relevant connected components are as follows:
\begin{center}
  \begin{tikzpicture}[scale=.3]
    \tikzset{pretzelshading1/.style={draw=pretzelcol1!20,line width=2pt}}
    \tikzset{pretzelshading2/.style={draw=pretzelcol2!20,line width=2pt}}
    % PART 1
    % big slanted line
    \foreach \x/\y in {
      7/2, 35/10,
      42/-8,42/-4,42/0,42/4,42/8
    }
      {\draw[pretzelshading1] (\x,\y) to (\x+1,\y+3);}
    % marked big slanted line
    \foreach \x/\y in {
      0/0, 7/-2, 14/-4,14/0,14/4,
      21/-6,21/-2,21/2,21/6, 28/-8,28/-4,28/0,28/4,28/8,
      35/-10,35/-6,35/-2,35/2,35/6,
      42/12
    }
      {\draw[pretzelshading1] (\x,\y) to node[gaussmark] {} (\x+1,\y+3);}
    %
    % small slanted line
    \foreach \x/\y in {
      0/0, 7/-2,7/2, 14/-4,14/0,14/4,
      21/-6,21/-2,21/2,21/6, 28/-8,28/-4,28/0,28/4,28/8,
      35/10
    }
      {\draw[pretzelshading1] (\x+8,\y+1) to (\x-1+8,\y+1+1);}
    % marked small slanted line
    \foreach \x/\y in {
      28/12,
      35/-10,35/-6,35/-2,35/2,35/6
    }
      {\draw[pretzelshading1] (\x+8,\y+1) to node[gaussmark] {} (\x-1+8,\y+1+1);}
    %
    % connecting lines
    \foreach \x/\y in {
      0/0, 7/-2,7/2, 14/-4,14/0,14/4,
      21/-6,21/-2,21/2,21/6,
      28/-8,28/-4,28/0,28/4,28/8,
      35/-10,35/-6,35/-2,35/2,35/6,35/10,
      0/4
    }
      {\draw[pretzelshading1] (\x,\y) to (\x+8,\y+1);}
    % marked connecting lines
    \foreach \x/\y in {
      0/4
    }
      {\draw[pretzelshading1] (\x,\y) to node[gaussmark] {} (\x+8,\y+1);}
    %
    % PART 2
    \begin{scope}[shift={(1,-2)}]
      % big slanted line
      \foreach \x/\y in {
        42/12
      }
        {\draw[pretzelshading2] (\x,\y) to (\x+1,\y+3);}
      % marked big slanted line
      \foreach \x/\y in {
        7/2, 14/0, 21/-2, 28/-4, 35/-6,
        14/4,
        21/2,21/6, 28/0,28/4,28/8,
        35/-2,35/2,35/6,35/10,
        42/-8,42/-4,42/0,42/4,42/8
      }
        {\draw[pretzelshading2] (\x,\y) to node[gaussmark] {} (\x+1,\y+3);}
      %
      % small slanted line
      \foreach \x/\y in {
        7/2, 14/0,14/4,
        21/-2,21/2,21/6, 28/-4,28/0,28/4,28/8,
        35/-6,35/-2,35/2,35/6,35/10
      }
        {\draw[pretzelshading2] (\x+8,\y+1) to (\x-1+8,\y+1+1);}
      % marked small slanted line
      \foreach \x/\y in {
        35/14
      }
        {\draw[pretzelshading2] (\x+8,\y+1) to node[gaussmark] {} (\x-1+8,\y+1+1);}
      %
      % connecting lines
      \foreach \x/\y in {
        0/4, 7/2, 14/0, 21/-2, 28/-4, 35/-6,
        14/4,
        21/2,21/6,
        28/0,28/4,28/8,
        35/-2,35/2,35/6,35/10,
        0/4
      }
        {\draw[pretzelshading2] (\x,\y) to (\x+8,\y+1);}
    \end{scope}
    \node[circle,fill=pretzelcol1!50,inner sep=2pt] (A1) at (7,2) {};
    \node[below left={-5pt}of A1] {\footnotesize $A$};
    \node[circle,fill=pretzelcol1!50,inner sep=2pt] (A2) at (35,10) {};
    \node[left={-3pt}of A2] {\footnotesize $B$};
    \node[circle,fill=pretzelcol1!50,inner sep=2pt] (A3) at (43,11) {};
    \node[right={-3pt}of A3] {\footnotesize $C$};
    \node[circle,fill=pretzelcol2!50,inner sep=2pt] (B) at (43,10) {};
    \node[right={-3pt}of B] {\footnotesize $D$};
    \node[circle,fill=pretzelcol2!50,inner sep=2pt] (E) at (1,2) {};
    \node[right={-3pt}of E] {\footnotesize $E$};
    \draw[pretzelcol1!50,line width=2pt] (A1) to[out=45,in=165] node[black,above=-1pt,pos=.2]{\footnotesize $n$} (A3);
    \draw[pretzelcol1!50,line width=2pt] (A2) to[out=45,in=165] node[black,above=-1pt,pos=.12]{\footnotesize $n$}(A3);
    \foreach \x/\y in {
      14/0,21/-2,28/-4,28/-8,35/-10,
      42/-8,42/-4,42/0,42/4,42/8
    }
      {\draw[thick,dashed] (\x,\y) to (\x+1,\y+3);}
    \foreach \x/\y in {
      28/-4,
      42/-8,42/-4,42/0,42/4,42/8
    }
      {\draw[thick,dashed] (\x,\y) to (\x+1,\y-1);}
    \foreach \x/\y in {
      7/2,14/0,21/-2,28/-8,35/-10
    }
      {\draw[thick,dashed] (\x,\y) to (\x+8,\y+1);}
  \end{tikzpicture}
\end{center}
We perform gaussian elimination on the arrows marked with $\tikzpic{\node[circle,fill=white,draw=black,line width=1pt,inner sep=3pt] at (0,0){};}$ .
The only surviving vertices are $A$, $B$, $C$, $D$ and $E$, as depicted. Moreover, this happens away from $C$ and $D$, and hence leaves the action of $\lief$ from $C$ to $D$ unaffected.
Gaussian elimination may induce maps between these vertices; to find these maps, one must compute the number of paths between two vertices, alternating between marked and unmarked edges.
For the blue connected component, no such path exists, and so $D$ and $E$ each contribute with a copy of $\bZ$ to homology.

There are paths from $B$ to $C$; they consist in going down a certain number of steps, then right, then up the same number of steps. This makes $n$ such paths.
Similarly, there are $n$ paths from $A$ to $C$, consisting of going down-right the second-to-last ``stair'' a certain number of steps, then going down, then going down-right the last ``stair'' until reaching the bottom-right, and finally going up to $C$; one of such paths is depicted in dashed lines.
Computing homology will (say) kill the vertex $B$ and make $C$ a copy of $\bZ/n\bZ$; the arrow from $A$ to $C$ is then zero, leaving $A$ to contribute with a copy of $\bZ$ to homology.

This concludes the proof of \cref{thm:main_pretzel}.

\newpage
% \nocite{*}
\printbibliography[heading=bibintoc]

\end{document}